\documentclass[reqno,11pt]{amsart}
\usepackage[latin1]{inputenc}
\usepackage[T1]{fontenc}
\usepackage{lmodern}
\usepackage{a4wide}
\usepackage{amssymb}
\usepackage{enumitem}

\newcommand\A{\mathbf{A}}
\newcommand\B{\mathcal{B}}
\newcommand\C{\mathcal{C}}
\newcommand\D{\mathcal{D}}
\newcommand\Db{\mathbf{D}}
\newcommand\E{\mathcal{E}}
\newcommand\F{\mathcal{F}}
\newcommand\Fb{\mathbf{F}}
\newcommand\G{\mathcal{G}}
\renewcommand\H{\mathcal{H}}
\newcommand\K{\mathcal{K}}
\newcommand\M{\mathrm{M}}
\newcommand\N{\mathbb{N}}
\newcommand\NN{\mathcal{N}}
\newcommand\R{\mathbb{R}}
\renewcommand\S{\mathbb{S}}
\newcommand\V{\mathbf{V}}
\newcommand\W{\mathbf{W}}
\newcommand\Y{\mathcal{Y}}
\newcommand\Z{\mathbb{Z}}

\renewcommand\d{\delta}
\newcommand\e{\varepsilon}
\newcommand\eps{\epsilon}
\newcommand\g{\gamma}
\newcommand\Gab{\boldsymbol{\Gamma}}
\newcommand\ph{\varphi}
\renewcommand\P{\mathbf{\Psi}}
\newcommand\XI{\boldsymbol{\xi}}
\newcommand\ZETA{\boldsymbol{\zeta}}

\newcommand\Cinfini{\C^{\infty}(\R,\R)}
\newcommand\Gt{\widetilde{G}}
\newcommand\St{\tilde{S}}
\newcommand\pa{\partial}
\renewcommand\leq{\leqslant}
\renewcommand\geq{\geqslant}

\newcommand\m{\mbox}
\newcommand\ds{\displaystyle}
\newcommand\be{\begin{equation}}
\newcommand\ee{\end{equation}}

\theoremstyle{plain}
\newtheorem{theorem}{Theorem}[section]
\newtheorem{lemma}[theorem]{Lemma}
\newtheorem{proposition}[theorem]{Proposition}
\newtheorem{corollary}[theorem]{Corollary}
\theoremstyle{definition}
\newtheorem{remark}[theorem]{Remark}

\numberwithin{equation}{section}

\title[Multi-bubbles for critical gKdV]{Construction of multi-bubble solutions \\ for the critical gKdV equation}

\author[V.~Combet]{Vianney Combet}
\address{Univ.~Lille, CNRS, UMR 8524 -- Laboratoire Paul Painlev\'e, F-59000 Lille, France}
\email{vianney.combet@math.univ-lille1.fr}

\author[Y.~Martel]{Yvan Martel}
\address{CMLS, \'Ecole Polytechnique, CNRS, Universit\'e Paris-Saclay, 91128 Palaiseau, France}
\email{yvan.martel@polytechnique.edu}

\begin{document}

\begin{abstract}
We prove the existence of solutions of the mass critical
generalized Korteweg--de~Vries equation $\pa_t u + \pa_x(\pa_{xx} u + u^5) = 0$
containing an arbitrary number $K\geq 2$ of blow up bubbles, for any choice of sign and scaling parameters:
for any $\ell_1>\ell_2>\cdots>\ell_K>0$ and $\eps_1,\ldots,\eps_K\in\{\pm1\}$,
there exists an $H^1$ solution $u$ of the equation such that
\[
u(t) - \sum_{k=1}^K \frac {\eps_k}{\lambda_k^\frac12(t)}
Q\left( \frac {\cdot - x_k(t)}{\lambda_k(t)} \right)
\longrightarrow 0 \quad\m{ in }\ H^1 \m{ as }\ t\downarrow 0,
\]
with $\lambda_k(t)\sim \ell_k t$ and $x_k(t)\sim -\ell_k^{-2}t^{-1}$ as $t\downarrow 0$.
The construction uses and extends techniques developed mainly in~\cite{MMjmpa} and~\cite{MMR1,MMR2,MMR3}.
Due to strong interactions between the bubbles, it also relies decisively on the sharp properties
of the minimal mass blow up solution (single bubble case) proved in~\cite{CM1}.
\end{abstract}

\maketitle

\section{Introduction}

\subsection{Main result}

We consider the mass critical generalized Korteweg--de Vries equation
\be \label{kdv}
\mathrm{(gKdV)} \quad \left\{
\begin{alignedat}{2}
&\pa_t u + \pa_x(\pa_{xx} u + u^5) = 0, \quad && (t,x) \in [0,T)\times\R, \\
&u(0,x) = u_0(x), && x\in\R.
\end{alignedat}
\right.
\ee
We first recall a few well-known facts about this equation.
The Cauchy problem is locally well-posed in the energy space $H^1$ from~\cite{Kato,KPV,KPV2}:
for a given $u_0 \in H^1$, there exists a unique (in a suitable functional space)
maximal solution $u$ of~\eqref{kdv} in $\C([0,T),H^1)$ and
\be \label{blowupcriterion}
T<+\infty \quad \m{implies} \quad \lim_{t\uparrow T} \|\pa_x u(t)\|_{L^2} = +\infty.
\ee
Moreover, such $H^1$ solutions satisfy the conservation of mass and energy
\[
\|u(t)\|_{L^2}^2 = \int u^2(t) = \|u_0\|_{L^2}^2, \quad
E(u(t)) = \frac12 \int (\pa_x u)^2(t) - \frac16 \int u^6(t) = E(u_0).
\]
Equation~\eqref{kdv} is invariant under scaling and translation:
if $u$ is a solution of~\eqref{kdv} then $u_{\lambda,x_0}$, defined by
\[
u_{\lambda,x_0}(t,x) = \lambda^{\frac12} u(\lambda^3 t,\lambda x + x_0), \quad \lambda>0, \ x_0\in \R,
\]
is also a solution of~\eqref{kdv}.

Recall that equation~\eqref{kdv} admits solutions of the form $u(t,x)=Q(x-t)$,
called \emph{solitons}, where $Q$ is the ground state
\[
Q(x) = \left( \frac 3{\cosh^2(2x)} \right)^{\frac14},\quad Q'' + Q^5 = Q,
\]
which attains the best constant in the following sharp Gagliardo--Nirenberg inequality, as proved in~\cite{W1983}:
\[
\forall v\in H^1,\quad \int v^6 \leq 3 \int (\pa_x v)^2 \left( \frac {\int v^2}{\int Q^2} \right)^2.
\]
It is well-known that the conservation of mass and energy, the above inequality,
and the blow up criterion~\eqref{blowupcriterion} ensure that $H^1$ initial data with \emph{subcritical mass},
\emph{i.e.}~$\|u_0\|_{L^2} < \|Q\|_{L^2}$, generate global in time and bounded solutions.

Concerning the mass threshold, it was proved in~\cite{MMR2} that there exists a unique (up to invariances)
blow up solution $S$ of~\eqref{kdv} with $\|S(t)\|_{L^2}=\|Q\|_{L^2}$, called the \emph{minimal mass blow up solution}.
Fixing by convention the blow up time as $t\downarrow 0$, $S(t)$ has the form of a single blow up bubble satisfying,
for some universal constant $c_0$, for all $t>0$ close to $0$,
\[
\left\| S(t) - \frac 1{t^{\frac12}} Q\left( \frac {\cdot + \frac 1t}t + c_0 \right) \right\|_{H^1} \lesssim t.
\]
In particular, it follows that $\|\pa_x S(t)\|_{L^2} \sim C t^{-1}$ as $t\downarrow 0$ for some $C>0$.
We refer to~\cite{CM1,MMR2} and to the Theorems~\ref{thm:previous}, \ref{thm:maintimeold}
and~\ref{thm:mainspaceold} below for detailed properties of $S(t)$ near the blow up time.
For blow up and classification results for initial data in a neighborhood of $Q$,
\emph{i.e.}~for $u_0$ satisfying $\|u_0-Q\|_{H^1}\leq \alpha$ with $\alpha>0$ small,
we refer to~\cite{MMjmpa,MMNR,MMR1,MMR2,MMR3,Mjams} and to the references therein.

\medskip

In this paper, we use and extend results and techniques developed in the above mentioned articles
to prove the following finite time blow up result with an arbitrary number of bubbles.

\begin{theorem} \label{thm:main}
Let $K\geq2$, $\ell_1>\ell_2>\cdots>\ell_K>0$ and $\eps_1,\ldots,\eps_K\in\{\pm1\}$.
Then there exist $T_0>0$ and a solution $u\in \C((0,T_0],H^1)$ to~\eqref{kdv} such that, for all $t\in (0,T_0]$,
\be \label{maintime}
\left\| u(t) - \sum_{k=1}^K \frac {\eps_k}{\lambda_k^\frac12(t)}
Q\left( \frac {\cdot - x_k(t)}{\lambda_k(t)} \right) \right\|_{H^1} \leq t^{\frac 1{22}},
\ee
where $\lambda_k(t)$ and $x_k(t)$ satisfy, for all $1\leq k\leq K$,
\[
\left| \lambda_k(t) - \ell_k t \right| \leq t^{\frac {23}{22}},\quad
\left| x_k(t) + \ell_k^{-2} t^{-1} \right| \leq t^{-\frac {21}{22}}.
\]
Moreover, the values of the mass and the energy of $u(t)$ are $\|u(t)\|_{L^2}^2 = K\|Q\|_{L^2}^2$ and
\[
E(u(t)) = \frac 1{16} \|Q\|_{L^1}^2 \sum_{k=1}^K \ell_k
\left( 1+2\sum_{j=1}^{k-1} \eps_k\eps_j \sqrt{\frac {\ell_k}{\ell_j}} \right) > 0.
\]
\end{theorem}

Note that the blow up points $x_k(t)$ go to infinity as $t\downarrow 0$, as in all previous blow up results
for the mass critical (gKdV) equation (see~\cite{MMR1,MMR2,MMR3} and references therein).
The assumption that the values of $\ell_k$ are all different implies some decoupling between the bubbles.
In particular, the solution $u(t)$ given by Theorem~\ref{thm:main} satisfies $\|u(t)\|_{L^2}^2 = K \|Q\|_{L^2}^2$,
which is expected to be the minimal amount of mass so that blow up occurs at $K$ different points.

However, because of the slow decay of $S(t,x)$ for $x \leq -\frac 1t - 1$ (see~\eqref{th:ptwise0} below),
the bubbles strongly interact close to the blow up time,
as shown by the value of the energy of $u(t)$ given in Theorem~\ref{thm:main}.
Indeed, if on the one hand the first part of the energy $\frac 1{16} \|Q\|_{L^1}^2 \sum_k \ell_k$
is due to the bubbles themselves (recall that $E(S(t)) = \frac 1{16} \|Q\|_{L^1}^2$ from~\cite{CM1}),
on the other hand the additional terms $\sum_{j<k} \eps_k\eps_j \sqrt{\frac {\ell_k}{\ell_j}}$ are due to interactions.

These interactions need to be carefully computed to construct the solution $u(t)$
and this is why the present paper relies decisively on the sharp properties
of the minimal mass blow up solution $S$ derived in~\cite{CM1},
as well as on some additional technical arguments from~\cite{MMR3}.
We recall relevant previous results in the next section
before commenting on the proof of Theorem~\ref{thm:main} in Section~\ref{sec:outline}.
We also briefly review some results on multi-bubble blow up
for other nonlinear models in Section~\ref{sec:othermodels}.

\subsection{Previous results on minimal mass blow up for critical (gKdV)}

We first recall the main result in~\cite{MMR2}.

\begin{theorem}[Existence and uniqueness of the minimal mass blow up solution~\cite{MMR2}] \label{thm:previous}
\leavevmode
\begin{enumerate}[label=\emph{(\roman*)}]
\item \label{existsS} \emph{Existence.}
There exist a solution $S\in \C((0,+\infty),H^1)$ to~\eqref{kdv}
and universal constants $c_0\in\R,C_0>0$ such that $\|S(t)\|_{L^2} = \|Q\|_{L^2}$ for all $t>0$ and
\begin{gather}
\|\pa_x S(t)\|_{L^2} \sim \frac {C_0}t \quad \m{ as }\ t\downarrow 0, \nonumber \\
S(t) - \frac 1{t^{\frac12}} Q\left( \frac {\cdot + \frac1t}t + c_0 \right) \longrightarrow 0
\quad \m{ in }\ L^2 \m{ as }\ t\downarrow 0. \label{d:S}
\end{gather}
\item \label{uniqueS} \emph{Uniqueness.}
Let $u_0\in H^1$ with $\|u_0\|_{L^2} = \|Q\|_{L^2}$
and assume that the corresponding solution $u(t)$ to~\eqref{kdv} blows up in finite or infinite time.
Then, $u\equiv S$ up to the symmetries of the flow.
\end{enumerate}
\end{theorem}

In~\cite{MMR2}, it is also proved that $S(t,x)$ has exponential decay in space for $x \geq -\frac 1t$,
but the behavior of $S(t,x)$ for $x \leq -\frac 1t$ is not studied,
and the convergence result~\eqref{d:S} is limited to the $L^2$ norm.
In view of their objective to eventually prove Theorem~\ref{thm:main}, the authors of the present paper
established in~\cite{CM1} the following sharp time and space asymptotics for $S(t)$ close to the blow up time.

\begin{theorem}[Time asymptotics~\cite{CM1}] \label{thm:maintimeold}
Let $S$ and $c_0$ be defined as in~\ref{existsS} of Theorem~\ref{thm:previous}.
Then there exist functions $\{Q_k\}_{k\geq 0} \subset \mathcal{S}(\R)$ (with $Q_0=Q$)
such that the following holds.

For all $m\geq 0$, there exists $T_0>0$ such that, for all $t\in (0,T_0]$,
\be \label{th:timem}
\left\| \pa_x^m S(t) - \sum_{k=0}^m \frac 1{t^{\frac 12 + m - 2k}}
Q_k^{(m-k)}\left( \frac {\cdot + \frac 1t}t + c_0 \right) \right\|_{L^2} \lesssim t^{1+m}.
\ee
\end{theorem}

\begin{theorem}[Space asymptotics~\cite{CM1}] \label{thm:mainspaceold}
Let $S$ be defined as in~\ref{existsS} of Theorem~\ref{thm:previous} and $m\geq 0$.
Then there exists $T_0>0$ such that the following hold.
\begin{enumerate}[label=\emph{(\roman*)}]
\item \label{item:ptleft} \emph{Pointwise asymptotics on the left.}
For all $t\in (0,T_0]$, for all $x \leq -\frac 1t - 1$,
\begin{gather}
\left| S(t,x) + \frac 12 \|Q\|_{L^1} |x|^{-\frac 32} \right| \lesssim |x|^{-\frac 32 - \frac 1{21}}, \label{th:ptwise0} \\
\left| \pa_x^m S(t,x) \right| \lesssim |x|^{-\frac 32 - m}. \label{th:ptwisem}
\end{gather}
\item \emph{Pointwise bounds on the right.}
There exists $\g_m>0$ such that, for all $t\in (0,T_0]$, for all $x\in\R$,
\be \label{th:ptwiser}
|\pa_x^m S(t,x)| \lesssim \frac 1{t^{\frac 12 + m}} \exp\left( -\g_m \frac {x + \frac 1t}t \right).
\ee
\end{enumerate}
\end{theorem}

As a corollary of Theorems~\ref{thm:maintimeold} and~\ref{thm:mainspaceold}, we also obtain time estimates
in exponential weighted spaces for $S(t)$, stated and proved in Appendix~\ref{appendix:corollary}.

\smallskip

Let us now comment on the main consequences of the above results on the construction of multi-bubbles.
First, Theorem~\ref{thm:maintimeold} means that $S(t)$ is smooth
and that its behavior as $t\downarrow 0$ is completely understood in any Sobolev norm,
which makes it a good candidate as the building brick to construct multi-bubble solutions.

Second, while the decay of $S(t,x)$ for $x \geq -\frac 1t$ is exponential,
which means weak interactions on the other bubbles on its right,
the asymptotic behavior of $S(t,x)$ for $x \leq -\frac 1t - 1$ is like~$|x|^{- \frac 32}$.
Surprisingly, this asymptotic behavior does not translate like the bubble,
but it rather describes a fixed explicit power-like tail,
which interacts with the other bubbles on the left of $S(t,x)$.
Such a power-like perturbation can be considered as a strong interaction
compared to the usual exponential interactions between solitons.

Finally, also observe that the convergence of $u(t)$ to the sum of bubbles in $H^1$
in~\eqref{maintime} is of order $t^{\frac 1{22}}$ as $t\downarrow 0$.
The reason why we cannot improve this convergence result
are the possible fluctuations of order $|x|^{-\frac 32 -\frac 1{21}}$
of the tail of $S(t,x)$ around the explicit tail
$-\frac 12 \|Q\|_{L^1} |x|^{-\frac 32}$ given by~\eqref{th:ptwise0}.
Of course, the exponent $\frac 1{21}$ is not optimal in~\cite{CM1},
but it seems difficult to obtain a significantly better estimate.
Another observation to complement Theorem~\ref{thm:main} is that $u(t)$ exhibits an explicit tail
made of the sum of the tails of each modulated version of~$S(t)$
(see Remark~\ref{rk:final} at the end of this paper).

\subsection{Outline of the proof of Theorem~\ref{thm:main}} \label{sec:outline}

Since we anticipate an explicit blow up rate, it is natural to use rescaled variables
--- see~\eqref{changevar} and~\eqref{rescaledkdv}.
Though not absolutely necessary, this change of variables allows us to reduce
to the case of bounded solitons in large time at the cost of an additional scaling term in the equation.

Next, we introduce an approximate solution to the multi-bubble problem whose main term
is a sum of $K$ rescaled and modulated versions of $S(t)$.
As discussed above, the interactions between the bubbles are strong
because of the tail on the left of the building brick $S(t,x)$.
However, continuing the formal discussion of Section~1.4 in~\cite{MMR3},
we notice that a tail of the form $c |x|^{-\frac 32}$,
with any $c \in\R$, is compatible with the blow up rate $t^{-1}$.
Thanks to the precise space asymptotics in Theorem~\ref{thm:mainspaceold}
and following the technique developed in~\cite{MMR3},
we improve the ansatz by adding terms taking into account the leading order of
the interactions of the bubbles with such fixed tails.
Constructed in this way, the approximate solution is sharp enough
to compute the energy of the multi-bubble solution given in Theorem~\ref{thm:main}.

The full ansatz $\V$ is introduced in Section~\ref{sec:V} and studied in Lemmas~\ref{lem:V} and~\ref{lem:Vmassenergy}.
Then, it only remains to construct a solution close to $\V$ using the by now standard (see references in the next section)
strategy of defining a sequence of backwards solutions satisfying uniform estimates close to $\V$.
To derive such uniform estimates, we study both the evolution of the modulation parameters
and the remaining infinite dimensional part (denoted $\e$ in the standard decomposition result Lemma~\ref{lem:syst_diff}).
To control $\e$, we use a modification of the mixed energy--virial functional introduced in~\cite{MMR1} for one bubble.
Due to the presence of $K$ bubbles, we need to run it recursively on each soliton.
At this point, the argument is reminiscent from the construction of bounded multi-solitons in~\cite{C,CMM,Ma1,MMT}.
Finally, some instability directions have to be controlled by
adjusting the initial parameters of the approximate solution,
\emph{e.g.}~as in~\cite{C,CMM} for the supercritical (gKdV) equation.

\subsection{Previous results for related models} \label{sec:othermodels}

We start by recalling early results on minimal mass blow up
for the mass critical nonlinear Schr\"odinger equation (in dimension $N\geq 1$),
\[
\mathrm{(NLS)} \quad i \pa_t u + \Delta u + |u|^{\frac 4N} u = 0, \quad (t,x)\in [0,T)\times \R^N.
\]
For (NLS), it is well-known (see \emph{e.g.}~\cite{Ca03}) that the pseudo conformal symmetry
generates an \emph{explicit} minimal mass blow up solution
\[
S_\mathrm{NLS}(t,x) = \frac 1{t^{\frac N2}} e^{- i\frac {|x|^2}{4t} - \frac it}
Q_\mathrm{NLS}\left( \frac xt \right),
\]
defined for all $t>0$ and blowing up as $t\downarrow 0$,
where $Q_\mathrm{NLS}$ is the unique radial ground state of~(NLS), solution to
\[
\Delta Q_\mathrm{NLS} - Q_\mathrm{NLS} + |Q_\mathrm{NLS}|^{\frac 4N} Q_\mathrm{NLS} = 0,
\quad Q_\mathrm{NLS}>0, \quad Q_\mathrm{NLS}\in H^1(\R^N).
\]
Using the pseudo conformal symmetry, it was proved by Merle~\cite{Mduke} that $S_\mathrm{NLS}$ is the unique
(up to the symmetries of the equation) minimal mass blow up solution of (NLS) in the energy space.
The situation is thus similar to the one for (gKdV),
though much more explicit and simple due to the pseudo conformal symmetry.

The first result on the construction of solutions of a dispersive PDE blowing up at $K$ given points in $\R^N$
is given in another pioneering paper of Merle~\cite{Mcmp} for the mass critical~(NLS).
In this paper, the solution is obtained as the limit of backwards solutions
containing $K$~focusing bubbles $S_\mathrm{NLS}$ and satisfying uniform estimates.
This strategy has been widely used and extended to other constructions of multi-bubble solutions,
both in regular or singular regimes --- see~\cite{C,CMM,Ma1,MMR1,RS} notably.
Since the blow up solution is obtained by gluing (rescaled and translated) solutions $S_\mathrm{NLS}(t)$,
it has the so-called \emph{conformal blow up rate} $t^{-1}$.
As a consequence of the exponential decay of $S_\mathrm{NLS}(t)$,
the interactions between the bubbles are exponentially small in $t^{-1}$.
This construction has been extended to (NLS) in bounded sets with Dirichlet boundary conditions in~\cite{G}.

Recall also that, for (NLS), the stable blow up is not the conformal one,
but the so-called \emph{log-log blow up}, whose rate is $\sqrt{\frac 1t \log|\log t|}$.
This blow up behavior has been studied thoughtfully by Merle and Rapha\"el~\cite{MeRa03,MeRa05,MeRa06}
(see also references therein) in a neighborhood of $Q_\mathrm{NLS}$.
Multiple point log-log blow up solutions and log-log blow up in a bounded set were studied in~\cite{F} and~\cite{PR}.

For the mass critical (NLS), we also mention the construction in~\cite{MR} of
a solution blowing up strictly faster than the conformal blow up rate
using strong interactions between several colliding blow up bubbles.

\smallskip

For the (gKdV) equation in the \emph{slightly supercritical case}, which reads
\[
\pa_t u + \pa_x(\pa_{xx} u + u^p) = 0, \quad (t,x)\in [0,T)\times\R,
\]
with $5<p<5+\alpha$ and $\alpha>0$ small,
Lan~\cite{L1,L2} constructed solutions blowing up at $K$ given points in the self-similar regime,
\emph{i.e.}~with blow up rate $t^{-\frac 13}$.
The idea to construct explicit self-similar blow up for slightly supercritical models
originates from~\cite{MRS} for the (NLS) model.
In contrast to~\cite{MRS}, where self-similar solutions are built using the soliton,
\cite{L1,L2} take advantage of the exact self-similar profiles constructed by Koch~\cite{Ko} for $p>5$ close to $5$.
As for the present paper, some technical tools in~\cite{L1,L2},
\emph{e.g.}~the mixed energy--virial functional, are taken from~\cite{MMjmpa,MMR1}.
However, the context and the difficulties of the construction are different in the self-similar blow up regime, since
the blow up points are finite and the interactions between the solitons are controlled as perturbation terms
(see the outline of the proof in~\cite{L2}).

For a numerical study of blow up for the critical and supercritical (gKdV) equations,
we refer to~\cite{KP} and to the references therein.

\smallskip

Note that, for the $L^2$ critical modified Benjamin--Ono equation,
a minimal mass blow up solution was recently constructed in~\cite{MP},
following~\cite{CM1} and~\cite{MMR2} as well as more specific previous works
on Benjamin--Ono type equations (see references in~\cite{MP}).

\smallskip

For the semilinear wave equation in the energy subcritical case, we refer to~\cite{CZ,MZ1,MZ2},
where multi-soliton profiles are relevant in the refined study of the behavior of solutions at a blow up point.
We also refer to the construction by Jendrej~\cite{J}
of radial two-bubbles for the energy critical wave equation in large dimensions.
In this work, the solution is global in one direction of time and one bubble stays bounded.
On the top of this standing bubble, a second bubble concentrates with a specific rate as time goes to infinity.

\smallskip

In the parabolic context, a similar result of blow up at $K$ points
for the energy subcritical nonlinear heat equation was proved in the early work~\cite{Mcpam},
using crucially continuity properties of blow up by perturbation of the data for this equation.
We also refer the reader to the recent result~\cite{CDM} where,
in the case of the energy critical heat equation in a bounded set with zero Dirichlet condition,
blow up in infinite time at $K$ given points in the domain is obtained,
provided a particular property related to the Green's function is satisfied.
This construction is another example of strong interactions,
both between the bubbles and the boundary condition and between the bubbles themselves.

\subsection{Notation}

For $f,g\in L^2(\R)$, their $L^2$ scalar product is denoted as
\[
\langle f,g \rangle = \int_\R f(x) g(x) \,dx.
\]

The Schwartz space $\mathcal{S}$ is classically defined as
\[
\mathcal{S} = \mathcal{S}(\R) = \{ f\in \Cinfini\ |\ \forall j\in\N, \forall \alpha\in\N, \exists C_{j,\alpha}\geq 0,
\forall y\in\R, |y^\alpha f^{(j)}(y)| \leq C_{j,\alpha} \}.
\]

For $f\in\mathcal{S}$ and $k\geq 0$, we use the notation $f^{(-k)}$ defined by induction as $f^{(0)} = f$ and
\be \label{antideriv}
f^{(-k-1)}(x) = -\int_x^{+\infty} f^{(-k)}(y) \,dy.
\ee

Let the generator of $L^2$ scaling be
\[
\Lambda f = \frac12 f + y \pa_y f.
\]

For brevity, $\sum_j$ and $\sum_k$ denote $\sum_{j=1}^K$ and $\sum_{k=1}^K$ respectively.
For a given $k\geq 1$, $\sum_{j<k}$ denotes $\sum_{j=1}^{k-1}$ when $k\geq 2$, and $0$ when $k=1$.

All numbers $C$ appearing in inequalities are real positive constants (with respect to the context),
which may change in each step of an inequality.

Finally, for $N\geq 1$, we denote by $\B_N$ (resp.~$\S_N$) the closed unit ball
(resp.~the unit sphere) of $\R^N$ for the euclidian norm.

\subsection*{Acknowledgements}

This work was partly supported by the Labex CEMPI (ANR-11-LABX-0007-01) and by the project ERC 291214 BLOWDISOL.

\section{Decomposition of the solution} \label{sec:decomposition}

\subsection{Properties of the linearized operator}

Let the functional space $\Y$ be defined by
\[
\Y = \{ f\in \Cinfini\ |\ \forall j\in\N, \exists C_j,r_j\geq 0,
\forall y\in\R, |f^{(j)}(y)| \leq C_j (1+|y|)^{r_j} e^{-|y|} \}
\]
and $L$ be the linearized operator close to $Q$ given by
\[
L f = -\pa_{yy} f + f - 5 Q^4 f.
\]

We recall without proof the following properties of $L$.
The properties (i)--(iii) are taken \emph{e.g.}~from~\cite{MMgafa,W1985},
while (iv) is proved in~\cite{MMR1} and (v) in~\cite{MMR3}.

\begin{lemma}[Properties of $L$] \label{lem:L}
The self-adjoint operator $L$ defined on $L^2$ satisfies:
\begin{enumerate}[label=\emph{(\roman*)}]
\item Kernel: $\ker L = \{aQ' \ ; \, a\in\R\}$.
\item Scaling: $L(\Lambda Q)=-2Q$.
\item Coercivity: there exists $\kappa_0>0$ such that, for all $f\in H^1$,
\be \label{coercivity}
\langle L f,f \rangle \geq \kappa_0 \|f\|_{H^1}^2 - \frac 1{\kappa_0}
\left( \langle f,Q \rangle^2 + \langle f,\Lambda Q \rangle^2 + \langle f,y\Lambda Q \rangle^2 \right).
\ee
\item There exists a unique function $P$ such that $P'\in\Y$ and
\begin{gather}
(LP)' = \Lambda Q, \quad \langle P,Q \rangle = \frac 1{16} \|Q\|_{L^1}^2 > 0, \quad \langle P,Q' \rangle = 0, \label{PQ} \\
\forall y>0, \quad |P(y)| + \left| P(-y) - \frac 12 \|Q\|_{L^1} \right| \lesssim e^{-y/2}. \label{est:P}
\end{gather}
\item There exists a unique even function $R\in\Y$ such that
\be \label{def:R}
L R = 5Q^4,\quad \langle Q,R \rangle = -\frac 34 \|Q\|_{L^1}.
\ee
\end{enumerate}
\end{lemma}

We recall that the function $P$ plays a crucial role in~\cite{CM1,MMR1,MMR2,MMR3} for understanding the blow up dynamics.
In particular, it appears naturally in the definition of $Q_1$, exhibited in Theorem~\ref{thm:maintimeold}.
Indeed, from~(3.47) in~\cite{CM1}, there exist $\lambda_0,c_1\in\R$ such that
\be \label{def:Q1}
Q_1 = -P' - \lambda_0(\Lambda Q)' + c_1Q''.
\ee

\subsection{Approximate multi-soliton of the rescaled flow} \label{sec:V}

Let $u(t,x)$ be any solution of~\eqref{kdv} defined for $t>0$ and $x\in\R$.
For $s<0$ and $y\in \R$, we consider the rescaled version of $u$ defined by
\be \label{changevar}
\tilde u(s,y) = \frac 1{(-2s)^{\frac 14}} u(t,x)\quad \m{with}\quad
t = \frac 1{\sqrt{-2s}}\quad \m{and}\quad x = \frac y{\sqrt{-2s}},
\ee
or, equivalently,
\[
u(t,x) = \frac 1{t^{\frac 12}} \tilde u(s,y)\quad \m{with}\quad
s = -\frac 1{2t^2}\quad \m{and}\quad y = \frac xt.
\]
In these new variables, $\tilde u$ is continuous with values in $H^1$ and is solution of the equation
\be \label{rescaledkdv}
\pa_s v + \frac 1{2s} \Lambda v + \pa_y (\pa_{yy} v + v^5) = 0.
\ee
Conversely, by a solution $v$ of~\eqref{rescaledkdv}, we mean $v = \tilde u$ where $u$ is
an $H^1$ solution of~\eqref{kdv} in the sense of~\cite{KPV,KPV2}.

Let $S$ be the minimal mass solution of~\eqref{kdv} defined in Theorem~\ref{thm:previous}.
Let $t_0>0$ be the minimum of the $T_0$ given by Theorems~\ref{thm:maintimeold} and~\ref{thm:mainspaceold}
applied with any $m$ such that $0\leq m\leq 20$, and let $s_0 = -\frac 1{2t_0^2}$.
We define $\St(s,y)$ from $S(t,x)$ as above, so that $\St$ satisfies~\eqref{rescaledkdv} on~$(-\infty,s_0]$.

As in Theorem~\ref{thm:main}, let $K\geq 2$, $\ell_1>\ell_2>\cdots>\ell_K>0$ and $\eps_1,\ldots,\eps_K\in\{\pm1\}$.
We look for an approximate solution $\V$ of~\eqref{rescaledkdv} under the form of
a sum of $K$ rescaled bubbles $\St$ plus correction terms and modulation.

Let $s_1 < 2\ell_1^3 s_0 < 0$ with $|s_1|\gg 1$ to be fixed later,
and let $I\subset (-\infty,s_1)$ be a compact interval.
For all $1\leq k\leq K$, we consider $\C^1$ functions $\mu_k$ and $y_k$
defined on $I$ to be determined later, and define $\tau_k$ by
\be \label{def:tauk}
\frac {d\tau_k}{d s} = \mu_k^{-3}.
\ee
In view of proving Proposition~\ref{prop:v},
which implies Theorem~\ref{thm:main} by rescaling (see Section~\ref{sec:conclusion}),
we assume that $\mu_k$, $\tau_k$ and $y_k$ satisfy, for some $0<\alpha\ll 1$, for all $s\in I$,
\be \label{apriori:mutauy}
\begin{gathered}
|\bar\mu_k(s)| \leq \alpha,\quad
|\bar\tau_k(s)| \leq \alpha,\quad
|\bar y_k(s)| \leq \alpha, \\
\m{with}\qquad
\bar\mu_k(s) = \frac {\mu_k(s)}{\ell_k} - 1,\quad
\bar\tau_k(s) = \frac {\tau_k(s)}{s\ell_k^{-3}} - 1,\quad
\bar y_k(s) = \frac {y_k(s)}{2s\ell_k^{-2}} - 1.
\end{gathered}
\ee

For all $1\leq k\leq K$, for all $s\in I$ and all $y\in\R$, we set
\[
W_k(s,y) = \eps_k \mu_k^{-\frac 12}(s) \St\left( \tau_k(s),\frac {y-z_k(s)}{\mu_k(s)} \right)
\quad\m{with}\quad z_k = y_k + \mu_k\left( -2\tau_k + c_0 - \frac {c_1}{2\tau_k} \right),
\]
and similarly, letting $\tilde\mu_k = \mu_k {\left(1+\frac {\lambda_0}{2\tau_k} \right)}^{-1}$,
\begin{gather*}
Q_k(s,y) = \eps_k \tilde\mu_k^{-\frac 12}(s) Q\left( \frac {y-y_k(s)}{\tilde\mu_k(s)} \right),\quad
R_k(s,y) = \eps_k \tilde\mu_k^{-\frac 12}(s) R\left( \frac {y-y_k(s)}{\tilde\mu_k(s)} \right), \\
\tilde P_k(s,y) = \eps_k \tilde\mu_k^{-\frac 12}(s) P\left( \frac {y-y_k(s)}{\tilde\mu_k(s)} \right),
\end{gather*}
where $R$ and $P$ are defined in Lemma~\ref{lem:L}, $\lambda_0$ and $c_1$ are defined in~\eqref{def:Q1},
and $c_0$ is defined in Theorem~\ref{thm:previous}.
As a consequence of~\eqref{apriori:mutauy}, we observe that
\be \label{apriori:z}
|\bar z_k(s)| \lesssim |\bar \mu_k(s)| + |\bar \tau_k(s)| + |\bar y_k(s)| + |s|^{-1} \lesssim \alpha,
\quad\m{with}\quad \bar z_k(s) = \frac {z_k(s)}s,
\ee
since, from the definition of $z_k$,
\begin{align*}
\bar z_k = \frac {z_k}s &= 2\ell_k^{-2}(1+\bar y_k) - 2\ell_k^{-2}(1+\bar \mu_k)(1+\bar \tau_k) + O(|s|^{-1}) \\
&= O(|\bar y_k| + |\bar \mu_k| + |\bar \tau_k| + |s|^{-1}).
\end{align*}
Similarly, directly from the definition of $\tilde\mu_k$ and~\eqref{apriori:mutauy}, we find
\be \label{apriori:mutilde}
|\tilde\mu_k(s) - \mu_k(s)| \lesssim |s|^{-1} \quad\m{and}\quad
\left| \frac {\dot{\tilde\mu}_k(s)}{\tilde\mu_k(s)} - \frac {\dot \mu_k(s)}{\mu_k(s)} \right| \lesssim |s|^{-2}.
\ee

As in~\cite{CM1,MMR1}, we proceed to a simple localization of $\tilde P_k$ to avoid some artificial growth at~$-\infty$.
Let
\[
\g = \min_{1\leq k\leq K-1} \left\{ \frac 1{4\ell_k} \left( \frac 1{\ell_{k+1}^2} - \frac 1{\ell_k^2} \right) \right\} > 0.
\]
Let $\chi\in\Cinfini$ be such that $0\leq \chi \leq 1$, $\chi'\geq 0$ on $\R$,
$\chi\equiv 1$ on $[-\g,+\infty)$ and $\chi\equiv 0$ on $(-\infty,-2\g]$.
We define
\[
P_k(s,y) = \tilde P_k(s,y) \chi\left( \frac {y-y_k(s)}{\tilde\mu_k(s)} |s|^{-1} \right).
\]
Setting for notational purposes $y_{K+1}(s) = y_K(s) - 6\g\ell_K |s|$,
we prove the following result.

\begin{lemma}[Estimates on $P_k$] \label{lem:Pk}
There exists $\alpha > 0$ small such that the following hold.

For all $m\geq 1$, for all $1\leq k\leq K$, for all $s\in I$ and for all $y\in \R$,
\be \label{b:Pky}
\left\{
\begin{aligned}
|P_k(y)| + |(y - y_k)\pa_y^m P_k(y)| &\lesssim e^{-\frac {|y - y_k|}{2\tilde\mu_k}}
+ \mathbf{1}_{\frac 12 (y_{k+1} + y_k) < y < y_k}(y), \\
|\pa_y^m P_k(y)| &\lesssim e^{-\frac {|y - y_k|}{2\tilde\mu_k}}
+ |s|^{-m} \mathbf{1}_{\frac 12 (y_{k+1} + y_k) < y < y_k - \frac {\ell_k\g}2 |s|}(y).
\end{aligned}
\right.
\ee
In particular, for all $m\geq 1$, for all $1\leq k\leq K$ and for all $s\in I$,
\be \label{eq:Pknorms}
\|(\cdot - y_k) P_k\|_{\dot H^m} + \|P_k\|_{L^2} \lesssim |s|^{\frac 12} \quad\m{and}\quad
\|\pa_y^m P_k\|_{L^2} \lesssim 1.
\ee
\end{lemma}

\begin{proof}
From~\eqref{apriori:mutauy} and the definitions of $\g$ and $\chi$,
we first claim that, for all $1\leq k\leq K$,
\be \label{ref:gamma}
\chi\left( \frac {y-y_k(s)}{\tilde\mu_k(s)} |s|^{-1} \right) =
\begin{cases}
0 &\m{ for all } y\leq \frac 12 \Bigl[y_{k+1}(s) + y_k(s)\Bigr], \\
1 &\m{ for all } y\geq y_k(s) - \frac {\ell_k\g}2 |s|.
\end{cases}
\ee
Indeed, if $\chi\left( \frac {y-y_k(s)}{\tilde\mu_k(s)} |s|^{-1} \right) < 1$
then $\frac {y-y_k(s)}{\tilde\mu_k(s)} |s|^{-1} < -\g$ and so,
since $\tilde\mu_k(s) \geq \frac {\ell_k}2$ from~\eqref{apriori:mutauy}
and~\eqref{apriori:mutilde} by taking $\alpha \leq \frac 12$,
\[
y < -\g |s|\tilde\mu_k(s) + y_k(s) \leq y_k(s) - \frac {\ell_k\g}2 |s|.
\]
Similarly, if $\chi\left( \frac {y-y_k(s)}{\tilde\mu_k(s)} |s|^{-1} \right) > 0$ then
$\frac {y-y_k(s)}{\tilde\mu_k(s)} |s|^{-1} > -2\g$, and so
\[
y > -2\g |s|\tilde\mu_k(s) + y_k(s) \geq -2(1+\alpha)|s| \left( \g \ell_k + \ell_k^{-2} \right).
\]
Thus, for $1\leq k\leq K-1$,
\[
y > -\frac 12 (1+\alpha)|s| \left( \ell_{k+1}^{-2} + 3\ell_k^{-2} \right)
\geq -\frac 12 (1+\alpha)|s| \left( 2\ell_{k+1}^{-2} + 2\ell_k^{-2} - 4\ell_K\g \right).
\]
But, from~\eqref{apriori:mutauy}, we have
\[
-4(1+\alpha) |s| \ell_K^{-2} \leq y_{k+1}(s) + y_k(s) \leq -2(1-\alpha) |s| (\ell_{k+1}^{-2} + \ell_k^{-2}),
\]
and so
\[
y > \frac 12 \Bigl[ y_{k+1}(s) + y_k(s) \Bigr] \left( \frac {1+\alpha}{1-\alpha} - \ell_K^3\g \right)
\geq \frac 12 \Bigl[ y_{k+1}(s) + y_k(s) \Bigr],
\]
by taking $\alpha>0$ small enough, namely $\alpha \leq \frac {\rho}{\rho+2}$ with $\rho = \ell_K^3\g > 0$.
For $k=K$, we find similarly, from~\eqref{apriori:mutauy} and the definition of $y_{K+1}$ above,
\[
y > -2(1+\alpha) \ell_K^{-2} |s| (1+\rho) \geq \ell_K^{-2} |s| (2\alpha - 2 - 3\rho)
\geq y_K(s) - 3\g\ell_K |s| = \frac 12 \Bigl[ y_{K+1}(s) + y_K(s) \Bigr]
\]
by taking again $\alpha > 0$ small enough, namely $\alpha \leq \frac {\rho}{2(\rho+2)}$.

Thus, the claim~\eqref{ref:gamma} is proved, and we deduce directly~\eqref{b:Pky}
from this claim, the definition of $P_k$ and the properties of $P$ as recalled in Lemma~\ref{lem:L}.
Finally, estimates~\eqref{eq:Pknorms} are obtained as direct consequences of~\eqref{b:Pky},
which concludes the proof of Lemma~\ref{lem:Pk}.
\end{proof}

In order to handle the growth of $\tilde P_k$ at $-\infty$, we also introduce the local norms
\[
\|f\|_{L^\infty_k} = \sup_{y\in \R} \left| f(y) e^{-\frac {|y-y_k|}{10\tilde\mu_k}} \right| \quad\m{and}\quad
\|f\|_{L^2_k}^2 = \int f^2(y) e^{-\frac {|y-y_k|}{10\tilde\mu_k}} \,dy.
\]

We may now prove the following key result for our analysis, which translates in terms of~$W_k$
the properties satisfied by $S$ given in Theorems~\ref{thm:maintimeold} and~\ref{thm:mainspaceold}
and in Corollary~\ref{cor:weighted}.

\begin{lemma}[Estimates on $W_k$] \label{lem:Wk}
Let $1\leq k\leq K$ and $1\leq m\leq 20$.
Then the function $W_k$ satisfies the following, for all $s\in I$.
\begin{enumerate}[label=\emph{(\roman*)}]
\item \emph{Sobolev norms estimates:}
\be \label{eq:WkHm}
\|W_k - Q_k\|_{L^2} \lesssim |s|^{-\frac 12},\quad
\|W_k - Q_k\|_{\dot H^m} \lesssim |s|^{-1},
\ee
and, in particular,
\be \label{eq:WkLinfty}
\|W_k - Q_k\|_{L^\infty} \lesssim |s|^{-\frac 34}.
\ee
More precisely,
\be \label{eq:WkHmbis}
\left\| W_k - Q_k - \frac 1{2\tau_k} \tilde P_k \right\|_{\dot H^m} \lesssim s^{-\frac 32}.
\ee
\item \emph{Exponential weighted estimate:}
\be \label{eq:Wkweighted}
\left\| W_k - Q_k - \frac 1{2\tau_k} \tilde P_k \right\|_{L^\infty_k} \lesssim |s|^{-2}.
\ee
\item \emph{Pointwise asymptotics on the left:}
for all $y \leq y_k - |s|^{\frac 34}$,
\begin{gather}
\left| W_k(s,y) + \frac {\eps_k}2 \|Q\|_{L^1} \mu_k\sqrt{-2\tau_k} |y-z_k|^{-\frac 32} \right|
\lesssim |s|^{\frac 12 + \frac 1{42}} |y-z_k|^{-\frac 32 - \frac 1{21}}, \label{eq:Wkleft0} \\
|\pa_y^m W_k(s,y)| \lesssim |s|^{\frac 12} |y-z_k|^{-\frac 32 - m}. \label{eq:Wkleftm}
\end{gather}
\item \emph{Pointwise bounds on the right:}
there exists $\rho_{m-1} > 0$ such that, for all $y\in\R$,
\be \label{eq:Wkright}
|\pa_y^{m-1} W_k(s,y)| \lesssim \exp\left[ -\rho_{m-1} \left( \frac {y-y_k}{\tilde\mu_k} \right) \right].
\ee
\item \emph{Polynomial weighted estimates:}
\be \label{eq:LkWk}
\| (\cdot-y_k) (W_k - Q_k) \|_{\dot H^m} \lesssim |s|^{-\frac 12}.
\ee
\item \emph{Equation:} \label{item:eqWk}
\be \label{eqWk}
\pa_s W_k + \frac 1{2s} \Lambda W_k + \pa_y\left( \pa_{yy} W_k + W_k^5 \right) = \vec m_k^0 \cdot \vec \M_k W_k,
\ee
with $\vec \M_k = \begin{pmatrix} -\Lambda_k \\ -\pa_y \end{pmatrix}$,
$\Lambda_k = \frac 12 + (y-y_k)\pa_y$ and
\[
\vec m_k^0 = \begin{pmatrix} \ds \frac {\dot \mu_k}{\mu_k} + \frac 1{2\mu_k^3 \tau_k} - \frac 1{2s} \\[12pt]
\ds \dot y_k - \frac {y_k}{2s} - \frac 1{\mu_k^2} -\frac {c_0}{2\mu_k^2\tau_k} + \frac {3c_1}{4\mu_k^2\tau_k^2} \end{pmatrix}.
\]
\end{enumerate}
\end{lemma}

\begin{proof}
(i) First note that~\eqref{apriori:mutauy} and~\eqref{eq:WkHmbis} imply directly
the second estimate in~\eqref{eq:WkHm} since $P'\in\Y$ from Lemma~\ref{lem:L}.
Note also that~\eqref{eq:WkLinfty} is a direct consequence of~\eqref{eq:WkHm} from the standard inequality
$\|f\|_{L^\infty}^2 \lesssim \|f\|_{L^2} \|\pa_y f\|_{L^2}$, valid for any $f\in H^1(\R)$.
Thus, we just have to prove the first estimate in~\eqref{eq:WkHm} and~\eqref{eq:WkHmbis}.

To prove the first estimate in~\eqref{eq:WkHm},
we apply~\eqref{th:timem} with $m=0$ and obtain, for all $t\in (0,t_0]$,
\[
\left\| S(t) - \frac 1{t^{\frac 12}} Q\left( \frac {\cdot + \frac 1t}t + c_0 \right) \right\|_{L^2} \lesssim t.
\]
Thus, applying the change of variables~\eqref{changevar}, we get, for all $s\leq s_0$,
\[
\left\| \St(s) - Q\left( \cdot - 2s + c_0 \right) \right\|_{L^2} \lesssim |s|^{-\frac 12}
\]
and so, from the definition of $W_k$, for all $s\in I$,
\[
\left\| W_k -\eps_k \mu_k^{-\frac 12} Q\left( \frac {\cdot-y_k}{\mu_k}
+ \frac {c_1}{2\tau_k} \right) \right\|_{L^2} \lesssim |s|^{-\frac 12}.
\]
But, from~\eqref{apriori:mutilde} and the \emph{a priori} estimate~\eqref{apriori:mutauy} on $\tau_k$, we have
\[
\left\| \mu_k^{-\frac 12} Q\left( \frac {\cdot}{\mu_k} + \frac {c_1}{2\tau_k} \right)
- \tilde\mu_k^{-\frac 12} Q\left( \frac {\cdot}{\tilde\mu_k} \right) \right\|_{L^2} \lesssim |s|^{-1}.
\]
We deduce the first estimate in~\eqref{eq:WkHm} from the two above estimates and the definition of $Q_k$.

To prove~\eqref{eq:WkHmbis}, we proceed similarly and first get from~\eqref{th:timem},
applied with any $1\leq m\leq 20$, for all $t\in (0,t_0]$,
\[
\left\| \pa_x^m S(t) - \frac 1{t^{\frac 12 + m}} Q^{(m)}\left( \frac {\cdot + \frac 1t}t + c_0 \right)
- \frac 1{t^{\frac 12 + m - 2}} Q_1^{(m-1)}\left( \frac {\cdot + \frac 1t}t + c_0 \right) \right\|_{L^2} \lesssim t^{3-m}.
\]
Note that we may obtain a sharper estimate in the case $m\geq 2$, but the above one will be enough for our purpose.
Thus, applying the change of variables~\eqref{changevar}, we get, for all $s\leq s_0$,
\[
\left\| \pa_y^m \St(s) - Q^{(m)}\left( \cdot - 2s + c_0 \right)
+ \frac 1{2s} Q_1^{(m-1)}\left( \cdot - 2s + c_0 \right) \right\|_{L^2} \lesssim |s|^{-\frac 32}
\]
or equivalently, from the expression~\eqref{def:Q1} of $Q_1$,
\[
\left\| \St(s) - \left[ Q + \frac 1{2s} P + \frac {\lambda_0}{2s} \Lambda Q
- \frac {c_1}{2s} Q' \right] (\cdot - 2s + c_0) \right\|_{\dot H^m} \lesssim |s|^{-\frac 32}.
\]
Noticing that, from a Taylor expansion with remainder of integral form, for all $j\geq 0$,
\be \label{Taylorscaling}
\left\| \left( 1 + \frac {\lambda_0}{2s} \right)^{\frac 12}
Q\left[ \left( 1 + \frac {\lambda_0}{2s} \right) \left( \cdot - \frac {c_1}{2s} \right) \right]
- \left( Q + \frac {\lambda_0}{2s} \Lambda Q - \frac {c_1}{2s} Q' \right) \right\|_{\dot H^j} \lesssim |s|^{-2},
\ee
we obtain
\[
\left\| \St(s) - \left( 1 + \frac {\lambda_0}{2s} \right)^{\frac 12} \left( Q + \frac 1{2s} P \right)
\left[ \left( 1 + \frac {\lambda_0}{2s} \right) \left( \cdot - 2s + c_0
- \frac {c_1}{2s} \right) \right] \right\|_{\dot H^m} \lesssim |s|^{-\frac 32}.
\]
From the definitions of $\tilde\mu_k$ and $W_k$, it gives, for all $s\in I$,
\[
\left\| W_k -\eps_k \tilde\mu_k^{-\frac 12} \left( Q + \frac 1{2\tau_k} P \right)
\left( \frac {\cdot-y_k}{\tilde\mu_k} \right) \right\|_{\dot H^m} \lesssim |s|^{-\frac 32},
\]
then~\eqref{eq:WkHmbis} from the definitions of $Q_k$ and $\tilde P_k$.

\medskip

(ii) The proof of~\eqref{eq:Wkweighted} follows closely the one of~\eqref{eq:WkHmbis} above.
From~\eqref{eq:timeweightedweak} applied with $B=10$, $M=1$ and $m=0,1$, we first obtain, for all $(0,t_0]$,
\[
\left\| \left[ S(t) - \frac 1{t^{\frac 12}} \left( Q + t^2 Q_1^{(-1)} \right)
\left( \frac {\cdot + \frac1t}t + c_0 \right) \right]
e^{-\frac {|\cdot + \frac 1t|}{10t}} \right\|_{L^\infty} \lesssim t^{\frac 72},
\]
where $Q_1^{(-1)} = -P - \lambda_0\Lambda Q + c_1 Q'$ from~\eqref{antideriv}, \eqref{est:P} and~\eqref{def:Q1},
and so, by~\eqref{changevar},
\[
\left\| \left[ \St(s) - \Bigl( Q + \frac 1{2s} P + \frac {\lambda_0}{2s} \Lambda Q
- \frac {c_1}{2s} Q' \Bigr) (\cdot - 2s + c_0) \right]
e^{-\frac {|\cdot - 2s|}{10}} \right\|_{L^\infty} \lesssim |s|^{-2}.
\]
Using~\eqref{Taylorscaling} for $j=0,1$, we get, for all $s\leq s_0$,
\[
\left\| \left[ \St(s) - \left( 1 +\frac {\lambda_0}{2s} \right)^{\frac 12} \left( Q + \frac 1{2s} P \right)
\left( \Bigl( 1 + \frac {\lambda_0}{2s} \Bigr) \Bigl( \cdot - 2s + c_0 - \frac {c_1}{2s} \Bigr) \right) \right]
e^{-\frac {|\cdot - 2s|}{10}} \right\|_{L^\infty} \lesssim |s|^{-2}.
\]
From the definitions of $\tilde\mu_k$ and $W_k$, it gives, for all $s\in I$,
\[
\left\| \left[ W_k -\eps_k \tilde\mu_k^{-\frac 12} \left( Q + \frac 1{2\tau_k} P \right)
\left( \frac {\cdot - y_k}{\tilde\mu_k} \right) \right]
e^{-\frac {|\cdot - y_k|}{10\tilde\mu_k}} \right\|_{L^\infty} \lesssim |s|^{-2},
\]
then~\eqref{eq:Wkweighted} from the definitions of $Q_k$, $\tilde P_k$ and $L^\infty_k$.

\medskip

(iii) From~\eqref{th:ptwise0}, \eqref{th:ptwisem} and~\eqref{changevar},
we obtain, for all $y\leq 2s - \sqrt{-2s}$,
\[
\left| \St(s,y) + \frac 12 \|Q\|_{L^1} \sqrt{-2s} |y|^{-\frac 32} \right|
\lesssim |s|^{\frac 12 + \frac 1{42}} |y|^{-\frac 32 - \frac 1{21}} \quad\m{ and }\quad
|\pa_y^m \St(s,y)| \lesssim |s|^{\frac 12} |y|^{-\frac 32 - m}.
\]
We deduce~\eqref{eq:Wkleft0} and~\eqref{eq:Wkleftm} from these two estimates,
the \emph{a priori} estimates~\eqref{apriori:mutauy} and the definition of $W_k$.

\medskip

(iv) From~\eqref{th:ptwiser} and~\eqref{changevar}, we obtain, for all $s\leq s_0$ and all $y\in\R$,
\[
|\pa_y^{m-1} \St(s,y)| \lesssim \exp\Bigl[ -\g_{m-1} (y - 2s) \Bigr],
\]
which gives, from the definition of $W_k$ and~\eqref{apriori:mutilde}, for all $s\in I$ and all $y\geq y_k$,
\[
|\pa_y^{m-1} W_k(s,y)| \lesssim \exp\left[ -\g_{m-1} \left( \frac {y-y_k}{\mu_k} \right) \right]
\lesssim \exp\left[ -\frac 23 \g_{m-1} \left( \frac {y-y_k}{\tilde\mu_k} \right) \right],
\]
then~\eqref{eq:Wkright} by letting $\rho_{m-1} = \frac 23 \g_{m-1} > 0$.
Note that~\eqref{eq:Wkright} holds also obviously for $y\leq y_k$
since $\|\pa_y^{m-1} W_k\|_{L^\infty} \lesssim 1$ from~\eqref{eq:WkHm}.

\medskip

(v) To prove~\eqref{eq:LkWk}, we first notice that
$\| \pa_y^{m-1} (W_k - Q_k)\|_{L^2} \lesssim |s|^{-\frac 12}$ from~\eqref{eq:WkHm}.
Then we decompose $\| (\cdot - y_k) \pa_y^m (W_k - Q_k)\|_{L^2}$
on the three regions $y < 2y_k$, $2y_k \leq y \leq 0$ and $y > 0$.
Indeed, using first~\eqref{eq:Wkleftm} and the exponential decay of $Q$, we obtain
\begin{align*}
\| (\cdot-y_k) \pa_y^m &(W_k - Q_k) \|_{L^2(y < 2y_k)}^2 \\
&\lesssim |s| \int_{y < 2y_k} (y_k-y)^2 (z_k-y)^{-3 - 2m} \,dy
+ \int_{y < 2y_k} (y_k-y)^2 e^{\frac {2(y-y_k)}{\tilde\mu_k}} \,dy \\
&\lesssim |s| \int_{y < 2y_k} (y_k-y)^{-3} \,dy + \int_{y < 2y_k} e^{\frac {y-y_k}{\tilde\mu_k}} \,dy
\lesssim |s| |y_k|^{-2} + e^{y_k/\tilde\mu_k} \lesssim |s|^{-1},
\end{align*}
from~\eqref{apriori:mutauy}, \eqref{apriori:z} and the fact that we assume $m\geq 1$.
Next, from~\eqref{eq:WkHmbis}, we get
\[
\| (\cdot-y_k) \pa_y^m (W_k - Q_k) \|_{L^2(2y_k\leq y\leq 0)}
\lesssim |s| \left\|W_k - Q_k -\frac 1{2\tau_k} \tilde P_k \right\|_{\dot H^m} \!
+ |s|^{-1} \| (\cdot-y_k) \pa_y^m \tilde P_k \|_{L^2} \lesssim |s|^{-\frac 12},
\]
using also the exponential decay of $P'\in\Y$.
Finally, from~\eqref{eq:Wkright} and the exponential decay of $Q$, we obtain,
with $\rho'_m = \min(\rho_m,1) > 0$,
\[
\| (\cdot - y_k) \pa_y^m (W_k - Q_k) \|_{L^2(y > 0)}^2
\lesssim \int_{y > 0} e^{-\rho'_m \left( \frac {y-y_k}{\tilde\mu_k} \right)} \,dy
\lesssim e^{\rho'_m y_k/\tilde\mu_k} \lesssim |s|^{-10}.
\]
Gathering the above estimates, we obtain~\eqref{eq:LkWk}.

\medskip

(vi) First, note that $\eps_k \St$ satisfies~\eqref{rescaledkdv} and
\[
\eps_k \St(\tau_k(s),y) = \mu_k^{\frac 12}(s) W_k(s,\mu_k(s) y + z_k(s)).
\]
Using~\eqref{def:tauk}, we compute
\begin{align*}
\eps_k \pa_s \St(\tau_k,y)
= \mu_k^3 \pa_s \left[ \eps_k \St(\tau_k,y) \right]
&= \mu_k^{\frac 72} \left[ \frac {\dot\mu_k}{\mu_k} \frac {W_k}2 + \pa_s W_k
+ \dot\mu_k y \pa_y W_k + \dot z_k \pa_y W_k \right] (s,\mu_k y + z_k) \\
&= \mu_k^{\frac 72} \left[ \pa_s W_k + \frac {\dot\mu_k}{\mu_k} \Lambda W_k
+ \left( \dot z_k - \frac {\dot \mu_k}{\mu_k} z_k \right) \pa_y W_k \right] (s,\mu_k y + z_k),
\end{align*}
then
\begin{align*}
\frac {\eps_k}{2\tau_k} \Lambda \St(\tau_k,y)
&= \mu_k^{\frac 12} \frac 1{2\tau_k} \left[ \frac {W_k}2 + \mu_k y \pa_y W_k \right] (s,\mu_k y + z_k)\\
&= \mu_k^{\frac 72} \frac 1{2\mu_k^3 \tau_k} \Bigl[ \Lambda W_k - z_k \pa_y W_k \Bigr] (s,\mu_k y + z_k),
\end{align*}
and
\[
\eps_k \pa_y (\pa_{yy}\St + \St^5) (\tau_k,y)
= \mu_k^{\frac 72} \left[ \pa_y \left( \pa_{yy} W_k + W_k^5 \right) \right] (s,\mu_k y + z_k).
\]
Thus, summing the above terms and using the definition of $z_k$, $W_k$ satisfies the equation~\eqref{eqWk},
which concludes the proof of Lemma~\ref{lem:Wk}.
\end{proof}

For $1\leq k\leq K$, we consider additional $\C^1$ functions $r_k$, $d_k$ and $a_k$
defined on $I$ by
\be \label{def:rkdk}
r_k(s) = \eps_k \tilde\mu_k^{\frac 12}(s) \sum_{j\neq k} W_j(s,y_k(s)),\quad
d_k(s) = \eps_k \tilde\mu_k^{\frac 32}(s) \sum_{j\neq k} \pa_y W_j(s,y_k(s)),
\ee
and $a_k$ to be determined later.
We assume that $a_k$ satisfies, for all $s\in I$,
\be \label{apriori:a}
|a_k(s)| \leq |s|^{-1},
\ee
and we observe that, from~\eqref{apriori:mutauy} and~\eqref{eq:Wkleft0}--\eqref{eq:Wkright},
\be \label{apriori:rd}
|r_k(s)| \lesssim |s|^{-1},\quad |d_k(s)| \lesssim |s|^{-2}.
\ee
We will prove more precise asymptotics on $r_k$ and $d_k$ in Lemma~\ref{lem:rk} below.

Finally, for all $s\in I$ and $y\in\R$, let
\[
V_k(s,y) = W_k(s,y) + r_k(s) R_k(s,y) + a_k(s) P_k(s,y)
\]
and define
\[
\W = \sum_{k=1}^K W_k \quad\m{and}\quad
\V = \sum_{k=1}^K V_k = \W +\sum_{k=1}^K (r_k R_k + a_k P_k).
\]

Note that, by~\eqref{eq:Pknorms}, \eqref{eq:WkHm}--\eqref{eq:WkLinfty},
\eqref{eq:LkWk} and~\eqref{apriori:a}--\eqref{apriori:rd}, we have,
for all $1\leq k\leq K$, all $1\leq m\leq 19$ and all $s\in I$,
\be \label{simpleVk}
\left\{
\begin{alignedat}{2}
\|(\cdot - y_k)(V_k-Q_k)\|_{\dot H^m} + \|V_k - Q_k\|_{L^2} &\lesssim |s|^{-\frac 12},
&\|V_k - Q_k\|_{L^\infty} &\lesssim |s|^{-\frac 34}, \\
\|V_k - Q_k\|_{\dot H^m} &\lesssim |s|^{-1},\quad
&\|\pa_y^m (V_k - Q_k)\|_{L^\infty} &\lesssim |s|^{-1},
\end{alignedat}
\right.
\ee
and, in particular, $\|\pa_y^{m-1}\V\|_{L^2} + \|\pa_y^{m-1}\V\|_{L^\infty} \lesssim 1$.

In the next lemma, we prove that such an ansatz $\V$ is indeed an approximate solution
of~\eqref{rescaledkdv}, in a precise sense.
In Lemma~\ref{lem:Vmassenergy} below, we also estimate the mass and the energy of $\V$,
relying on the sharp estimates of Lemma~\ref{lem:rk}.

\begin{lemma}[Approximate rescaled multi-soliton] \label{lem:V}
The error of the flow~\eqref{rescaledkdv} at $\V$, defined as
\[
\E_\V = \pa_s \V + \frac 1{2s} \Lambda \V + \pa_y(\pa_{yy} \V + \V^5),
\]
decomposes as
\be \label{decompEV}
\E_\V = \sum_j \vec m_j \cdot \vec \M_j V_j + \sum_j (\dot r_j R_j + \dot a_j P_j) + \P
\ee
where, for all $1\leq j\leq K$,
\be \label{def:mod}
\vec m_j = \begin{pmatrix} m_{j,1} \\ m_{j,2} \end{pmatrix}
= \begin{pmatrix} \ds \frac {\dot \mu_j}{\mu_j} + \frac 1{2\mu_j^3 \tau_j} - \frac 1{2s} + \frac {a_j}{\mu_j^3} \\[12pt]
\ds \dot y_j - \frac {y_j}{2s} - \frac 1{\mu_j^2} - \frac {c_0}{2\mu_j^2\tau_j} + \frac {3c_1}{4\mu_j^2\tau_j^2} \end{pmatrix}
= \vec m_j^0 + \begin{pmatrix} \dfrac {a_j}{\mu_j^3} \\[12pt] 0 \end{pmatrix},
\ee
and $\vec m_j^0$ and $\vec \M_j$ are defined in~\ref{item:eqWk} of Lemma~\ref{lem:Wk}.
Moreover, for all $s\in I$,
\be \label{est:Psi}
\|\P\|_{H^2} \lesssim |s|^{-\frac 74} + |s|^{-\frac 12} \sum_j |a_j|,
\ee
and, for any $1\leq k\leq K$,
\begin{gather}
\|\P\|_{H^2(y > y_k - |s|^{\frac 14})} \lesssim |s|^{-\frac {13}8}
+ |s|^{-\frac 12} \sum_{j<k} |a_j|,\qquad
\|\P\|_{L^\infty_k} \lesssim |s|^{-2}, \label{est:Psik} \\
|\langle \P,Q_k \rangle - \Omega_k|\lesssim |s|^{-\frac 52}, \label{est:PsiQ}
\end{gather}
with
\[
\Omega_k(s) = \frac {\|Q\|_{L^1}^2}{8\mu_k^3(s)} \left[ \frac {a_k(s)}{\tau_k(s)}
+ a_k^2(s) \right] + \frac {\|Q\|_{L^1}}{\mu_k^3(s)} d_k(s).
\]
\end{lemma}

\begin{proof}
\textbf{Equation of $\V$.}
We insert the definition of $\V = \sum_j (W_j + r_j R_j + a_j P_j)$ in $\E_\V$ and,
using the equation~\eqref{eqWk} satisfied by $W_j$, we obtain
\begin{align*}
\E_\V & = \sum_j \vec m_j^0 \cdot \vec \M_j W_j + \sum_j \left( \dot r_j R_j + \dot a_j P_j \right) + |s|^{-1} \sum_j a_j Z_j \\
& \quad -\sum_j \frac {\dot{\tilde\mu}_j}{\tilde\mu_j} \Lambda_j \left( r_j R_j + a_j P_j \right)
- \sum_j \dot y_j \pa_y\left( r_j R_j + a_j P_j \right) + \frac 1{2s} \sum_j \Lambda \left( r_j R_j + a_j P_j \right) \\
& \quad + \sum_j \pa_{yyy}(r_j R_j + a_j P_j) + \pa_y\left( \V^5 - \sum_j W_j^5 \right),
\end{align*}
where we have denoted
\be \label{def:Zj}
Z_j (s,y) = \left( \frac {y-y_j(s)}{\tilde\mu_j(s)} |s|^{-1} \right)
\eps_j \tilde\mu_j^{-\frac 12}(s) P\left( \frac {y-y_j(s)}{\tilde\mu_j(s)} \right)
\chi'\left( \frac {y-y_j(s)}{\tilde\mu_j(s)} |s|^{-1} \right).
\ee
Since $\ds \frac {\dot{\tilde\mu}_j}{\tilde\mu_j} = \frac {\dot\mu_j}{\mu_j}
+ \frac {\lambda_0}{2\mu_j^3\tau_j^2} \left( 1 + \frac {\lambda_0}{2\tau_j} \right)^{-1}$
from the definition of $\tilde\mu_j$, we may decompose $\E_\V$ as
\[
\E_\V = \sum_j \vec m_j \cdot \vec \M_j V_j + \sum_j \left( \dot r_j R_j + \dot a_j P_j \right) + \P,
\]
where we have set
\begin{align*}
\P &= |s|^{-1} \sum_j a_j Z_j + \sum_j \left[ \frac 1{2\mu_j^3\tau_j} + \frac {a_j}{\mu_j^3}
- \frac {\lambda_0}{2\mu_j^3\tau_j^2} \left( 1 + \frac {\lambda_0}{2\tau_j} \right)^{-1} \right] \Lambda_j(r_j R_j + a_j P_j) \\
&\quad + \sum_j \left[ -\frac {c_0}{2\mu_j^2\tau_j} + \frac {3c_1}{4\mu_j^2\tau_j^2} \right] \pa_y(r_j R_j + a_j P_j) \\
&\quad + \sum_j a_j \mu_j^{-3} \Lambda_j W_j + \sum_j a_j \pa_y\left( \pa_{yy} P_j - \mu_j^{-2} P_j + 5Q_j^4 P_j \right) \\
&\quad + \sum_j r_j \pa_y\left( \pa_{yy} R_j - \mu_j^{-2} R_j + 5Q_j^4 R_j \right)
+ \pa_y\left[ \V^5 - \sum_j W_j^5 - 5\sum_j Q_j^4 (r_j R_j + a_j P_j)\right] \\
&= (\P^\mathrm{I} + \P^\mathrm{II}) + \P^\mathrm{III} + \P^\mathrm{IV} + \P^\mathrm{V}.
\end{align*}

\medskip

\textbf{Estimate of $\P^\mathrm{I}$.}
We rely on~\eqref{est:P} and~\eqref{ref:gamma} to estimate $Z_j$.
For instance, we first note that $\|Z_j\|_{H^2} \lesssim |s|^{\frac 12}$.
Thus,
\[
\|\P^\mathrm{I}\|_{H^2} \lesssim |s|^{-\frac 12} \sum_j |a_j|.
\]
Moreover, we observe that, for $j \geq k$, $Z_j(y) = 0$ for $y > y_k - |s|^{\frac 14}$, and so
\[
\|\P^\mathrm{I}\|_{H^2(y > y_k - |s|^{\frac 14})} \lesssim |s|^{-\frac 12} \sum_{j<k} |a_j|.
\]
Finally, we note that $\|Z_j\|_{L^\infty_k} \lesssim |s|^{-10}$ for all $1\leq j\leq K$.
Thus, using also the exponential decay of $Q$, we get
\[
|\langle \P^\mathrm{I},Q_k \rangle| \lesssim \|\P^\mathrm{I}\|_{L^\infty_k} \lesssim |s|^{-12}.
\]

\medskip

\textbf{Estimate of $\P^\mathrm{II}$.}
First, from~\eqref{apriori:mutauy} and~\eqref{apriori:a}, we notice that
\[
\left| \frac 1{2\mu_j^3\tau_j} + \frac {a_j}{\mu_j^3} \right| \lesssim |s|^{-1} \quad\m{and}\quad
\left| \frac {\lambda_0}{2\mu_j^3\tau_j^2} \left( 1 + \frac {\lambda_0}{2\tau_j} \right)^{-1} \right| \lesssim |s|^{-2}.
\]
Second, from the exponential decay of $R\in\Y$ and~\eqref{eq:Pknorms},
we have $\|\Lambda_j R_j\|_{H^2} \lesssim 1$ and $\|\Lambda_j P_j\|_{H^2} \lesssim |s|^{\frac 12}$.
Thus, using also~\eqref{apriori:rd}, we find
\[
\|\P^\mathrm{II}\|_{H^2} \lesssim |s|^{-2} + |s|^{-\frac 12} \sum_j |a_j|.
\]
Moreover, since $\|\Lambda_j P_j\|_{H^2(y > y_k - |s|^{\frac 14})} \lesssim |s|^{\frac 18}$ for $j\geq k$ from~\eqref{b:Pky},
we have
\[
\|\P^\mathrm{II}\|_{H^2(y > y_k - |s|^{\frac 14})} \lesssim |s|^{-2} + |s|^{-2+\frac 18} + |s|^{-\frac 12} \sum_{j<k} |a_j|
\lesssim |s|^{-\frac {15}8} + |s|^{-\frac 12} \sum_{j<k} |a_j|.
\]
Similarly, since $\|\Lambda_j P_j\|_{L^\infty_k} \lesssim |s|^{-10}$ for $j\neq k$
and $\|\Lambda_k P_k\|_{L^\infty_k} \lesssim 1$, we find
\[
\|\P^\mathrm{II}\|_{L^\infty_k} \lesssim |s|^{-2}.
\]
Finally, projecting on $Q_k$, observing that $\langle \Lambda_k R_k,Q_k \rangle = \langle \Lambda R,Q \rangle$
and, from the properties of~$\chi$,
$\langle \Lambda_k P_k,Q_k \rangle = \langle \Lambda P,Q \rangle + O(|s|^{-10})$,
we obtain, using also~\eqref{ref:gamma},
\[
\left| \langle \P^\mathrm{II},Q_k \rangle - \left( \frac {r_k}{2\mu_k^3\tau_k} + \frac{a_k r_k}{\mu_k^3} \right) \langle \Lambda R,Q \rangle
- \left( \frac {a_k}{2\mu_k^3\tau_k} + \frac {a_k^2}{\mu_k^3} \right) \langle \Lambda P,Q \rangle \right| \lesssim |s|^{-3}.
\]

\medskip

\textbf{Estimate of $\P^\mathrm{III}$.}
We proceed as in the estimate of $\P^\mathrm{II}$ except that, from Lemma~\ref{lem:L},
$\langle R_k,\pa_y Q_k \rangle = \tilde\mu_k^{-1} \langle R,Q' \rangle = 0$ by parity,
and
\[
\langle P_k,\pa_y Q_k \rangle = \langle \tilde P_k,\pa_y Q_k \rangle + O(|s|^{-10})
= \tilde\mu_k^{-1} \langle P,Q' \rangle + O(|s|^{-10}) = O(|s|^{-10}).
\]
Since we also have $\|\pa_y P_j\|_{H^2} + \|\pa_y R_j\|_{H^2} \lesssim 1$ from~\eqref{eq:Pknorms},
we obtain
\[
\|\P^\mathrm{III}\|_{H^2}\lesssim |s|^{-2} \quad\m{and}\quad
|\langle \P^\mathrm{III},Q_k \rangle| \lesssim |s|^{-12}.
\]

\medskip

\textbf{Estimate of $\P^\mathrm{IV}$.}
For $(l,m)\in\N^2$, set
\[
P_j^{(l,m)}(s,y) = \eps_j \tilde\mu_j^{-(\frac 12+l+m)}(s) P^{(l)}\left( \frac {y-y_j(s)}{\tilde\mu_j(s)} \right)
\chi^{(m)}\left( \frac {y-y_j(s)}{\tilde\mu_j(s)} |s|^{-1} \right).
\]
Note for instance that $P_j^{(0,0)} = P_j$.
From the relation $(LP)' = \Lambda Q$ in~\eqref{PQ}, we find
\[
\pa_y\left( \pa_{yy}\tilde P_j - \tilde\mu_j^{-2}\tilde P_j
+ 5Q_j^4 \tilde P_j \right) = -\tilde\mu_j^{-3} \Lambda_j Q_j,
\]
and so
\begin{align*}
&\pa_y \left(\pa_{yy} P_j - \mu_j^{-2} P_j +5Q_j^4 P_j \right) \\
&= -\tilde\mu_j^{-3} (\Lambda_j Q_j) \chi\left( \frac {\cdot - y_j(s)}{\tilde \mu_j(s)} |s|^{-1} \right)
+ \tilde\mu_j (\tilde\mu_j^{-2} - \mu_j^{-2}) P_j^{(1,0)} - \mu_j^{-2} |s|^{-1} P_j^{(0,1)} \\
&\quad + 5|s|^{-1} Q_j^4 P_j^{(0,1)} + 3|s|^{-1} P_j^{(2,1)} + 3|s|^{-2} P_j^{(1,2)} + |s|^{-3} P_j^{(0,3)}.
\end{align*}
Thus,
\begin{align*}
\P^\mathrm{IV} &= \sum_j a_j\mu_j^{-3} \Lambda_j(W_j - Q_j)
+ \sum_j a_j\tilde\mu_j^{-3} \Lambda_j Q_j \left[ 1 - \chi\left( \frac {\cdot - y_j(s)}{\tilde\mu_j(s)} |s|^{-1} \right) \right] \\
&\quad + \sum_j a_j (\mu_j^{-3} - \tilde\mu_j^{-3}) \Lambda_j Q_j + \sum_j a_j \tilde\mu_j (\tilde\mu_j^{-2} - \mu_j^{-2}) P_j^{(1,0)}
- \sum_j a_j \mu_j^{-2} |s|^{-1} P_j^{(0,1)} \\
&\quad + \sum_j a_j |s|^{-1} \left[ 5 Q_j^4 P_j^{(0,1)}
+ 3P_j^{(2,1)} + 3|s|^{-1} P_j^{(1,2)} + |s|^{-2} P_j^{(0,3)} \right] \\
&= \Sigma_1 +\Sigma_2 +\Sigma_3 +\Sigma_4 +\Sigma_5 +\Sigma_6.
\end{align*}

By the properties of $P$ and $\chi$, we have
\begin{gather*}
\|Q_j^4 P_j^{(0,1)}\|_{H^2} + \|P_j^{(2,1)}\|_{H^2} + \|P_j^{(1,2)}\|_{H^2} \lesssim |s|^{-10}, \\
\|\Lambda_j Q_j\|_{H^2} + \|P_j^{(1,0)}\|_{H^2} \lesssim 1, \quad
\|P_j^{(0,1)}\|_{H^2} + \|P_j^{(0,3)}\|_{H^2} \lesssim |s|^{\frac 12}.
\end{gather*}
Thus, using also~\eqref{apriori:mutilde} and~\eqref{apriori:a}, we obtain
\[
\|\Sigma_3\|_{H^2} + \|\Sigma_4\|_{H^2} \lesssim |s|^{-2},\quad
\|\Sigma_5\|_{H^2} \lesssim |s|^{-\frac 12} \sum_j |a_j| \quad\m{and}\quad
\|\Sigma_6\|_{H^2} \lesssim |s|^{-\frac 72}.
\]
By~\eqref{ref:gamma} and the exponential decay of $Q$, we also have $\|\Sigma_2\|_{H^2}\lesssim |s|^{-10}$.
Now, from~\eqref{eq:WkHm} and~\eqref{eq:LkWk}, we notice that
\[
\| \Lambda_j (W_j - Q_j) \|_{H^2} \lesssim |s|^{-\frac 12}, \quad\m{and so}\quad
\|\Sigma_1\|_{H^2} \lesssim |s|^{-\frac 12} \sum_j |a_j|.
\]
Thus, gathering the previous estimates, we have obtained
\[
\|\P^\mathrm{IV}\|_{H^2} \lesssim |s|^{-2} + |s|^{-\frac 12} \sum_j |a_j|.
\]

Moreover, we observe as before from~\eqref{ref:gamma} that, for $j\geq k$,
$P_j^{(0,1)}(y) = 0$ for $y > y_k - |s|^{\frac 14}$, and so
\[
\|\Sigma_5\|_{H^2(y > y_k - |s|^{\frac 14})} \lesssim |s|^{-\frac 12} \sum_{j<k} |a_j|.
\]
Next, we claim that, for all $j\geq k$,
\[
\| \Lambda_j (W_j-Q_j) \|_{H^2(y > y_k - |s|^{\frac 14})} \lesssim |s|^{-\frac 58}.
\]
Note that it is enough to prove the claim for $j=k$ since $y_j \leq y_k$.
To do so, we first notice that
\[
\| W_k-Q_k \|_{\dot H^m(y > y_k - |s|^{\frac 14})} \lesssim |s|^{-1}
\]
for $m=1,2$ from~\eqref{eq:WkHm}.
Second, using~\eqref{eq:WkLinfty} and~\eqref{eq:Wkright}, we estimate
\begin{align*}
\| W_k - Q_k \|_{L^2(y > y_k - |s|^{\frac 14})}^2
&= \| W_k - Q_k \|_{L^2(y_k - |s|^{\frac 14} < y < y_k + |s|^{\frac 14})}^2
+ \| W_k - Q_k \|_{L^2(y > y_k + |s|^{\frac 14})}^2 \\
&\lesssim \| W_k - Q_k \|_{L^\infty}^2 |s|^{\frac 14} + e^{-2\rho_0 |s|^{\frac 14}/\tilde\mu_k}
+ e^{-2|s|^{\frac 14}/\tilde\mu_k} \lesssim |s|^{-\frac 54}.
\end{align*}
Next, to estimate $\|(\cdot - y_k)\pa_y^m(W_k-Q_k)\|_{L^2(y > y_k - |s|^{\frac 14})}$ for $m=1,2,3$,
we proceed similarly and find, using~\eqref{eq:WkHm} and again~\eqref{eq:Wkright},
\[
\|(\cdot - y_k)\pa_y^m(W_k-Q_k)\|_{L^2(y > y_k - |s|^{\frac 14})}^2
\lesssim |s|^{\frac 12} \|W_k - Q_k\|_{\dot H^m}^2 + e^{-\rho_0 |s|^{\frac 14}/\tilde\mu_k}
+ e^{-|s|^{\frac 14}/\tilde\mu_k} \lesssim |s|^{-\frac 32}.
\]
Thus the claim is proved and, together with the above estimates, it gives
\[
\|\P^\mathrm{IV}\|_{H^2(y > y_k - |s|^{\frac 14})} \lesssim |s|^{-\frac {13}8} + |s|^{-\frac 12} \sum_{j<k} |a_j|.
\]

Now, we control $\|\Sigma_1\|_{L^\infty_k}$ and $\|\Sigma_5\|_{L^\infty_k}$.
First, by~\eqref{ref:gamma}, we find $\|P_j^{(0,1)}\|_{L^\infty_k} \lesssim |s|^{-10}$
for all $1\leq j\leq K$, and so
\[
\|\Sigma_5\|_{L^\infty_k} \lesssim |s|^{-12}.
\]
Next, by~\eqref{eq:WkHm} and~\eqref{eq:Wkweighted}, we find $\|\Lambda_k(W_k - Q_k)\|_{L^\infty_k} \lesssim |s|^{-1}$.
But, for $j\neq k$, by~\eqref{apriori:mutauy}, \eqref{apriori:z} and~\eqref{eq:Wkleft0}--\eqref{eq:Wkright},
we have $\|\Lambda_j(W_j - Q_j)\|_{L^\infty_k} \lesssim |s|^{-10}$.
Thus, using~\eqref{apriori:a}, we obtain $\|\Sigma_1\|_{L^\infty_k} \lesssim |s|^{-2}$,
which proves, together with the above estimates,
\[
\|\P^\mathrm{IV}\|_{L^\infty_k} \lesssim |s|^{-2}.
\]

Finally, we look at the projection on $Q_k$.
Note that, from~\eqref{eq:Wkweighted},
\begin{align*}
\langle \Lambda_k (W_k - Q_k),Q_k \rangle
&= - \langle W_k - Q_k, \Lambda_k Q_k \rangle
= - \frac 1{2\tau_k} \langle \tilde P_k,\Lambda_k Q_k \rangle + O(|s|^{-2}) \\
&= - \frac 1{2\tau_k} \langle P,\Lambda Q \rangle + O(|s|^{-2})
= \frac 1{2\tau_k} \langle \Lambda P, Q \rangle + O(|s|^{-2}).
\end{align*}
Moreover, as before, we have
$|\langle \Lambda_j(W_j - Q_j),Q_k \rangle| \lesssim \|\Lambda_j(W_j - Q_j)\|_{L_k^\infty} \lesssim |s|^{-10}$
for $j\neq k$, and similarly
$|\langle \Sigma_5,Q_k \rangle| \lesssim \|\Sigma_5\|_{L^\infty_k} \lesssim |s|^{-12}$.
Also, since $\langle \Lambda_k Q_k,Q_k \rangle = \langle \Lambda Q,Q \rangle = 0$
and $\langle P_k^{(1,0)},Q_k \rangle = -\tilde\mu_k^{-1} \langle P,Q' \rangle + O(|s|^{-10}) = O(|s|^{-10})$
from~\eqref{PQ}, we find $|\langle \Sigma_3 + \Sigma_4,Q_k \rangle| \lesssim |s|^{-12}$.
To conclude for this term, we have obtained, using also~\eqref{apriori:a},
\[
\left| \langle \P^\mathrm{IV},Q_k \rangle - \frac {a_k}{2\mu_k^3\tau_k}
\langle \Lambda P,Q \rangle \right| \lesssim |s|^{-3}.
\]

\medskip

\textbf{Decomposition of $\P^\mathrm{V}$.}
From the relation $LR = 5Q^4$ in~\eqref{def:R}, we find
\[
\pa_{yy} R_j - \tilde\mu_j^{-2} R_j + 5Q_j^4 R_j = -5\eps_j \tilde\mu_j^{-\frac 12} Q_j^4,
\]
and so
\[
\P^\mathrm{V} = \pa_y\left[ \V^5 - \sum_j W_j^5 - 5 \sum_j Q_j^4 (r_j R_j + a_j P_j)
+ \sum_j r_j (\tilde\mu_j^{-2} - \mu_j^{-2}) R_j - 5 \sum_j r_j \eps_j \tilde\mu_j^{-\frac 12} Q_j^4 \right].
\]
Thus, we may further decompose $\P^\mathrm{V}$ as
\[
\P^\mathrm{V} = \P^\mathrm{VI} + \P^\mathrm{VII} + \P^\mathrm{VIII},
\]
with
\begin{align*}
\P^\mathrm{VI} &= \pa_y\left[ \V^5 - \W^5 - 5 \sum_j Q_j^4 (r_j R_j + a_j P_j) \right], \\
\P^\mathrm{VII} &= \pa_y\left[ \W^5 - \sum_j W_j^5 - 5 \sum_j r_j \eps_j \tilde\mu_j^{-\frac 12} Q_j^4 \right], \\
\P^\mathrm{VIII} &= \pa_y\left[ \sum_j r_j (\tilde\mu_j^{-2} - \mu_j^{-2}) R_j \right].
\end{align*}

\medskip

\textbf{Estimate of $\P^\mathrm{VI}$.}
A binomial expansion first gives
\[
\V^5 - \W^5 = \Bigl[ \W + (\V-\W) \Bigr]^5 - \W^5 = \sum_{i=1}^5 \binom 5i \W^{5-i} (\V - \W)^i.
\]
We estimate each term of the sum separately, and we recall that $\V - \W = \sum_j (r_j R_j + a_j P_j)$.
First, for $i=2$, by~\eqref{eq:WkHm} and~\eqref{apriori:a}--\eqref{apriori:rd},
\[
\|\W^3 (\V - \W)^2\|_{H^3} \lesssim \|\W\|_{H^3}^3 \sum_{p=0}^3 \|\pa_y^p(\V - \W)\|_{L^\infty}^2
\lesssim \sum_j \left( r_j^2 + a_j^2 \right) \lesssim |s|^{-2}.
\]
Similarly, for $i=3,4$,
\[
\|\W^{5-i} (\V - \W)^i\|_{H^3} \lesssim \|\W\|_{H^3}^{5-i} \sum_{p=0}^3 \|\pa_y^p(\V - \W)\|_{L^\infty}^i
\lesssim |s|^{-i}.
\]
And, for $i=5$, using also~\eqref{eq:Pknorms},
\[
\|(\V - \W)^5\|_{H^3} \lesssim \|\V - \W\|_{H^3} \sum_{p=0}^3 \|\pa_y^p(\V - \W)\|_{L^\infty}^4
\lesssim |s|^{-\frac 12} |s|^{-4} = |s|^{-\frac 92}.
\]
To estimate the term corresponding to $i=1$,
\[
\W^4(\V - \W) = \left( \sum_j W_j \right)^4 \left( \sum_j (r_j R_j + a_j P_j) \right)
= \sum_{j_1,\ldots,j_5} (r_{j_1} R_{j_1} + a_{j_1} P_{j_1}) \prod_{l=2}^5 W_{j_l},
\]
we rely on the following claim.
Let $2\leq p\leq 5$. Let $j_1,\ldots,j_p\in\{1,\ldots,K\}$ with $j_1 \neq j_l$ for $2\leq l\leq p$.
Then, from the decay properties of $W_j$, $R_j$ and $P_j$, we have
\begin{gather*}
\left\| R_{j_1} W_{j_1}^{5-p} \prod_{l=2}^p W_{j_l} \right\|_{H^3} \lesssim |s|^{-p + 1} \quad\m{and}\quad
\left\| P_{j_1} W_{j_1}^{5-p} \prod_{l=2}^p W_{j_l} \right\|_{H^3} \lesssim |s|^{-p + \frac 32}.
\end{gather*}
Indeed, if $j_1 > j_2,\ldots,j_p$, we first find,
decomposing on the two regions $y < y_{j_1} + |s|^{\frac 14}$
and $y > y_{j_1} + |s|^{\frac 14}$, and using~\eqref{eq:Wkleft0},
\begin{align*}
\left\| R_{j_1} W_{j_1}^{5-p} \prod_{l=2}^p W_{j_l} \right\|_{L^2}^2
&\lesssim |s|^{p-1} \int_{y < y_{j_1} + |s|^{\frac 14}} \left( e^{-\frac {|y-y_{j_1}|}{\tilde\mu_{j_1}}}
\prod_{l=2}^p |y-z_{j_l}|^{-3}\right) dy + e^{-|s|^{\frac 14}/\tilde\mu_{j_1}} \\
&\lesssim |s|^{p-1} |s|^{-3(p-1)} \int e^{-\frac {|y-y_{j_1}|}{\tilde\mu_{j_1}}} \,dy + |s|^{-10}
\lesssim |s|^{-2(p-1)}.
\end{align*}
Note that such an estimate also holds in $\dot H^m$ for $m=1,2,3$ from~\eqref{eq:Wkleftm}.
In the case where there exists $2\leq l\leq p$ such that $j_l > j_1$, we find,
decomposing on the regions $y < y_{j_1} - |s|^{\frac 14}$
and $y > y_{j_1} - |s|^{\frac 14}$, and using~\eqref{eq:Wkright},
\[
\left\| R_{j_1} W_{j_1}^{5-p} \prod_{l=2}^p W_{j_l} \right\|_{H^3}
\lesssim e^{-|s|^{\frac 14}/\tilde\mu_{j_1}} + e^{-\rho_0|s|/\tilde\mu_{j_l}} \lesssim |s|^{-10}.
\]
Proceeding similarly with $P_j$ and~\eqref{b:Pky}, the claim is proved and we obtain, as a consequence,
\[
\left\| \W^4(\V - \W) - \sum_j Q_j^4 (r_j R_j + a_j P_j) - \sum_j (W_j^4 - Q_j^4)
(r_j R_j + a_j P_j) \right\|_{H^3} \!\! \lesssim |s|^{-2} + |s|^{-\frac 12} \sum_j |a_j|.
\]
But we also observe that, using~\eqref{eq:WkHm} and~\eqref{eq:WkLinfty},
\begin{align*}
&\left\| (W_j^4 - Q_j^4) (r_j R_j + a_j P_j) \right\|_{H^3} \\
&\lesssim \left( \|W_j\|_{H^3}^3 + \|Q_j\|_{H^3}^3 \right)
\left( \sum_{p=0}^3 \| \pa_y^p(r_j R_j + a_j P_j)\|_{L^\infty} \right)
\left( \sum_{p=0}^3 \| \pa_y^p(W_j - Q_j)\|_{L^\infty} \right) \\
&\lesssim \left( |r_j| + |a_j| \right) |s|^{-\frac 34} \lesssim |s|^{-\frac 74}.
\end{align*}
Gathering the previous estimates, we have obtained
\[
\|\P^\mathrm{VI}\|_{H^2} \lesssim |s|^{-\frac 74} + |s|^{-\frac 12} \sum_j |a_j|,
\]
but also the more precise estimates
\be \label{pourPsiVIbis}
\left\| \P^\mathrm{VI} - \pa_y\left[ 20 \sum_j a_j P_j W_j^3
\sum_{j_2\neq j} W_{j_2} \right] \right\|_{H^2} \lesssim |s|^{-\frac 74}
\ee
and
\be \label{pourPsiVI}
\begin{aligned}
\left\| \P^\mathrm{VI} - \pa_y\left[ 5 \sum_j (W_j^4 - Q_j^4) \right. \right.
& (r_j R_j + a_j P_j) + 10 \W^3 (\V - \W)^2\\
&\left.\left. {}+ 20 \sum_j (r_j R_j + a_j P_j) W_j^3 \sum_{j_2\neq j} W_{j_2} \right] \right\|_{H^2}
\lesssim |s|^{-\frac 52}.
\end{aligned}
\ee

For $j\geq k$, we note that $\|P_j\|_{H^3(y > y_k - |s|^{\frac 14})} \lesssim |s|^{\frac 18}$
from~\eqref{b:Pky}, and thus, using~\eqref{pourPsiVIbis},
\[
\|\P^\mathrm{VI}\|_{H^2(y > y_k - |s|^{\frac 14})} \lesssim |s|^{-\frac 74} + |s|^{-2+\frac 18} + |s|^{-\frac 12} \sum_{j<k} |a_j|
\lesssim |s|^{-\frac 74} + |s|^{-\frac 12} \sum_{j<k} |a_j|.
\]

To control $\|\P^\mathrm{VI}\|_{L^\infty_k}$, we rely on~\eqref{pourPsiVI}.
First, as observed before,
\[
\|\pa_y[\W^3(\V - \W)^2] \|_{L^\infty_k} \lesssim \|\W^3 (\V - \W)^2 \|_{H^2} \lesssim |s|^{-2}.
\]
Next, for $j\neq k$, it follows from similar arguments as before that
\[
\left\| \pa_y\left[ (W_j^4 - Q_j^4) (r_j R_j + a_j P_j) \right] \right\|_{L^\infty_k}
+ \left\| \pa_y\left[ (r_j R_j + a_j P_j) W_j^3 \sum_{j_2\neq j} W_{j_2} \right] \right\|_{L^\infty_k} \lesssim |s|^{-10}.
\]
Finally, by~\eqref{eq:WkHm}, \eqref{eq:Wkweighted} and $|a_k| + |r_k| \lesssim |s|^{-1}$,
\[
\left\| \pa_y\left[ (W_k^4 - Q_k^4) (r_k R_k + a_k P_k) \right] \right\|_{L^\infty_k} \lesssim |s|^{-2},
\]
and, by~\eqref{eq:Wkleft0}--\eqref{eq:Wkright},
\[
\left\| \pa_y\left[(r_k R_k + a_k P_k) W_k^3 \sum_{k_2\neq k} W_{k_2} \right] \right\|_{L^\infty_k} \lesssim |s|^{-2},
\]
which proves
\[
\|\P^\mathrm{VI}\|_{L^\infty_k} \lesssim |s|^{-2}.
\]

Concerning the projection on $Q_k$, we note from~\eqref{eq:WkLinfty} and~\eqref{pourPsiVI} that
\begin{multline*}
\left| \langle \P^\mathrm{VI},Q_k \rangle + 5 \sum_j \left\langle (W_j^4 - Q_j^4) (r_j R_j + a_j P_j),\pa_y Q_k \right\rangle
+ 10 \left\langle \W^3 (\V - \W)^2,\pa_y Q_k \right\rangle \right. \\
\left. {}+ 20 \sum_j \left\langle (r_j R_j + a_j P_j) Q_j^3
\sum_{j_2\neq j} W_{j_2},\pa_y Q_k \right\rangle \right| \lesssim |s|^{-\frac 52}.
\end{multline*}
Since all the terms corresponding to $j\neq k$ in this estimate are controlled by $|s|^{-10}$,
and since we find similarly as above
\[
\left| \left\langle \W^3(\V - \W)^2 - W_k^3 (r_k R_k + a_k P_k)^2,\pa_y Q_k \right\rangle \right| \lesssim |s|^{-3},
\]
we obtain, using again~\eqref{eq:WkLinfty},
\begin{multline*}
\left| \vphantom{\sum_k} \langle \P^\mathrm{VI},Q_k \rangle
+ 5 \left\langle (W_k^4 - Q_k^4) (r_k R_k + a_k P_k),\pa_y Q_k \right\rangle
+ 10 \left\langle Q_k^3 (r_k R_k + a_k P_k)^2,\pa_y Q_k \right\rangle \right. \\
\left. {}+ 20 \left\langle (r_k R_k + a_k P_k) Q_k^3 \sum_{k_2\neq k} W_{k_2},\pa_y Q_k \right\rangle \right|
\lesssim |s|^{-\frac 52}.
\end{multline*}
By~\eqref{apriori:mutilde} and~\eqref{eq:Wkweighted}, we have
\begin{align*}
&5 \left\langle (W_k^4 - Q_k^4) (r_k R_k + a_k P_k),\pa_y Q_k \right\rangle \\
&= 10 \frac {r_k}{\tau_k} \langle Q_k^3 \tilde P_k R_k,\pa_y Q_k \rangle
+ 10 \frac {a_k}{\tau_k} \langle Q_k^3 \tilde P_k P_k ,\pa_y Q_k \rangle + O(|s|^{-\frac 52}) \\
&= 10 \frac {r_k}{\mu_k^3\tau_k} \langle Q^3 P R,Q' \rangle + 10 \frac {a_k}{\mu_k^3\tau_k}
\langle Q^3 P^2 ,Q' \rangle + O(|s|^{-\frac 52}).
\end{align*}
Moreover, since $\langle Q_k^3 R_k^2,\pa_y Q_k \rangle = \tilde \mu_k^{-3} \langle Q^3 R^2, Q' \rangle = 0$
by parity, we have
\begin{align*}
10 \left\langle Q_k^3 (r_k R_k + a_k P_k)^2,\pa_y Q_k \right\rangle
& = 20 r_k a_k \langle Q_k^3 R_k P_k,\pa_y Q_k \rangle + 10 a_k^2 \langle Q_k^3 P_k^2,\pa_y Q_k \rangle \\
& = 20 \frac {r_k a_k}{\mu_k^3} \langle Q^3 R P, Q' \rangle + 10 \frac {a_k^2}{\mu_k^3} \langle Q^3 P^2,Q' \rangle + O(|s|^{-10}).
\end{align*}
Finally, noticing by~\eqref{eq:Wkleftm} that, for all $i\geq 1$,
\be \label{Wneqk}
\left\| Q_k^i \sum_{k_2\neq k} \Bigl[ W_{k_2}(s) - W_{k_2}(s,y_k(s)) \Bigr] \right\|_{L^\infty} \lesssim |s|^{-2},
\ee
then, by the definition~\eqref{def:rkdk} of $r_k$ and the cancellation
$\langle R_k Q_k^3,\pa_y Q_k \rangle = \eps_k \tilde \mu_k^{-\frac 52} \langle R Q^3,Q' \rangle = 0$
obtained again by parity, we find
\begin{align*}
20 \left\langle (r_k R_k + a_k P_k) Q_k^3 \sum_{k_2\neq k} W_{k_2},\pa_y Q_k \right\rangle
& = 20 \eps_k a_k r_k \tilde\mu_k^{-\frac 12} \langle P_k Q_k^3,\pa_y Q_k \rangle + O(|s|^{-3}) \\
& = 20 \frac {a_k r_k}{\mu_k^3} \langle P Q^3,Q' \rangle + O(|s|^{-3}).
\end{align*}
Therefore, we have obtained
\begin{multline*}
\left| \langle \P^\mathrm{VI},Q_k \rangle + 10 \frac {r_k}{\mu_k^3\tau_k} \langle Q^3 P R,Q' \rangle
+ 10 \frac {a_k}{\mu_k^3\tau_k} \langle Q^3 P^2,Q' \rangle + 10 \frac {a_k^2}{\mu_k^3} \langle Q^3 P^2,Q' \rangle \right. \\
\left. {}+ 20 \frac {a_k r_k}{\mu_k^3} \langle Q^3 R P, Q' \rangle
+ 20 \frac {a_k r_k}{\mu_k^3} \langle P Q^3,Q' \rangle \right| \lesssim |s|^{-\frac 52}.
\end{multline*}

\medskip

\textbf{Estimate of $\P^\mathrm{VII}$.}
Using the definition~\eqref{def:rkdk} of $r_j$, we decompose $\P^\mathrm{VII}$ as
\begin{align*}
\P^\mathrm{VII} &= 5 \pa_y\left[ \sum_j \left( W_j^4 - Q_j^4 \right) \sum_{j_1\neq j} W_{j_1} \right] \\
&\quad + 5 \pa_y\left[ \sum_j Q_j^4 \sum_{j_1\neq j} \Bigl( W_{j_1}(s) - W_{j_1}(s,y_j(s)) \Bigr) \right] \\
&\quad + \pa_y\left[ \left( \sum_j W_j \right)^5 - \sum_j W_j^5 - 5 \sum_j W_j^4 \sum_{j_1\neq j} W_{j_1} \right].
\end{align*}
As before, using~\eqref{eq:WkHm}--\eqref{eq:WkLinfty} and~\eqref{eq:Wkleft0}--\eqref{eq:Wkright}, we have
\[
\left\| \left( W_j^4 - Q_j^4 \right) \sum_{j_1\neq j} W_{j_1} \right\|_{H^3} \lesssim |s|^{-\frac 74} \quad\m{and}\quad
\left\| Q_j^4 \sum_{j_1\neq j} \Bigl[ W_{j_1}(s) - W_{j_1}(s,y_j(s)) \Bigr] \right\|_{H^3} \lesssim |s|^{-2}.
\]
Using also~\eqref{eq:Wkweighted}, we get
\[
\left\| \left( W_j^4 - Q_j^4\right) \sum_{j_1\neq j} W_{j_1} \right\|_{L^\infty_k} \lesssim |s|^{-2}.
\]
Projecting on $Q_k$, proceeding as before, we find
\begin{align*}
5 \left\langle \pa_y \left[ \left( W_k^4 - Q_k^4 \right) \sum_{k_1\neq k} W_{k_1} \right],Q_k \right\rangle
&= -10 \frac {r_k}{\tau_k} \eps_k \tilde\mu_k^{-\frac 12} \langle Q_k^3\tilde P_k,\pa_y Q_k \rangle + O(|s|^{-\frac 52}) \\
&= -10 \frac {r_k}{\mu_k^3 \tau_k} \langle Q^3 P,Q' \rangle + O(|s|^{-\frac 52}).
\end{align*}
Moreover, from~\eqref{apriori:mutilde}, the definition~\eqref{def:rkdk} of $d_k$ and the relation $Q''+Q^5=Q$, we find
\begin{align*}
&5 \left\langle \pa_y \left[ Q_k^4 \sum_{k_1\neq k} \Bigl( W_{k_1}(s) - W_{k_1}(s,y_k(s)) \Bigr) \right],Q_k \right\rangle
= -5 \eps_k d_k \tilde\mu_k^{-\frac 32} \langle (\cdot-y_k) Q_k^4,\pa_y Q_k \rangle + O(|s|^{-3}) \\
&= -5 \frac {d_k}{\tilde\mu_k^3} \int y Q^4(y) Q'(y) \,dy + O(|s|^{-3}) = \frac {d_k}{\tilde\mu_k^3} \int Q^5(y) \,dy + O(|s|^{-3})
= \frac {d_k}{\mu_k^3} \|Q\|_{L^1} + O(|s|^{-3}).
\end{align*}

To estimate the last term in $\P^\mathrm{VII}$, we use the decay properties of $W_j$ and obtain as before
\[
\left\| \left( \sum_j W_j \right)^5 - \sum_j W_j^5 - 5 \sum_j W_j^4 \sum_{j_1\neq j} W_{j_1} \right\|_{H^3} \lesssim |s|^{-2},
\]
and similarly, using also~\eqref{eq:WkLinfty},
\be \label{pourPsiVII}
\left\| \left( \sum_j W_j \right)^5 - \sum_j W_j^5 - 5 \sum_j W_j^4 \sum_{j_1\neq j} W_{j_1}
- 10 \sum_j Q_j^3 \sum_{j_1,j_2\neq j} W_{j_1} W_{j_2} \right\|_{H^3} \lesssim |s|^{-\frac {11}4}.
\ee
But, since $\langle Q_k^3,\pa_y Q_k \rangle = \tilde \mu_k^{-2} \langle Q^3,Q' \rangle = 0$ by parity,
we find, from the definition~\eqref{def:rkdk} of $r_k$,
\begin{align*}
&\left\langle Q_k^3 \sum_{k_1,k_2\neq k} W_{k_1} W_{k_2},\pa_y Q_k \right\rangle \\
&= \sum_{k_1,k_2\neq k} \int Q_k^3(y) \left[ W_{k_1}(s,y_k(s)) + \Bigl( W_{k_1}(s,y) - W_{k_1}(s,y_k(s)) \Bigr) \right] \\
&\qquad\qquad\qquad\qquad\qquad \left[ W_{k_2}(s,y_k(s)) + \Bigl( W_{k_2}(s,y) - W_{k_2}(s,y_k(s)) \Bigr) \right] \pa_y Q_k(y) \,dy \\
&= 2\eps_k \tilde\mu_k^{-\frac 12} r_k \sum_{k'\neq k} \int Q_k^3(y) \Bigl[ W_{k'}(s,y) - W_{k'}(s,y_k(s)) \Bigr] \pa_y Q_k(y) \,dy \\
&\quad + \sum_{k_1,k_2\neq k} Q_k^3(y) \Bigl[ W_{k_1}(s,y) - W_{k_1}(s,y_k(s)) \Bigr]
\Bigl[ W_{k_2}(s,y) - W_{k_2}(s,y_k(s)) \Bigr] \pa_y Q_k(y) \,dy
\end{align*}
and so, from~\eqref{apriori:rd} and~\eqref{Wneqk},
\begin{align*}
\left| \left\langle Q_k^3 \sum_{k_1,k_2\neq k} W_{k_1} W_{k_2},\pa_y Q_k \right\rangle \right|
&\lesssim |r_k| \sum_{k'\neq k} \left\| Q_k^3 \Bigl[ W_{k'}(s) - W_{k'}(s,y_k(s)) \Bigr] \right\|_{L^\infty} \\
&\quad + \sum_{k'\neq k} \left\| Q_k \Bigl[ W_{k'}(s) - W_{k'}(s,y_k(s)) \Bigr] \right\|_{L^\infty}^2 \\
&\lesssim |s|^{-1} |s|^{-2} + |s|^{-4} \lesssim |s|^{-3}.
\end{align*}
Thus, from~\eqref{pourPsiVII}, we finally deduce
\[
\left| \left\langle \left( \sum_j W_j \right)^5 - \sum_j W_j^5
- 5 \sum_j W_j^4 \sum_{j_1\neq j} W_{j_1},\pa_y Q_k \right\rangle \right| \lesssim |s|^{-\frac {11}4}.
\]
Therefore, we have obtained
\begin{gather*}
\|\P^\mathrm{VII}\|_{H^2} \lesssim |s|^{-\frac 74}, \quad
\|\P^\mathrm{VII}\|_{L^\infty_k} \lesssim |s|^{-2}, \\
\left|\langle \P^\mathrm{VII},Q_k \rangle + 10 \frac {r_k}{\mu_k^3 \tau_k} \langle Q^3 P,Q' \rangle
- \frac {d_k}{\mu_k^3} \|Q\|_{L^1} \right| \lesssim |s|^{-\frac 52}.
\end{gather*}

\medskip

\textbf{Estimate of $\P^\mathrm{VIII}$.}
We estimate $\P^\mathrm{VIII}$ as $\P^\mathrm{III}$,
noticing that $\|R_j\|_{H^3} \lesssim 1$, $|r_j| \lesssim |s|^{-1}$ from~\eqref{apriori:rd},
$|\tilde\mu_j - \mu_j| \lesssim |s|^{-1}$ from~\eqref{apriori:mutilde},
$\langle R_k,\pa_y Q_k \rangle = \tilde \mu_k^{-1} \langle R,Q' \rangle = 0$ by parity
and $|\langle R_j,\pa_y Q_k \rangle| \lesssim |s|^{-10}$ for $j\neq k$,
and thus obtain
\[
\|\P^\mathrm{VIII}\|_{H^2} \lesssim |s|^{-2},\quad
|\langle \P^\mathrm{VIII},Q_k \rangle | \lesssim |s|^{-12}.
\]

\medskip

\textbf{Conclusion.}
Gathering the estimates in $H^2$ of $\P^\mathrm{I},\ldots,\P^\mathrm{VIII}$ above, we obtain~\eqref{est:Psi}.
Then, using also the additional estimates in $H^2(y > y_k - |s|^{\frac 14})$ and $L^\infty_k$ when necessary, we get~\eqref{est:Psik}.

Finally, gathering the projections of $\P^\mathrm{I},\ldots,\P^\mathrm{VIII}$ on $Q_k$ yields~\eqref{est:PsiQ}, with
\begin{align*}
\Omega_k &= \left( \frac {r_k}{2\mu_k^3\tau_k} + \frac{a_kr_k}{\mu_k^3} \right)
\left( \langle \Lambda R,Q \rangle - 20 \langle Q^3 P R,Q' \rangle - 20 \langle P Q^3,Q' \rangle \right) \\
&\quad + \left( \frac {a_k}{\mu_k^3\tau_k} + \frac {a_k^2}{\mu_k^3} \right)
\left( \langle \Lambda P,Q \rangle - 10 \langle Q^3 P^2,Q' \rangle \right) + \frac {d_k}{\mu_k^3} \|Q\|_{L^1}.
\end{align*}
But, from~\cite{MMR1} and~\cite{MMR3}, we recall the identities
\[
\langle \Lambda P,Q \rangle - 10 \langle Q^3 P^2,Q' \rangle = \frac 18 \|Q\|_{L^1}^2 \quad\m{and}\quad
\langle \Lambda R,Q \rangle - 20 \langle Q^3 P R,Q' \rangle - 20 \langle P Q^3,Q' \rangle = 0,
\]
and so
\[
\Omega_k = \frac 18 \|Q\|_{L^1}^2 \left( \frac {a_k}{\mu_k^3\tau_k}
+ \frac {a_k^2}{\mu_k^3} \right) + \frac {d_k}{\mu_k^3} \|Q\|_{L^1},
\]
which concludes the proof of Lemma~\ref{lem:V}.
\end{proof}

Next, we give precise asymptotics on $\dot r_k$ and $r_k$.

\begin{lemma} \label{lem:rk}
\begin{enumerate}[label=\emph{(\roman*)}]
\item \emph{Asymptotics of $\dot r_k$:}
for all $1\leq k\leq K$, for all $s\in I$,
\be \label{rkdk}
\left| \frac {d_k}{\mu_k^3} - \left( \dot r_k + \frac {r_k}{4\mu_k^3 \tau_k} + \frac {a_k r_k}{2\mu_k^3} \right) \right|
\lesssim |s|^{-3} + |s|^{-1} |\vec m_k| + |s|^{-1} \sum_{j<k} |\vec m_j^0| + |s|^{-10} \sum_{j>k} |\vec m_j^0|,
\ee
and, in particular,
\be \label{e:rk}
|\dot r_k| \lesssim |s|^{-2} + |s|^{-1} \sum_j |\vec m_j|.
\ee
\item \emph{Asymptotics of $r_k$:}
for all $1\leq k\leq K$, for all $s\in I$,
\be \label{eq:rkasymp}
\begin{aligned}
\left| r_k(s) - \frac {\|Q\|_{L^1}}{4s} \ell_k^3\theta_k \right. & \left. \left( 1 + \frac 12 \bar\mu_k - \frac 32 \bar y_k \right) \right| \\
&\lesssim |s|^{-1} \left( |s|^{-\frac 1{42}} + |\bar\mu_k|^2 + |\bar y_k|^2
+ \sum_{j<k} (|\bar\mu_j| + |\bar\tau_j| + |\bar y_j|) \right),
\end{aligned}
\ee
where the constant $\theta_k\in\R$ is defined by
\be \label{def:thetak}
\theta_k = \sum_{j<k} \eps_k \eps_j \sqrt{\frac {\ell_k}{\ell_j}}.
\ee
\end{enumerate}
\end{lemma}

\begin{proof}
(i) By the definition~\eqref{def:rkdk} of $r_k$, we have
\[
\dot r_k = \frac 12 \frac {\dot{\tilde\mu}_k}{\tilde\mu_k} r_k + \eps_k\tilde\mu_k^{\frac 12} \sum_{j\neq k} \pa_s W_j(s,y_k)
+ \eps_k \tilde\mu_k^{\frac 12} \dot y_k \sum_{j\neq k} \pa_y W_j(s,y_k).
\]

First, from~\eqref{apriori:mutilde} and the definition~\eqref{def:mod} of $m_{k,1}$, we find
\[
\frac {\dot{\tilde\mu}_k}{\tilde\mu_k} = \frac {\dot{\mu}_k}{\mu_k} + O(|s|^{-2})
= -\frac 1{2\mu_k^3 \tau_k} + \frac 1{2s} - \frac {a_k}{\mu_k^3} + O(|s|^{-2}) + O(|\vec m_k|).
\]
Thus, since $|r_k| \lesssim |s|^{-1}$, by~\eqref{apriori:rd},
\[
\frac 12 \frac {\dot{\tilde\mu}_k}{\tilde\mu_k} r_k
= -\frac {r_k}{4\mu_k^3 \tau_k} + \frac {r_k}{4s} - \frac {a_k r_k}{2\mu_k^3} + O(|s|^{-3}) + O(|s|^{-1} |\vec m_k|).
\]

Second, by~\eqref{eqWk}, for all $j\neq k$,
\[
\pa_s W_j = - \frac 1{2s} \Lambda W_j - \pa_y\left(\pa_{yy} W_j + W_j^5 \right) + \vec m_j^0 \cdot \vec \M_j W_j.
\]
By~\eqref{apriori:mutauy}--\eqref{apriori:z} and~\eqref{eq:Wkleft0}--\eqref{eq:Wkleftm}, we estimate, for $j<k$,
\[
|\pa_{yyy} W_j(s,y_k)| \lesssim |s|^{-4},\quad
|\pa_y(W_j^5)(s,y_k)| \lesssim |s|^{-6},\quad
|\vec \M_j W_j(s,y_k)| \lesssim |s|^{-1},
\]
and, for $j>k$, using~\eqref{eq:Wkright},
\[
|y_k \pa_y W_j(s,y_k)| + |W_j(s,y_k)| + |\pa_{yyy} W_j(s,y_k)|
+ |\pa_y(W_j^5)(s,y_k)| + |\vec \M_j W_j(s,y_k)| \lesssim |s|^{-10}.
\]
Thus, for $j<k$,
\[
\pa_s W_j (s,y_k) = - \frac 1{4s} W_j(s,y_k) - \frac {y_k}{2s} \pa_y W_j(s,y_k)
+ O(|s|^{-4}) + O(|s|^{-1} |\vec m_j^0|),
\]
and, for $j>k$,
\[
|\pa_s W_j (s,y_k)| \lesssim |s|^{-10} + |s|^{-10} |\vec m_j^0|.
\]
Therefore, from~\eqref{def:rkdk},
\[
\eps_k \tilde\mu_k^{\frac 12} \sum_{j\neq k} \pa_s W_j(s,y_k)
= -\frac {r_k}{4s} - \frac {y_k}{2s} \tilde\mu_k^{-1} d_k
+ O(|s|^{-4}) + O\left(|s|^{-1} \sum_{j<k} |\vec m_j^0| \right)
+ O\left( |s|^{-10} \sum_{j>k} |\vec m_j^0| \right).
\]

Third, from~\eqref{apriori:mutilde}, \eqref{apriori:rd} and~\eqref{def:mod},
\begin{align*}
\eps_k \tilde\mu_k^{\frac 12} \dot y_k \sum_{j\neq k} \pa_y W_j(s,y_k) = \tilde\mu_k^{-1} \dot y_k d_k
&= \tilde\mu_k^{-1} \left( \frac {y_k}{2s} + \frac 1{\mu_k^2} \right) d_k + O(|s|^{-3}) + O(|s|^{-2} |\vec m_k|) \\
&= \frac {y_k}{2s} \tilde\mu_k^{-1} d_k + \frac {d_k}{\mu_k^3} + O(|s|^{-3}) + O(|s|^{-2} |\vec m_k|).
\end{align*}

Gathering the above estimates, we find~\eqref{rkdk}.
Finally, using the \emph{a priori} estimates~\eqref{apriori:mutauy} and~\eqref{apriori:a}--\eqref{apriori:rd},
we directly deduce~\eqref{e:rk} from~\eqref{rkdk}.

(ii) From the definition~\eqref{def:rkdk} of $r_k$, the \emph{a priori} estimates~\eqref{apriori:mutauy}--\eqref{apriori:z},
and estimates~\eqref{apriori:mutilde}, \eqref{eq:Wkleft0} and~\eqref{eq:Wkright}, we find
\begin{align*}
r_k &= -\frac {\|Q\|_{L^1}}2 \eps_k \tilde\mu_k^{\frac 12} \sum_{j<k} \eps_j \mu_j \sqrt{-2\tau_j} |y_k - z_j|^{-\frac 32}
+ O\left(|s|^{\frac 12 + \frac 1{42}} \sum_{j<k} |y_k - z_j|^{-\frac 32 - \frac 1{21}} \right) + O(|s|^{-10}) \\
&= -\frac {\|Q\|_{L^1}}2 \eps_k \mu_k^{\frac 12} |y_k|^{-\frac 32} \sum_{j<k} \eps_j \mu_j \sqrt{-2\tau_j}
+ O\left(|s|^{-1} \sum_{j<k} |\bar z_j| \right) + O(|s|^{-1 - \frac 1{42}}) \\
&= -\frac {\|Q\|_{L^1}}2 \eps_k \mu_k^{\frac 12} |y_k|^{-\frac 32} \sqrt{-2s} \sum_{j<k} \eps_j \ell_j^{-\frac 12}
+ O\left(|s|^{-1} \sum_{j<k} (|\bar z_j| + |\bar\mu_j| + |\bar\tau_j|) \right) + O(|s|^{-1 - \frac 1{42}}).
\end{align*}
From~\eqref{apriori:mutauy}, we have $\mu_k = \ell_k (1 + \bar\mu_k)$ and $|y_k| = -2s \ell_k^{-2} (1 + \bar y_k)$,
thus a Taylor expansion gives
\[
\mu_k^{\frac 12} = \ell_k^{\frac 12} \left[ 1 + \frac 12 \bar\mu_k + O(|\bar\mu_k|^2) \right] \quad\m{and}\quad
|y_k|^{-\frac 32} = (-2s)^{-\frac 32} \ell_k^3 \left[ 1 - \frac 32 \bar y_k + O(|\bar y_k|^2) \right],
\]
which leads to~\eqref{eq:rkasymp} together with the above estimate of $r_k$ and~\eqref{apriori:z}.
\end{proof}

Finally, we give sharp asymptotics on the mass and the energy of the approximate rescaled multi-soliton $\V$.

\begin{lemma} \label{lem:Vmassenergy}
\begin{enumerate}[label=\emph{(\roman*)}]
\item \emph{Mass of $\V$:} for all $s\in I$,
\be \label{eq:Vmass}
\left| \|\V(s)\|_{L^2}^2 - K \|Q\|_{L^2}^2 \right|
\lesssim \sum_k |a_k| + |s| \sum_k a_k^2
+ |s|^{-1} \sum_k (|\bar\mu_k| + |\bar\tau_k| + |\bar y_k|) + |s|^{-1 - \frac 1{42}}.
\ee
\item \emph{Energy of $\V$:}
for all $1\leq k\leq K$, for all $s\in I$, let the local energy $e_k$ be defined by
\be \label{def:ek}
e_k(s) = \frac s{\mu_k^2(s)} \left[ a_k(s) + \frac 1{2\tau_k(s)} + \frac {4r_k(s)}{\|Q\|_{L^1}} \right].
\ee
Then, for all $s\in I$,
\be \label{eq:Venergy}
\left| E(\V(s)) + \frac {\|Q\|_{L^1}^2}{16s} \sum_k e_k(s) \right| \lesssim |s|^{-\frac 32}
\ee
and, in particular,
\be \label{eq:Venergybis}
\left| E(\V(s)) + \frac {\|Q\|_{L^1}^2}{32s} \sum_k \ell_k (1 + 2\theta_k) \right| \lesssim \sum_k |a_k|
+ |s|^{-1} \sum_k (|\bar\mu_k| + |\bar\tau_k| + |\bar y_k|) + |s|^{-1 - \frac 1{42}}.
\ee
\end{enumerate}
\end{lemma}

\begin{proof}
(i) First, expanding $\V = \sum_k (W_k + r_k R_k + a_k P_k)$,
using~\eqref{apriori:rd} and $\|P_k\|_{L^2} \lesssim |s|^{\frac 12}$ from~\eqref{eq:Pknorms}, we find
\[
\int \V^2 = \sum_k \sum_j \int \Bigl( W_k W_j + 2 r_j W_k R_j + 2 a_j W_k P_j \Bigr)
+ O\left(|s| \sum_k a_k^2 \right) + O(|s|^{-2}).
\]
Thus, distinguishing the cases $j=k$ and $j\neq k$, expanding $W_k = Q_k + (W_k - Q_k)$
and using estimates~\eqref{eq:WkHm}, \eqref{eq:WkLinfty} and~\eqref{eq:Wkweighted},
\[
\int \V^2 = \sum_k \int \Bigl( W_k^2 + 2 r_k Q_k R_k \Bigr) + \sum_k \sum_{j\neq k} \int W_k W_j
+ O\left( \sum_k |a_k| \right) + O\left(|s| \sum_k a_k^2 \right) + O(|s|^{-1 - \frac 34}).
\]
Now note that, by the definitions of $W_k$ and $\St$, and by the value $\|S(t)\|_{L^2} = \|Q\|_{L^2}$, we have
\[
\|W_k\|_{L^2} = \|\St(\tau_k)\|_{L^2} = \left\| S\left(\frac 1{\sqrt{-2\tau_k}} \right) \right\|_{L^2} = \|Q\|_{L^2}.
\]
Thus, using also $\langle Q_k,R_k \rangle = \langle Q,R \rangle = -\frac 34 \|Q\|_{L^1}$ from~\eqref{def:R},
we obtain
\begin{align*}
\|\V\|_{L^2}^2 &= K\|Q\|_{L^2}^2 + \sum_k \left( 2 \sum_{j<k} \int W_k W_j - \frac 32 \|Q\|_{L^1} r_k \right) \\
&\quad + O\left( \sum_k |a_k| \right) + O\left(|s| \sum_k a_k^2 \right) + O(|s|^{-1-\frac 34}).
\end{align*}
Finally, we claim that, for all $1\leq k\leq K$,
\be \label{eq:WkWj}
\left| \sum_{j<k} \int W_k W_j - \frac {3\|Q\|_{L^1}^2}{16s} \ell_k^3 \theta_k \right|
\lesssim |s|^{-1} \sum_{j\leq k} (|\bar\mu_j| + |\bar\tau_j| + |\bar y_j|) + |s|^{-1 - \frac 1{42}},
\ee
which gives~\eqref{eq:Vmass}, together with~\eqref{eq:rkasymp} and the calculation above.

To prove~\eqref{eq:WkWj}, we estimate $\langle W_k,W_j \rangle$ with $j<k$ on the three regions
$y < y_k - |s|^{\frac 34}$, $|y - y_k| \leq |s|^{\frac 34}$ and $y > y_k + |s|^{\frac 34}$.
First, using~\eqref{eq:Wkleft0} then~\eqref{apriori:mutauy}--\eqref{apriori:z}, we obtain
\begin{align*}
&\int_{y < y_k - |s|^{\frac 34}} W_k W_j = \frac 12 \|Q\|_{L^1}^2 \eps_k \eps_j \mu_k \mu_j
\sqrt{\tau_k \tau_j} \int_{y < y_k - |s|^{\frac 34}} |y - z_k|^{-\frac 32} |y - z_j|^{-\frac 32} \,dy + O(|s|^{-1 - \frac 1{42}}) \\
&= \frac 14 \|Q\|_{L^1}^2 \eps_k \eps_j \mu_k \mu_j \sqrt{\tau_k \tau_j} |y_k|^{-2}
+ O\left(|s|^{-1} (|\bar z_j| + |\bar z_k|) \right) + O(|s|^{-1 - \frac 1{42}}) \\
&= -\frac {\|Q\|_{L^1}^2}{16s} \eps_k \eps_j \ell_k^{\frac 72} \ell_j^{-\frac 12}
+ O\left(|s|^{-1} (|\bar \mu_j| + |\bar \mu_k| + |\bar \tau_j| + |\bar \tau_k| + |\bar y_j| + |\bar y_k|) \right) + O(|s|^{-1 - \frac 1{42}}).
\end{align*}
Second, we decompose the next term as
\begin{align*}
\int_{|y - y_k| \leq |s|^{\frac 34}} W_k W_j &= \int_{|y - y_k| \leq |s|^{\frac 34}} (W_k - Q_k) W_j
+ \int_{|y - y_k| \leq |s|^{\frac 34}} Q_k (W_j - W_j(y_k)) \\
&\quad + W_j(y_k) \int_{|y - y_k| \leq |s|^{\frac 34}} Q_k = A_{j,k} + B_{j,k} + C_{j,k}.
\end{align*}
By the Cauchy--Schwarz inequality, \eqref{eq:WkHm} and again~\eqref{eq:Wkleft0}, we find
\[
|A_{j,k}| \lesssim |s|^{-1} \int_{|y - y_k| \leq |s|^{\frac 34}} |W_k - Q_k|
\lesssim |s|^{-1} |s|^{\frac 38} \|W_k - Q_k\|_{L^2} \lesssim |s|^{-1 - \frac 18}.
\]
Next, similarly to~\eqref{Wneqk}, we obtain $|B_{j,k}|\lesssim |s|^{-2}$.
To estimate $C_{j,k}$, we use the exponential decay of $Q$ and~\eqref{apriori:mutilde} to get
\begin{align*}
\int_{|y - y_k| \leq |s|^{\frac 34}} Q_k(y) \,dy &= \eps_k \tilde\mu_k^{\frac 12}
\int_{|z| \leq \tilde\mu_k |s|^{\frac 34}} Q(z) \,dz \\
&= \eps_k \tilde\mu_k^{\frac 12} \|Q\|_{L^1} + O(|s|^{-10})
= \eps_k \mu_k^{\frac 12} \|Q\|_{L^1} + O(|s|^{-1}),
\end{align*}
and again~\eqref{eq:Wkleft0} then \eqref{apriori:mutauy}--\eqref{apriori:z} to obtain as before
\begin{align*}
&W_j(y_k) = -\frac {\eps_j}2 \|Q\|_{L^1} \mu_j \sqrt{-2\tau_j} |y_k - z_j|^{-\frac 32} + O(|s|^{-1 - \frac 1{42}}) \\
&= -\frac {\eps_j}2 \|Q\|_{L^1} \ell_j \sqrt{-2s\ell_j^{-3}} (-2s \ell_k^{-2})^{-\frac 32}
+ O\left(|s|^{-1} (|\bar z_j| + |\bar \mu_j| + |\bar \tau_j| + |\bar y_k|) \right) + O(|s|^{-1 - \frac 1{42}}) \\
&= \frac {\eps_j}{4s} \|Q\|_{L^1} \ell_k^3 \ell_j^{-\frac 12}
+ O\left(|s|^{-1} (|\bar \mu_j| + |\bar \tau_j| + |\bar y_j| + |\bar y_k|) \right) + O(|s|^{-1 - \frac 1{42}}).
\end{align*}
Thus,
\[
C_{j,k} = \frac {\|Q\|_{L^1}^2}{4s} \eps_k \eps_j \ell_k^{\frac 72} \ell_j^{-\frac 12}
+ O\left(|s|^{-1} (|\bar \mu_j| + |\bar \mu_k| + |\bar \tau_j| + |\bar y_j| + |\bar y_k|) \right) + O(|s|^{-1 - \frac 1{42}}).
\]
Finally, from~\eqref{eq:Wkright}, we observe that
\[
\left| \int_{y > y_k + |s|^{\frac 34}} W_k W_j \right| \lesssim |s|^{-10}.
\]
Thus, gathering the previous estimates, we have obtained
\[
\left| \int W_k W_j - \frac {3\|Q\|_{L^1}^2}{16s} \eps_k \eps_j \ell_k^{\frac 72} \ell_j^{-\frac 12} \right|
\lesssim |s|^{-1} (|\bar \mu_j| + |\bar \mu_k| + |\bar \tau_j| + |\bar \tau_k| + |\bar y_j| + |\bar y_k|) + |s|^{-1 - \frac 1{42}}
\]
and so, summing over $j<k$ and using the definition~\eqref{def:thetak} of $\theta_k$,
we obtain~\eqref{eq:WkWj}, which concludes the proof of~\eqref{eq:Vmass}.

(ii) To compute the gradient term in the energy of $\V$, we proceed as in (i).
Indeed, ex\-pan\-ding first $\V = \sum_k (W_k + r_k R_k + a_k P_k)$,
using~\eqref{apriori:a}--\eqref{apriori:rd} and $\|\pa_y P_k\|_{L^2} \lesssim 1$ from~\eqref{eq:Pknorms},
we find similarly
\[
\frac 12 \int (\pa_y\V)^2 = \sum_k \sum_j \int \left( \frac 12 \pa_y W_k \pa_y W_j
+ r_j \pa_y W_k \pa_y R_j + a_j \pa_y W_k \pa_y P_j \right) + O(|s|^{-2}).
\]
Now note that, because of the stronger decay~\eqref{eq:Wkleftm} on the left of $\pa_y W_k$
than the decay~\eqref{eq:Wkleft0} on the left of $W_k$, we may estimate simply, for $j<k$,
\[
\left| \int \pa_y W_k \pa_y W_j \right| \lesssim \|\pa_y W_k\|_{L^2} \|\pa_y W_j\|_{L^2(y < y_k + |s|^{\frac 34})}
+ \|\pa_y W_j\|_{L^2} \|\pa_y W_k\|_{L^2(y > y_k + |s|^{\frac 34})} \lesssim |s|^{-\frac 32}.
\]
Thus, distinguishing again the cases $j=k$ and $j\neq k$, expanding $W_k = Q_k + (W_k - Q_k)$,
integrating by parts and using estimates~\eqref{eq:WkHm} and~\eqref{eq:Wkweighted}, we obtain
\[
\frac 12 \int (\pa_y\V)^2 = \sum_k \int \left( \frac 12 (\pa_y W_k)^2
- r_k \pa_{yy} Q_k R_k - a_k \pa_{yy} Q_k P_k \right) + O(|s|^{-\frac 32}).
\]

To compute the nonlinear term in the energy of $\V$, we follow the estimates
of $\P^\mathrm{VI}$ and $\P^\mathrm{VII}$ in the proof of Lemma~\ref{lem:V}.
Indeed, expanding $\V = \W + (\V - \W)$, we first decompose this term as
\[
\frac 16 \int \V^6 = \frac 16 \sum_{i=0}^6 \binom 6i \int \W^{6-i} (\V - \W)^i
= \frac 16 \int \W^6 + \int \W^5 (\V - \W) + O(|s|^{-2}),
\]
since, for all $2\leq i\leq 6$, using $\|\V - \W\|_{L^\infty} \lesssim |s|^{-1}$
and $\|\V - \W\|_{L^2} \lesssim |s|^{-\frac 12}$,
\[
\left| \int \W^{6-i} (\V - \W)^i \right| \lesssim |s|^{-2}.
\]
Now we expand $\W = \sum_k W_k$ to decompose the nonlinear term as
\begin{align*}
\frac 16 \int \V^6 &= \frac 16 \sum_k \int W_k^6 + \sum_k \sum_{j\neq k} \int W_k^5 W_j
+ \frac 16 \int \left[ \left( \sum_k W_k \right)^6 - \sum_k W_k^6 - 6 \sum_k W_k^5 \sum_{j\neq k} W_j \right] \\
&\quad + \int \left( \sum_k W_k^5 \right) (\V - \W)
+ \int \left[ \left( \sum_k W_k \right)^5 - \sum_k W_k^5 \right] (\V - \W) + O(|s|^{-2}).
\end{align*}
As before, using~\eqref{eq:Wkleft0} and~\eqref{eq:Wkright}, we observe that
\begin{align*}
&\int \left| \left( \sum_k W_k \right)^6 - \sum_k W_k^6 - 6 \sum_k W_k^5 \sum_{j\neq k} W_j \right|
\lesssim \sum_{\substack{1\leq k_6\leq \cdots \leq k_1\leq K \\ k_5 < k_1}} \int \prod_{j=1}^6 |W_{k_j}| \\
&\lesssim \sum_{\substack{1\leq k_6\leq k_5 \leq k_2 \leq k_1\leq K \\ k_5 < k_1}}
\int_{y < y_{k_1} + |s|^{\frac 34}} |W_{k_1}| |W_{k_2}| |W_{k_5}| |W_{k_6}|
+ \sum_{k_1} \int_{y > y_{k_1} + |s|^{\frac 34}} |W_{k_1}| \\
& \lesssim |s|^{-2} \sum_{1\leq k_2\leq k_1\leq K} \|W_{k_1}\|_{L^2} \|W_{k_2}\|_{L^2} + |s|^{-10} \lesssim |s|^{-2},
\end{align*}
and similarly
\[
\int \left| \left( \sum_k W_k \right)^5 - \sum_k W_k^5 \right| \lesssim |s|^{-1}.
\]
Thus, using also $\|\V - \W\|_{L^\infty} \lesssim |s|^{-1}$, we obtain
\[
\frac 16 \int \V^6 = \frac 16 \sum_k \int W_k^6 + \sum_k \sum_{j\neq k} \int W_k^5 W_j
+ \int \left( \sum_k W_k^5 \right) (\V - \W) + O(|s|^{-2}).
\]
Next, expanding $\V - \W = \sum_k (r_k R_k + a_k P_k)$ and $W_k = Q_k + (W_k - Q_k)$,
using~\eqref{eq:WkHm}--\eqref{eq:WkLinfty} and~\eqref{eq:Wkweighted}, we obtain
\[
\int \left( \sum_k W_k^5 \right) (\V - \W) = \sum_k \int Q_k^5 (r_k R_k + a_k P_k) + O(|s|^{-2}).
\]

Gathering the above estimates and using the identity $\pa_{yy} Q_k + Q_k^5 = \tilde \mu_k^{-2} Q_k$,
we have obtained
\[
E(\V) = \sum_k E(W_k) - \sum_k \int \tilde \mu_k^{-2} (r_k Q_k R_k + a_k Q_k P_k)
- \sum_k \sum_{j\neq k} \int W_k^5 W_j + O(|s|^{-\frac 32}).
\]
Now note that, by the definitions of $W_k$ and $\St$,
and by the value $E(S(t))=\frac {\|Q\|_{L^1}^2}{16}$, we have
\[
E(W_k) = \mu_k^{-2} E(\St(\tau_k)) = \frac {\mu_k^{-2}}{-2\tau_k}
E \left( S\left( \frac 1{\sqrt{-2\tau_k}} \right) \right)
= - \frac {\|Q\|_{L^1}^2}{16} \frac 1{2\mu_k^2\tau_k}.
\]
Thus, using also~\eqref{apriori:mutilde} and the identities
$\langle Q_k,R_k \rangle = \langle Q,R \rangle = -\frac 34 \|Q\|_{L^1}$
from~\eqref{def:R} and
\[
\langle Q_k,P_k \rangle = \langle Q,P \rangle + O(|s|^{-10}) = \frac {\|Q\|_{L^1}^2}{16} + O(|s|^{-10})
\]
from~\eqref{PQ}, we find
\[
E(\V) = - \frac {\|Q\|_{L^1}^2}{16} \sum_k \frac 1{2\mu_k^2\tau_k}
+ \frac {3\|Q\|_{L^1}}4 \sum_k \frac {r_k}{\mu_k^2} - \frac {\|Q\|_{L^1}^2}{16} \sum_k \frac {a_k}{\mu_k^2}
- \sum_k \sum_{j\neq k} \int W_k^5 W_j + O(|s|^{-\frac 32}).
\]
Finally, we claim that, for all $1\leq k\leq K$,
\be \label{eq:Wk5Wj}
\left| \sum_{j\neq k} \int W_k^5 W_j - \frac {r_k}{\mu_k^2} \|Q\|_{L^1} \right| \lesssim |s|^{-\frac 74},
\ee
so that, from the definition~\eqref{def:ek} of $e_k$, we obtain
\[
E(\V) = - \frac {\|Q\|_{L^1}^2}{16} \sum_k \frac 1{\mu_k^2} \left( \frac 1{2\tau_k}
+ \frac {4r_k}{\|Q\|_{L^1}} + a_k \right) + O(|s|^{-\frac 32})
= - \frac {\|Q\|_{L^1}^2}{16s} \sum_k e_k + O(|s|^{-\frac 32}).
\]
Thus~\eqref{eq:Venergy} is proved, and~\eqref{eq:Venergybis} is obtained as a direct consequence
by inserting~\eqref{apriori:mutauy} and~\eqref{eq:rkasymp} into~\eqref{eq:Venergy}.

We conclude by proving the claim~\eqref{eq:Wk5Wj}.
First note that, by~\eqref{eq:Wkleft0} and~\eqref{eq:Wkright}, for $j\neq k$,
\[
\left| \int_{y > y_k + |s|^{\frac 34}} W_k^5 W_j \right| \lesssim |s|^{-10} \quad\m{ and }\quad
\left| \int_{y < y_k - |s|^{\frac 34}} W_k^5 W_j \right| \lesssim |s|^{-4},
\]
since $|W_k(y)|^4 \lesssim |s|^{-4}$ for $y < y_k - |s|^{\frac 34}$.
Next, for $j\neq k$ and $|y - y_k| \leq |s|^{\frac 34}$,
we have $|W_j(y)| \lesssim |s|^{-1}$ from~\eqref{eq:Wkleft0} and~\eqref{eq:Wkright} again,
and thus, using also~\eqref{eq:WkLinfty},
\[
\left| \int_{|y - y_k| \leq |s|^{\frac 34}} \left( W_k^5 - Q_k^5 \right) W_j \right| \lesssim |s|^{-\frac 74}.
\]
Moreover, by~\eqref{apriori:mutilde}, \eqref{def:rkdk}, \eqref{apriori:rd} and the exponential decay of $Q$,
we get similarly as before
\begin{align*}
\sum_{j\neq k} \int_{|y - y_k| \leq |s|^{\frac 34}} Q_k^5 W_j
&= \sum_{j\neq k} W_j(y_k) \int_{|y - y_k| \leq |s|^{\frac 34}} Q_k^5(y) \,dy + O(|s|^{-2}) \\
&= r_k \tilde \mu_k^{-2} \int_{|z| \leq \tilde \mu_k |s|^{\frac 34}} Q^5(z) \,dz + O(|s|^{-2}) \\
&= r_k \tilde \mu_k^{-2} \int Q^5 + O(|s|^{-2}) = r_k \mu_k^{-2} \|Q\|_{L^1} + O(|s|^{-2}).
\end{align*}
Combining these estimates, we obtain~\eqref{eq:Wk5Wj},
which concludes the proof of Lemma~\ref{lem:Vmassenergy}.
\end{proof}

\subsection{Modulation around the approximate multi-soliton}

We want to construct solutions $v$ of~\eqref{rescaledkdv} of the form
\be \label{apriori:eps}
v(s,y) = \V(s,y) + \e(s,y), \quad\m{with}\quad \|\e(s)\|_{H^1} \leq |s|^{-\frac 12}.
\ee
From the definition of $\E_\V$ in Lemma~\ref{lem:V}, the equation of $\e$ reads
\be \label{eq:eps}
\pa_s \e + \frac 1{2s} \Lambda \e + \pa_y\left(\pa_{yy} \e + \left( \V + \e \right)^5 - \V^5 \right) + \E_\V = 0.
\ee

First, for a suitable solution $v$ of~\eqref{rescaledkdv},
we construct in the next lemma a time-dependent parameter vector
\[
\Gab = \left( \tau_1,\mu_1,y_1,a_1,\ldots,\tau_K,\mu_K,y_K,a_K \right)^\mathsf{T}
\in \left( (-\infty,0)\times(0,+\infty)\times\R\times\R \right)^K
\]
to be inserted in the definition of $\V = \V[\Gab]$ in Section~\ref{sec:V},
so that $\e$ defined by~\eqref{apriori:eps} satisfies the orthogonality conditions
\be \label{ortho}
\frac d{d s} \langle \e,\Lambda_k Q_k \rangle = \frac d{d s} \langle \e, (\cdot - y_k)\Lambda_k Q_k \rangle
= \frac d{d s} \langle \e, Q_k \rangle = 0.
\ee
As a consequence of~\eqref{eq:eps} and~\eqref{ortho},
we also give estimates on the time derivatives of the parameters in $\Gab$
in Lemmas~\ref{lem:ortho} and~\ref{lem:orthobis}.

\begin{lemma}[Decomposition of the solution] \label{lem:syst_diff}
There exist $\alpha>0$ small and $s_1<0$ with $|s_1|$ large such that the following hold.
Let $v$ be an $H^1$ solution of~\eqref{rescaledkdv} defined in a neighborhood of $s^{in}<0$ such that $s^{in} < 2s_1$.
Assume that $v(s^{in})$ and $\Gab^{in}$ satisfy, for all $1\leq k\leq K$,
\[
\left\{
\begin{gathered}
\left| \frac {\mu_k^{in}}{\ell_k} - 1 \right| \leq \frac \alpha2,\quad
\left| \frac {\tau_k^{in}}{s^{in}\ell_k^{-3}} - 1 \right| \leq \frac \alpha 2,\quad
\left| \frac {y_k^{in}}{2s^{in}\ell_k^{-2}} - 1 \right| \leq \frac \alpha 2, \\
|a_k^{in}| \leq \frac 12 |s^{in}|^{-1},\quad \left\| v(s^{in}) - \V[\Gab^{in}] \right\|_{H^1} \leq \frac 12 |s^{in}|^{-\frac 12}.
\end{gathered}
\right.
\]
Then there exist $0 < \bar s < |s_1|$ and a unique $\C^1$ function
$\Gab$ defined on $[s^{in},s^{in} + \bar s]$ and taking values into
$\left( (-\infty,0)\times(0,+\infty)\times\R\times\R \right)^K$
such that $\Gab(s^{in}) = \Gab^{in}$ and, on $[s^{in},s^{in} + \bar s]$,
$\Gab$ and $\e = v - \V[\Gab]$ satisfy~\eqref{def:tauk} and~\eqref{ortho}.
Moreover, $\Gab$ and $\e$ also satisfy~\eqref{apriori:mutauy},
\eqref{apriori:a} and~\eqref{apriori:eps}.
\end{lemma}

\begin{remark} \label{rem:system}
The strategy of the proof of Lemma~\ref{lem:syst_diff} below is to write the relation~\eqref{def:tauk}
and the orthogonality conditions~\eqref{ortho} as a non-autonomous differential system $\dot \Gab = \A(s,\Gab)$
satisfied by $\Gab$ and to apply the Cauchy--Lipschitz theorem.
More precisely, we prove that the function $\A$ is continuous and Lipschitz in $\Gab$ on $[s^{in},s^{in} + \bar s] \times \D_{\Gab}$,
where $\bar s > 0$ is small enough and $\D_{\Gab}$ is the compact set of $\Gab$ satisfying~\eqref{apriori:mutauy} and~\eqref{apriori:a}.
\end{remark}

\begin{proof}[Proof of Lemma~\ref{lem:syst_diff}]
To start, we observe that the relation~\eqref{def:tauk} may be rewritten as
\[
L^{\tau_k} \dot\Gab = F^{\tau_k} \quad\m{with}\quad
L^{\tau_k} = (0,\ldots,0,1,0,\ldots,0) \quad\m{and}\quad
F^{\tau_k} = \mu_k^{-3},
\]
where the nonzero coefficient in $L^{\tau_k}$ is located at the position $4(k-1)+1$.
Note that $F^{\tau_k}$ is locally Lipschitz in $\Gab$ from~\eqref{apriori:mutauy}.

Next, we compute $\frac d{d s} \langle \e, \Lambda_k Q_k \rangle
= \langle \pa_s\e, \Lambda_k Q_k \rangle + \langle \e, \pa_s (\Lambda_k Q_k) \rangle$.
Note that $Q_k$ satisfies
\be \label{relationQk}
\pa_s Q_k = -\frac {\dot{\tilde\mu}_k}{\tilde\mu_k} \Lambda_k Q_k - \dot y_k \pa_y Q_k,
\ee
and in particular
$\ds \pa_s (\Lambda_k Q_k) = -\frac {\dot{\tilde\mu}_k}{\tilde \mu_k}
\Lambda^2_k Q_k - \dot y_k \pa_y(\Lambda_k Q_k)$.
Thus, using the equation~\eqref{eq:eps} of $\e$ and the expression~\eqref{decompEV} of $\E_\V$, we find
\be \label{ortho1}
\begin{aligned}
\frac d{d s} \langle \e, \Lambda_k Q_k \rangle
& = - \frac 1{2s} \langle \Lambda \e,\Lambda_k Q_k \rangle
- \left\langle \pa_y\left( \pa_{yy} \e + \left( \V + \e \right)^5 - \V^5 \right),\Lambda_k Q_k \right\rangle - \langle \P,\Lambda_k Q_k \rangle \\
& \quad - \sum_j \langle \vec m_j \cdot \vec \M_j V_j,\Lambda_k Q_k \rangle
- \sum_j \dot r_j \langle R_j,\Lambda_k Q_k \rangle - \sum_j \dot a_j \langle P_j,\Lambda_k Q_k \rangle \\
& \quad - \frac {\dot{\tilde \mu}_k}{\tilde \mu_k} \langle \e,\Lambda_k^2 Q_k \rangle - \dot y_k \langle \e,\pa_y (\Lambda_k Q_k) \rangle.
\end{aligned}
\ee

The first term in the right-hand side of~\eqref{ortho1} rewrites
\[
- \frac 1{2s} \langle \Lambda \e,\Lambda_k Q_k \rangle
= \frac 1{2s} \langle v,\Lambda \Lambda_k Q_k \rangle
- \frac 1{2s} \langle \V,\Lambda \Lambda_k Q_k \rangle.
\]
Note that $v$ is continuous in $L^2$ as a function of $s$, and $\Lambda \Lambda_k Q_k[\Gab]$
is locally Lipschitz in $L^2$ as a function of $\Gab$,
so $\frac 1{2s} \langle v,\Lambda \Lambda_k Q_k \rangle$ is continuous in $s$ and locally Lipschitz in $\Gab$.
For the second term, we only have to check that $\V[\Gab] = \sum_j V_j[\Gab]$ is locally Lipschitz in $L^2$ as a function of $\Gab$.
For the regularity in $\tau_j$ of $W_j$, it suffices to use the fact that $S$
satisfies~\eqref{kdv} and the regularity of $S$ from Theorem~\ref{thm:maintimeold}.
The regularity in $\mu_j$ and $y_j$ of $W_j$ follows from the decay property of the derivative of $W_j$ from Lemma~\ref{lem:Wk}.
The other terms in $V_j$, \emph{i.e.}~$r_j R_j$ and $a_j P_j$, are clearly locally Lipschitz in $L^2$ as functions of $\Gab$
(again the regularity of $r_j$ follows from the properties of $S$ in Theorem~\ref{thm:maintimeold}).

The second term in the right-hand side of~\eqref{ortho1} rewrites
\[
- \left\langle \pa_y\left(\pa_{yy} \e + \left( \V + \e \right)^5 - \V^5 \right),\Lambda_k Q_k \right\rangle
= \left\langle v - \V,\pa_{yyy}(\Lambda_k Q_k) \right\rangle + \left\langle v^5 - \V^5,\pa_y(\Lambda_k Q_k) \right\rangle
\]
and the regularity of these terms is obtained as before.
Next, observe that the regularity of the term $\langle \P,\Lambda_k Q_k \rangle$ follows from the explicit expression
of $\P$ in the proof of Lemma~\ref{lem:V} and similar arguments.
The regularity of all the other terms in the right-hand side of~\eqref{ortho1}
is proved similarly, and we will not comment on it further.

Now note that, from~\eqref{eq:Wkweighted}, we have
\[
\langle \Lambda_k W_k,\Lambda_k Q_k \rangle = \langle \Lambda_k Q_k,\Lambda_k Q_k \rangle + \eta = \|\Lambda Q\|_{L^2}^2 + \eta
\]
and, using also $\langle Q',\Lambda Q \rangle = 0$ by parity,
\[
\langle \pa_y W_k,\Lambda_k Q_k \rangle = \langle \pa_y Q_k,\Lambda_k Q_k \rangle + \eta = \eta,
\]
where the notation $\eta$ denotes various functions which are locally Lipschitz in $\Gab$
and small for $|s_1|$ large and $\alpha$ small.
Thus, we find
\[
- \langle \vec m_k \cdot \vec \M_k V_k,\Lambda_k Q_k \rangle
= \frac {\dot \mu_k}{\ell_k} \|\Lambda Q\|_{L^2}^2 + \dot \mu_k \eta + \dot y_k \eta + \eta.
\]
Moreover, for $j\neq k$, with similar notation,
since $\langle \Lambda_j W_j,\Lambda_k Q_k \rangle = \eta$ and $\langle \pa_y W_j,\Lambda_k Q_k \rangle = \eta$,
\[
- \langle \vec m_j \cdot \vec \M_j V_j,\Lambda_k Q_k \rangle = \dot \mu_j \eta + \dot y_j \eta + \eta.
\]
Next, using the computations in the proof of Lemma~\ref{lem:rk} and recalling the identity
\[
\frac {\dot {\tilde \mu}_j}{\tilde \mu_j} = \frac {\dot \mu_j}{\mu_j}
+ \frac {\lambda_0}{2\mu_j^3\tau_j^2} {\left( 1 + \frac {\lambda_0}{2\tau_j} \right)}^{-1},
\]
we have also $\dot r_j = \dot \mu_j \eta + \dot y_j \eta + \eta$ for all $1\leq j\leq K$.
For the next term, we simply note that $\langle P_k,\Lambda_k Q_k \rangle = \langle P,\Lambda Q \rangle + \eta$
and $\langle P_j,\Lambda_k Q_k \rangle = \eta$ for $j\neq k$.
Finally, we find similarly
\[
\frac {\dot{\tilde \mu}_k}{\tilde \mu_k} \langle \e,\Lambda_k^2 Q_k \rangle = \dot \mu_k \eta + \eta \quad\m{and}\quad
\dot y_k \langle \e,\pa_y (\Lambda_k Q_k) \rangle = \dot y_k \eta
\]
where, replacing $\e = v - \V$, the functions $\eta$ above depend again continuously on $s$ through the function $v$,
are locally Lipschitz in $\Gab$ and are small for $|s_1|$ large and $\alpha$ small.

In conclusion of these computations, the relation $\frac d{d s} \langle \e,\Lambda_k Q_k \rangle = 0$
rewrites from~\eqref{ortho1} as
\[
L^{\mu_k} \dot \Gab = F^{\mu_k} \quad\m{with}\quad
L^{\mu_k} = L^{\mu_k}_0 + L^{\mu_k}_\eta \ \m{ and }\
L^{\mu_k}_0 = (0,\ldots,0,\ell_k^{-1} \|\Lambda Q\|_{L^2}^2,0,\langle P,\Lambda Q \rangle,0,\ldots,0),
\]
where the nonzero coefficients in $L^{\mu_k}_0$ are located at the positions $4(k-1)+2$ and $4k$,
and where $L^{\mu_k}_\eta$ and $F^{\mu_k}$ are locally Lipschitz in $\Gab$ and continuous in $s$,
and $\|L^{\mu_k}_\eta\|\ll 1$ for $|s_1|$ large and $\alpha$ small.

Similarly, using $\langle Q',y\Lambda Q \rangle = \frac 12 \|y Q\|_{L^2}^2$,
$\langle P,Q \rangle = \frac 1{16} \|Q\|_{L^1}^2$ and
$\langle \Lambda Q,y \Lambda Q \rangle = \langle \Lambda Q,Q \rangle = \langle Q',Q \rangle = 0$,
we check that the other two orthogonality conditions in~\eqref{ortho} rewrite as
\begin{gather*}
L^{y_k} \dot \Gab = F^{y_k} \quad\m{with}\quad
L^{y_k} = L^{y_k}_0 + L^{y_k}_\eta \ \m{ and }\
L^{y_k}_0 = (0,\ldots,0,\tfrac 12 \|y Q\|_{L^2}^2,\langle P,y\Lambda Q \rangle,0,\ldots,0), \\
L^{a_k} \dot \Gab = F^{a_k} \quad\m{with}\quad
L^{a_k} = L^{a_k}_0 + L^{a_k}_\eta \ \m{ and }\
L^{a_k}_0 = (0,\ldots,0,\tfrac 1{16} \|Q\|_{L^1}^2,0,\ldots,0),
\end{gather*}
where the nonzero coefficients in $L^{y_k}_0$ are located at the positions $4(k-1)+3$ and $4k$,
the nonzero coefficient in $L^{a_k}_0$ is located at the position $4k$, and where
$L^{y_k}_\eta$, $F^{y_k}$, $L^{a_k}_\eta$, $F^{a_k}$ are locally Lipschitz in $\Gab$ and continuous in $s$,
and $\|L^{y_k}_\eta\| + \|L^{a_k}_\eta\| \ll 1$ for $|s_1|$ large and $\alpha$ small.
We refer to~\eqref{epsykLkQk} and~\eqref{epsQk} in the proof of Lemma~\ref{lem:ortho} below
for more details about the computation of $\frac d{d s} \langle \e,(\cdot-y_k)\Lambda_k Q_k \rangle$
and $\frac d{d s} \langle \e,Q_k \rangle$.

Denoting by $\Db$ the $(4K) \times (4K)$ matrix formed by the line vectors
$L^{\tau_1},L^{\mu_1},L^{y_1},L^{a_1}$ up to $L^{\tau_K},L^{\mu_K},L^{y_K},L^{a_K}$,
and denoting $\Fb = (F^{\tau_1},F^{\mu_1},F^{y_1},F^{a_1},\ldots,F^{\tau_K},F^{\mu_K},F^{y_K},F^{a_K})^\mathsf{T}$,
we are reduced to solve the differential system $\Db \dot \Gab = \Fb$.
Note that $\Db$ is a perturbation of the block matrix
\[
\Db_0 = \begin{pmatrix} D_0^1 & 0 & \cdots & 0 \\ 0 & D_0^2 & \ddots & \vdots \\
\vdots & \ddots & \ddots & 0 \\ 0 & \cdots & 0 & D_0^K \end{pmatrix}, \quad\m{with}\quad
D_0^k = \begin{pmatrix} 1 & 0 & 0 & 0 \\ 0 & \ell_k^{-1} \|\Lambda Q\|_{L^2}^2 & 0 & \langle P,\Lambda Q \rangle \\
0 & 0 & \frac 12 \|y Q\|_{L^2}^2 & \langle P,y\Lambda Q \rangle \\ 0 & 0 & 0 & \frac 1{16} \|Q\|_{L^1}^2 \end{pmatrix},
\]
in the sense that $\Db = \Db_0 + \Db_\eta$ with $\|\Db_\eta\| \ll 1$ for $|s_1|$ large and $\alpha$ small,
and so $\Db$ is invertible, with bounded inverse.
The system thus rewrites $\dot \Gab = \Db^{-1} \Fb$
and the existence and uniqueness of a $\C^1$ solution $\Gab$ satisfying~\eqref{def:tauk} and~\eqref{ortho}
on some time interval $[s^{in},s^{in} + \bar s]$ with $\bar s > 0$ follows from the Cauchy--Lipschitz theorem.
By the assumptions at $s^{in}$ and continuity arguments, possibly reducing $\bar s$, \eqref{apriori:mutauy}, \eqref{apriori:a}
and~\eqref{apriori:eps} hold on $[s^{in},s^{in} + \bar s]$.
\end{proof}

\begin{lemma}[Modulation equations] \label{lem:ortho}
Assume that the hypotheses of Lemma~\ref{lem:syst_diff} hold on $I$.
Then, for all $1\leq k\leq K$ and for all $s\in I$,
\begin{align}
|\vec m_k| & \lesssim \|\e\|_{L^2_k} + \sum_j \|\e\|_{L_j^2}^2 + |s|^{-2}, \label{est:mk} \\
|\dot a_k| & \lesssim \sum_j \|\e\|_{L^2_j}^2 + |s|^{-2}, \label{est:dotak} \\
|\dot r_k| & \lesssim |s|^{-1} \sum_j \|\e\|_{L^2_j} + |s|^{-2}. \label{est:dotrk}
\end{align}
More precisely, denoting
\[
A_k(s,y) = \frac 1{\|\Lambda Q\|_{L^2}^2} \left[ \pa_{yyy}(\Lambda_k Q_k)
- \mu_k^{-2} \pa_y(\Lambda_k Q_k) + 5Q_k^4 \pa_y(\Lambda_k Q_k) \right],
\]
there holds
\be \label{est:muk}
\left| \frac {\dot \mu_k}{\mu_k} + \frac 1{2\mu_k^3\tau_k} - \frac 1{2s} + \frac {a_k}{\mu_k^3}
+ \langle \e,A_k \rangle \right| \lesssim \sum_j \|\e\|_{L_j^2}^2 + |s|^{-2}.
\ee
Finally, there holds
\be \label{est:dotek}
|\dot e_k| \lesssim \sum_{j\leq k} \|\e\|_{L_j^2} + |s| \sum_j \|\e\|_{L_j^2}^2 + \sum_{j<k} |a_j| + |s|^{-\frac 32}.
\ee
\end{lemma}

\begin{proof}
\textbf{Computation of $\frac d{d s} \langle \e, \Lambda_k Q_k \rangle$.}
We continue the calculation started in the above proof of Lemma~\ref{lem:syst_diff}
and rewrite~\eqref{ortho1}, after integrating by parts,
\begin{align*}
\frac d{d s} \langle \e, \Lambda_k Q_k \rangle
&= - \left( \frac {\dot{\tilde \mu}_k}{\tilde \mu_k} - \frac 1{2s} \right) \langle \e,\Lambda_k^2 Q_k \rangle
- \left( \dot y_k - \frac {y_k}{2s} - \frac 1{\mu_k^2} \right) \langle \e,\pa_y (\Lambda_k Q_k) \rangle \\
&\quad + \left\langle \pa_{yy} \e - \mu_k^{-2} \e + \left( \V + \e \right)^5 - \V^5,\pa_y (\Lambda_k Q_k) \right\rangle \\
&\quad - \sum_j \langle \vec m_j \cdot \vec \M_j V_j,\Lambda_k Q_k \rangle
- \sum_j \dot r_j \langle R_j,\Lambda_k Q_k \rangle - \sum_j \dot a_j \langle P_j,\Lambda_k Q_k \rangle - \langle \P,\Lambda_k Q_k \rangle.
\end{align*}
For the first term, we note that, from~\eqref{apriori:mutilde} and~\eqref{def:mod},
\begin{align*}
- \left( \frac {\dot{\tilde \mu}_k}{\tilde \mu_k} - \frac 1{2s} \right) \langle \e,\Lambda_k^2 Q_k \rangle
&= - \left( \frac {\dot \mu_k}{\mu_k} - \frac 1{2s} \right) \langle \e,\Lambda_k^2 Q_k \rangle + O\left( |s|^{-2} \|\e\|_{L_k^2} \right) \\
&= -m_{k,1} \langle \e,\Lambda_k^2 Q_k \rangle + O\left( |s|^{-1} \|\e\|_{L_k^2} \right) \\
&= O\left( |\vec m_k| \|\e\|_{L_k^2} \right) + O\left( |s|^{-1} \|\e\|_{L_k^2} \right)
\end{align*}
and similarly
\[
- \left( \dot y_k - \frac {y_k}{2s} - \frac 1{\mu_k^2} \right) \langle \e,\pa_y (\Lambda_kQ_k) \rangle
= O\left( |\vec m_k| \|\e\|_{L_k^2} \right) + O\left( |s|^{-1} \|\e\|_{L_k^2} \right).
\]
Next, from the definition of $A_k$ and integration by parts, we find
\begin{align*}
&\left\langle \pa_{yy} \e - \mu_k^{-2} \e + \left( \V + \e \right)^5 - \V^5,\pa_y(\Lambda_k Q_k) \right\rangle \\
&= \|\Lambda Q\|_{L^2}^2 \langle \e,A_k \rangle + \left\langle \left( \V + \e \right)^5 - \V^5 - 5 Q_k^4 \e,\pa_y(\Lambda_k Q_k) \right\rangle \\
&= \|\Lambda Q\|_{L^2}^2 \langle \e,A_k \rangle + O(|s|^{-1} \|\e\|_{L_k^2}) + O(\|\e\|_{L_k^2}^2).
\end{align*}
Indeed, since
\begin{align*}
\left| \left( \V + \e \right)^5 - \V^5 - 5 Q_k^4 \e \right|
& \leq \left| \left( \V +\e \right)^5 - \V^5 - \left( \left( Q_k + \e \right)^5 - Q_k^5 \right) \right|
+ \left| \left( Q_k + \e \right)^5 - Q_k^5 - 5 Q_k^4 \e \right| \\
& \lesssim |\V - Q_k| |\e| + |\e|^2 \lesssim |V_k - Q_k| |\e| + \sum_{j\neq k} |V_j| |\e| + |\e|^2,
\end{align*}
we have, using~\eqref{eq:Wkweighted}, \eqref{eq:Wkleft0} and~\eqref{eq:Wkright},
\[
\left| \left\langle \left( \V + \e \right)^5 - \V^5 - 5 Q_k^4 \e,\pa_y(\Lambda_k Q_k) \right\rangle \right|
\lesssim |s|^{-1} \|\e\|_{L^2_k} + \|\e\|_{L^2_k}^2.
\]
Moreover, since $\langle \Lambda_k W_k,\Lambda_k Q_k \rangle = \langle \Lambda_k Q_k,\Lambda_k Q_k \rangle + O(|s|^{-1})
= \|\Lambda Q\|_{L^2}^2 + O(|s|^{-1})$ from~\eqref{eq:Wkweighted},
and $\langle Q',\Lambda Q \rangle = 0$ by parity, we have
\[
\langle \vec m_k \cdot \vec \M_k V_k,\Lambda_k Q_k \rangle
= -m_{k,1} \left[ \|\Lambda Q\|_{L^2}^2 + O(|s|^{-1}) \right] + O(|s|^{-1} |\vec m_k|).
\]
Note also that, for $j\neq k$,
\[
|\langle \vec \M_j V_j,\Lambda_k Q_k \rangle| + |\langle R_j,\Lambda_k Q_k \rangle|
+ |\langle P_j,\Lambda_k Q_k \rangle| \lesssim |s|^{-10}.
\]
Thus, using $|\langle \P,\Lambda_k Q_k \rangle| \lesssim \|\P\|_{L_k^\infty} \lesssim |s|^{-2}$
from~\eqref{est:Psik}, we find
\begin{align*}
\frac d{d s} \langle \e, \Lambda_k Q_k \rangle
& = m_{k,1} \left[ \|\Lambda Q\|_{L^2}^2 + O(|s|^{-1}) \right] + \|\Lambda Q\|_{L^2}^2 \langle \e,A_k \rangle
+ O(|s|^{-1} |\vec m_k|) + O(|\vec m_k| \|\e\|_{L^2_k}) \\
& \ + O(\|\e\|_{L^2_k}^2) + O(|\dot r_k|) + O(|\dot a_k|)
+ O\left(|s|^{-10} \sum_{j\neq k} (|\vec m_j| + |\dot a_j| + |\dot r_j|) \right) + O(|s|^{-2}).
\end{align*}
Finally, we use~\eqref{e:rk} to estimate $\dot r_k$ and $\dot r_j$ for $j\neq k$, and simplify the last expression as
\be \label{epsLkQk}
\begin{aligned}
\frac d{d s} \langle \e,\Lambda_k Q_k \rangle
&= m_{k,1} \left[ \|\Lambda Q\|_{L^2}^2 + O(|s|^{-1}) \right]
+ \|\Lambda Q\|_{L^2}^2 \langle \e,A_k \rangle + O(|\vec m_k| \|\e\|_{L^2_k}) \\
&\quad + O(\|\e\|_{L^2_k}^2) + O\left( \sum_j |\dot a_j| \right)
+ O\left(|s|^{-1} \sum_j |\vec m_j| \right) + O(|s|^{-2}).
\end{aligned}
\ee

\medskip

\textbf{Computation of $\frac d{d s} \langle \e,(\cdot-y_k)\Lambda_k Q_k \rangle$.}
Using $\langle Q',y\Lambda Q \rangle = \frac 12 \|y Q\|_{L^2}^2$
and $\langle \Lambda Q,y\Lambda Q \rangle = 0$ by parity,
we find with similar computation
\be \label{epsykLkQk}
\begin{aligned}
\frac d{d s} \langle \e,(\cdot-y_k)\Lambda_k Q_k \rangle
&= m_{k,2} \left[ \frac 12 \|y Q\|_{L^2}^2 + O(|s|^{-1}) \right] + O(|\vec m_k| \|\e\|_{L^2_k}) \\
&\quad + O(\|\e\|_{L^2_k}) + O\left( \sum_j |\dot a_j| \right) + O\left( |s|^{-1} \sum_j |\vec m_j| \right) + O(|s|^{-2}).
\end{aligned}
\ee

\medskip

\textbf{Computation of $\frac d{d s} \langle \e,Q_k \rangle$.}
From the equation~\eqref{eq:eps} of $\e$, the expression~\eqref{decompEV} of $\E_\V$ and~\eqref{relationQk}, we find
\begin{align*}
\frac d{d s} \langle \e,Q_k \rangle
&= - \frac 1{2s} \langle \Lambda \e,Q_k \rangle
- \left\langle \pa_y\left(\pa_{yy} \e + \left( \V + \e \right)^5 - \V^5 \right),Q_k \right\rangle
- \sum_j \langle \vec m_j \cdot \vec \M_j V_j,Q_k \rangle \\
&\quad - \sum_j \dot r_j \langle R_j,Q_k \rangle - \sum_j \dot a_j \langle P_j,Q_k \rangle - \langle \P,Q_k \rangle
- \frac {\dot{\tilde \mu}_k}{\tilde \mu_k} \langle \e,\Lambda_k Q_k \rangle - \dot y_k \langle \e,\pa_y Q_k \rangle.
\end{align*}
Note again that, for $j\neq k$,
\[
|\langle \vec \M_j V_j,Q_k \rangle| + |\langle R_j,Q_k \rangle| + |\langle P_j,Q_k \rangle| \lesssim |s|^{-10}.
\]
Moreover, from~\eqref{PQ} and~\eqref{def:R},
\[
\langle P_k,Q_k \rangle = \frac 1{16} \|Q\|_{L^1}^2 + O(|s|^{-10}),\quad
\langle R_k,Q_k \rangle = -\frac 34 \|Q\|_{L^1},
\]
and
\[
\langle \Lambda_k Q_k,Q_k \rangle = \langle \pa_y Q_k,Q_k \rangle = 0.
\]
Thus, after integration by parts,
\begin{align*}
\frac d{d s} \langle \e,Q_k \rangle
& = \left\langle \pa_{yy} \e - \tilde\mu_k^{-2} \e + \left( \V + \e \right)^5 - \V^5,\pa_y Q_k \right\rangle
+ \frac 34 \|Q\|_{L^1} \dot r_k - \frac 1{16} \|Q\|_{L^1}^2 \dot a_k \\
& \quad - \langle \P,Q_k \rangle - \left( \frac {\dot{\tilde \mu}_k}{\tilde \mu_k}
- \frac 1{2s} \right) \langle \e,\Lambda_k Q_k \rangle
- \left( \dot y_k - \frac {y_k}{2s} - \frac 1{\tilde \mu_k^2} \right) \langle \e,\pa_y Q_k \rangle \\
& \quad + O(|s|^{-1} |\vec m_k|) + O\left(|s|^{-10} \sum_{j\neq k} (|\vec m_j| + |\dot a_j| + |\dot r_j|) \right)
+ O(|s|^{-10} |\dot a_k|).
\end{align*}
But, from the cancellation $LQ'=0$ and the definition of $Q_k$, we have
\be \label{LQkprime}
\pa_{yyy} Q_k -\tilde\mu_k^{-2}\pa_y Q_k + 5Q_k^4\pa_y Q_k = 0.
\ee
Thus, from~\eqref{apriori:mutilde} and again integration by parts, we obtain as before
\begin{align*}
\frac d{d s} \langle \e,Q_k \rangle
&= \left\langle \left( \V + \e \right)^5 - \V^5 - 5 Q_k^4 \e,\pa_y Q_k \right\rangle
+ \frac 34 \|Q\|_{L^1} \dot r_k - \frac 1{16} \|Q\|_{L^1}^2 \dot a_k \\
&\quad - \langle \P,Q_k \rangle + O\left( |\vec m_k| \|\e\|_{L_k^2} \right) + O\left( |s|^{-1} \|\e\|_{L_k^2} \right) \\
&\quad + O(|s|^{-1} |\vec m_k|) + O\left(|s|^{-10} \sum_j (|\vec m_j| + |\dot a_j| + |\dot r_j|) \right),
\end{align*}
and also
\[
\left| \left\langle \left( \V + \e \right)^5 - \V^5 - 5 Q_k^4 \e,\pa_y Q_k \right\rangle \right|
\lesssim |s|^{-1} \|\e\|_{L^2_k} + \|\e\|_{L^2_k}^2.
\]
Therefore, using~\eqref{est:PsiQ}, \eqref{rkdk} and~\eqref{e:rk}, we obtain
\be \label{epsQk}
\begin{aligned}
\frac d{d s} \langle \e,Q_k \rangle
& = - \frac 1{16} \|Q\|_{L^1}^2 \left( \dot a_k + \frac {2a_k}{\mu_k^3\tau_k} + \frac {2a_k^2}{\mu_k^3} \right)
- \frac 14 \|Q\|_{L^1} \left( \dot r_k + \frac {r_k}{\mu_k^3\tau_k} + \frac {2a_k r_k}{\mu_k^3} \right) \\
& \quad + O(|s|^{-1} \|\e\|_{L_k^2}) + O(|\vec m_k| \|\e\|_{L_k^2}) + O(\|\e\|_{L_k^2}^2) \\
& \quad + O\left( |s|^{-1} \sum_{j<k} |a_j| \right) + O\left( |s|^{-10} \sum_j |\dot a_j| \right) \\
& \quad + O\left( |s|^{-1} \sum_{j\leq k} |\vec m_j| \right) + O\left( |s|^{-10} \sum_{j>k} |\vec m_j| \right) + O(|s|^{-\frac 52}).
\end{aligned}
\ee

\bigbreak

\textbf{Proof of~\eqref{est:mk}--\eqref{est:muk}.}
We assume now that the three orthogonality conditions~\eqref{ortho} are satisfied.
First, we give a simplified bound on $|\dot a_k|$ deduced from $\frac d{d s} \langle \e,Q_k \rangle = 0$
in~\eqref{epsQk}. Using also~\eqref{e:rk}, we find
\[
|\dot a_k| \lesssim |s|^{-2} + |\vec m_k| \|\e\|_{L_k^2} + \|\e\|_{L_k^2}^2
+ |s|^{-1} \sum_j (|\dot a_j| + |\vec m_j|).
\]
Summing over $1\leq k\leq K$ and taking $|s_1|$ large enough, we obtain
\[
\sum_k |\dot a_k| \lesssim |s|^{-2} + \sum_j |\vec m_j| \|\e\|_{L_j^2} + \sum_j \|\e\|_{L_j^2}^2
+ |s|^{-1} \sum_j |\vec m_j|,
\]
and so, inserting this estimate in the above bound on $|\dot a_k|$, we get
\[
|\dot a_k| \lesssim |s|^{-2} + |\vec m_k| \|\e\|_{L_k^2} + \sum_j \|\e\|_{L_j^2}^2 + |s|^{-1} \sum_j |\vec m_j|.
\]
Now, we insert this estimate on $|\dot a_k|$ into~\eqref{epsLkQk} and~\eqref{epsykLkQk} and we obtain,
using~\eqref{ortho},
\[
|\vec m_k| \lesssim \|\e\|_{L^2_k} + \sum_j |\vec m_j| \|\e\|_{L^2_j}
+ \sum_j \|\e\|_{L_j^2}^2 + |s|^{-1} \sum_j |\vec m_j| + |s|^{-2}.
\]
Summing as before over $1\leq k\leq K$, using~\eqref{apriori:eps} and taking $|s_1|$ large enough, we find
\[
\sum_k |\vec m_k| \lesssim \sum_j \|\e\|_{L_j^2} + |s|^{-2},
\]
and so~\eqref{est:mk}.
Inserting this estimate into the above simplified bound on $|\dot a_k|$ and into~\eqref{e:rk},
we obtain respectively~\eqref{est:dotak} and~\eqref{est:dotrk}.
Finally, inserting~\eqref{est:mk} into~\eqref{epsLkQk}, we obtain~\eqref{est:muk}.

\medskip

\textbf{Proof of~\eqref{est:dotek}.}
Directly from the definition~\eqref{def:ek} of $e_k$, we first compute
\[
\dot e_k = \frac 1{\mu_k^2} \left[ \left( a_k + \frac 1{2\tau_k} + \frac {4r_k}{\|Q\|_{L^1}} \right)
- 2s \frac {\dot\mu_k}{\mu_k} \left( a_k + \frac 1{2\tau_k} + \frac {4r_k}{\|Q\|_{L^1}} \right)
+ s \left( \dot a_k - \frac {\dot\tau_k}{2\tau_k^2} + \frac {4\dot r_k}{\|Q\|_{L^1}} \right) \right].
\]
Using~\eqref{def:tauk} and the definition~\eqref{def:mod} of $m_{k,1}$, we find
\[
\dot e_k = \frac s{\mu_k^2} \left[ \left( \dot a_k + \frac {2a_k}{\mu_k^3\tau_k} + \frac {2a_k^2}{\mu_k^3} \right)
+ \frac 4{\|Q\|_{L^1}} \left( \dot r_k + \frac {r_k}{\mu_k^3\tau_k} + \frac {2a_k r_k}{\mu_k^3} \right) \right]
+ O(|\vec m_k|).
\]
But, inserting estimates~\eqref{est:mk}--\eqref{est:dotrk} into~\eqref{epsQk},
and using from~\eqref{ortho} the orthogonality condition $\frac d{d s} \langle \e,Q_k \rangle = 0$,
we also find
\begin{multline*}
\left| \frac 1{16} \|Q\|_{L^1}^2 \left( \dot a_k + \frac {2a_k}{\mu_k^3\tau_k} + \frac {2a_k^2}{\mu_k^3} \right)
+ \frac 14 \|Q\|_{L^1} \left( \dot r_k + \frac {r_k}{\mu_k^3\tau_k} + \frac {2a_k r_k}{\mu_k^3} \right) \right| \\
\lesssim |s|^{-1} \sum_{j\leq k} \|\e\|_{L_j^2} + \sum_j \|\e\|_{L_j^2}^2 + |s|^{-1} \sum_{j<k} |a_j| + |s|^{-\frac 52}.
\end{multline*}
Gathering the two last estimates and~\eqref{est:mk}, we obtain~\eqref{est:dotek},
which concludes the proof of Lemma~\ref{lem:ortho}.
\end{proof}

Next, we prove more precise estimates on the modulation parameters,
using the notation introduced in~\eqref{apriori:mutauy}.

\begin{lemma}[Refined modulation equations] \label{lem:orthobis}
Assume that the hypotheses of Lemma~\ref{lem:syst_diff} hold on $I$.
Let $f_k = \bar \mu_k + \bar y_k$.
Then, for all $1\leq k\leq K$ and for all $s\in I$,
\be \label{est:fk}
\begin{aligned}
\left| \dot f_k + \frac 12 \left( 1 + 3\theta_k \right) \frac {f_k}s \!\! \right.
& \left. \vphantom{\frac 12} + (1 + \bar\mu_k) \langle \e,A_k \rangle \right|
\lesssim |s|^{-1} \left| e_k - \ell_k \left( \frac 12 + \theta_k \right) \right| + \sum_j \|\e\|_{L_j^2}^2 \\
&\qquad\quad + |s|^{-1} \left( |s|^{-\frac 1{42}} + |\bar\mu_k|^2 + |\bar y_k|^2
+ \sum_{j<k} (|\bar\mu_j| + |\bar\tau_j| + |\bar y_j|) \right)
\end{aligned}
\ee
and
\be \label{est:barykfk}
\left| \dot{\bar y}_k - \frac {\bar y_k}{2s} + \frac {f_k}s \right|
\lesssim \sum_j \|\e\|_{L_j^2}^2 + |s|^{-1} |\bar\mu_k|^2 + |s|^{-2}.
\ee
\end{lemma}

\begin{proof}
We first prove the two estimates
\be \label{est:barmuk}
\begin{aligned}
\left| \dot{\bar\mu}_k - \frac {\bar\mu_k}{2s} \!\! \right.
& \left. {}+ \frac 1s \left( \frac {e_k}{\ell_k} - \frac 12 - \theta_k \right)
+ \frac {3\theta_k}{2s} (\bar\mu_k + \bar y_k) + (1 + \bar\mu_k) \langle \e,A_k \rangle \right| \\
&\qquad \lesssim \sum_j \|\e\|_{L_j^2}^2 + |s|^{-1} \left( |s|^{-\frac 1{42}} + |\bar\mu_k|^2
+ |\bar y_k|^2 + \sum_{j<k} (|\bar\mu_j| + |\bar\tau_j| + |\bar y_j|) \right)
\end{aligned}
\ee
and
\be \label{est:baryk}
\left| \dot{\bar y}_k + \frac {\bar y_k}{2s} + \frac {\bar \mu_k}s \right|
\lesssim \sum_j \|\e\|_{L_j^2}^2 + |s|^{-1} |\bar\mu_k|^2 + |s|^{-2}.
\ee
Indeed, multiplying~\eqref{est:muk} by $\mu_k$, replacing
\[
\frac {a_k}{\mu_k^2} + \frac 1{2\mu_k^2\tau_k} = \frac {e_k}s - \frac {4r_k}{\mu_k^2 \|Q\|_{L^1}}
\]
from~\eqref{def:ek} and $r_k$ by its asymptotics~\eqref{eq:rkasymp}, we find
\begin{multline*}
\left| \dot\mu_k + \frac {e_k}s - \frac {\mu_k}{2s} + \mu_k \langle \e,A_k \rangle
- \frac {\ell_k^3\theta_k}{\mu_k^2 s} \left( 1 + \frac 12 \bar\mu_k - \frac 32 \bar y_k \right) \right| \\
\lesssim \sum_j \|\e\|_{L_j^2}^2 + |s|^{-1} \left( |s|^{-\frac 1{42}} + |\bar\mu_k|^2 + |\bar y_k|^2
+ \sum_{j<k} (|\bar\mu_j| + |\bar\tau_j| + |\bar y_j|) \right),
\end{multline*}
and so~\eqref{est:barmuk} after expanding $\mu_k = \ell_k(1 + \bar\mu_k)$ and
$\mu_k^{-2} = \ell_k^{-2} \left[ 1 - 2\bar\mu_k + O(|\bar\mu_k|^2) \right]$
in the left-hand side. Similarly, from $y_k = 2\ell_k^{-2}s(1 + \bar y_k)$, we get
$\dot y_k = 2\ell_k^{-2}(1 + \bar y_k) + 2\ell_k^{-2} s \dot{\bar y}_k$, which gives,
inserted into the second line of~\eqref{est:mk},
\[
|2s \dot{\bar y}_k + \bar y_k + 2\bar\mu_k| \lesssim |\bar\mu_k|^2
+ \|\e\|_{L_k^2} + \sum_j \|\e\|_{L_j^2}^2 + |s|^{-1}.
\]
Dividing the last estimate by $2s$, we get~\eqref{est:baryk}.
Finally, adding~\eqref{est:barmuk} and~\eqref{est:baryk} yields~\eqref{est:fk},
and replacing $\bar \mu_k = f_k - \bar y_k$ into~\eqref{est:baryk} yields~\eqref{est:barykfk}.
\end{proof}

Finally, as in the previous works~\cite{MMjmpa,MMR1,Mjams} related to blow up for (gKdV),
we need some $L^1$ type estimates on $\e$.
Technically, we prove \emph{a priori} estimates on the quantity $\langle \e,P_k \rangle$
and its time variation in the following lemma.

\begin{lemma} \label{lem:epsPk}
Assume that the hypotheses of Lemma~\ref{lem:syst_diff} hold on $I$.
Then, for all $1\leq k\leq K$ and for all $s\in I$,
\be \label{apriori:epsPk}
|\langle \e,P_k \rangle| \lesssim |s|^{\frac 12} \|\e\|_{L^2(y > y_{k+1})},
\ee
and, for some constant $c>0$,
\be \label{epsPk}
\left| \frac d{d s} \langle \e,P_k \rangle - cs\dot a_k \right| \lesssim |s|^{-\frac 12} \|\e\|_{L^2(y > y_{k+1})}
+ \sum_{j\leq k} \|\e\|_{L^2_j} + \sum_{j\leq k} |a_j| + \sum_j \|\e\|_{L_j^2}^2 + |s|^{-\frac 98}.
\ee
\end{lemma}

\begin{proof}
To prove~\eqref{apriori:epsPk}, we note that, from~\eqref{ref:gamma} and the definition of $P_k$,
we have $P_k(y)=0$ for $y\leq \frac 12 (y_{k+1} + y_k)$.
Thus, using also $\|P_k\|_{L^2} \lesssim |s|^{\frac 12}$ from~\eqref{eq:Pknorms},
\[
|\langle \e,P_k \rangle| \lesssim \|P_k\|_{L^2} \|\e\|_{L^2(y > \frac 12 (y_{k+1} + y_k))}
\lesssim |s|^{\frac 12} \|\e\|_{L^2(y > y_{k+1})}.
\]

To prove~\eqref{epsPk}, we proceed as in the proof of Lemma~\ref{lem:ortho}.
First note that
\[
\pa_s P_k = - \frac {\dot{\tilde\mu}_k}{\tilde\mu_k} \Lambda_k P_k - \dot y_k \pa_y P_k + |s|^{-1} Z_k,
\]
where $Z_k$ is defined by~\eqref{def:Zj} and satisfies, from~\eqref{apriori:mutauy} and~\eqref{ref:gamma},
\[
|Z_k(y)| \lesssim \mathbf{1}_{\frac 12 (y_{k+1} + y_k) < y < y_k - \frac {\ell_k\g}2 |s|}(y).
\]
Thus, using the equation~\eqref{eq:eps} of $\e$, the decomposition~\eqref{decompEV} of $\E_\V$
and integrations by parts, we find
\begin{align*}
\frac d{d s} \langle \e,P_k \rangle
& = - \left( \frac {\dot{\tilde \mu}_k}{\tilde \mu_k} - \frac 1{2s} \right) \langle \e,\Lambda_k P_k \rangle
- \left( \dot y_k - \frac {y_k}{2s} \right) \langle \e,\pa_y P_k \rangle \\
& \quad + |s|^{-1} \langle \e,Z_k \rangle + \langle \e,\pa_{yyy} P_k \rangle
+ \left\langle \left( \V + \e \right)^5 - \V^5,\pa_y P_k \right\rangle \\
& \quad - \sum_j \langle \vec m_j \cdot \vec \M_j V_j,P_k \rangle - \sum_j \dot r_j \langle R_j,P_k \rangle
- \sum_j \dot a_j \langle P_j,P_k \rangle - \langle \P,P_k \rangle.
\end{align*}

For the first term, from~\eqref{apriori:mutilde}, \eqref{def:mod},
\eqref{apriori:eps} and~\eqref{est:mk}, we note that
\[
\left| \frac {\dot{\tilde\mu}_k}{\tilde\mu_k} - \frac 1{2s} \right|
\lesssim \left| \frac {\dot{\mu}_k}{\mu_k} - \frac 1{2s} \right| + |s|^{-2}
\lesssim |\vec m_k| + |s|^{-1} \lesssim \|\e\|_{L^2_k} + |s|^{-1}.
\]
Thus, using also~\eqref{b:Pky}, we find
\begin{align*}
\left| \left( \frac {\dot{\tilde \mu}_k}{\tilde \mu_k} - \frac 1{2s} \right) \langle \e,\Lambda_k P_k \rangle \right|
&\lesssim \left( \|\e\|_{L^2_k} + |s|^{-1} \right) \left( \|\e\|_{L_k^2} + |s|^{\frac 12} \|\e\|_{L^2(y > y_{k+1})} \right) \\
&\lesssim \|\e\|_{L_k^2} + |s|^{-\frac 12} \|\e\|_{L^2(y > y_{k+1})}.
\end{align*}
Similarly, we also find
\[
\left| \left( \dot y_k - \frac {y_k}{2s} \right) \langle \e,\pa_y P_k \rangle \right|
\lesssim \int |\e| |\pa_y P_k| \lesssim \|\e\|_{L_k^2} + |s|^{-\frac 12} \|\e\|_{L^2(y > y_{k+1})}
\]
and
\[
|s|^{-1} |\langle \e,Z_k \rangle| \lesssim |s|^{-\frac 12} \|\e\|_{L^2(y > y_{k+1})}.
\]
Using again~\eqref{b:Pky}, we find
\[
\left| \langle \e,\pa_{yyy} P_k \rangle \right| \lesssim \|\e\|_{L_k^2} + |s|^{-\frac 52} \|\e\|_{L^2(y > y_{k+1})}
\]
and, since $\|\V\|_{L^\infty} + \|\e\|_{L^\infty} \lesssim 1$,
\[
\left| \left\langle \left( \V + \e \right)^5 - \V^5,\pa_y P_k \right\rangle \right|
\lesssim \|\e\|_{L_k^2} + |s|^{-\frac 12} \|\e\|_{L^2(y > y_{k+1})}.
\]

For the next term, we first note that, from~\eqref{est:P}, \eqref{apriori:mutauy} and~\eqref{ref:gamma}, for $j\neq k$,
\[
\left\{
\begin{gathered}
|\langle \Lambda_j R_j,P_k \rangle| + |\langle \pa_y R_j,P_k \rangle|
+ |\langle \Lambda_j P_j,P_k \rangle| + |\langle \pa_y P_j,P_k\rangle| \lesssim |s|^{-10}, \\
|\langle \Lambda_k R_k,P_k \rangle| + |\langle \pa_y R_k,P_k \rangle| \lesssim 1 \quad\m{and}\quad
\langle \Lambda_k P_k,P_k \rangle = \langle \pa_y P_k,P_k \rangle = 0.
\end{gathered}
\right.
\]
Then, for $j>k$, from~\eqref{ref:gamma} and~\eqref{eq:Wkright}, we find
\[
|\langle \Lambda_j W_j,P_k \rangle| + |\langle \pa_y W_j,P_k \rangle| \lesssim |s|^{-10}.
\]
Finally, for $j\leq k$, using~\eqref{eq:Pknorms}, \eqref{eq:WkHm} and~\eqref{eq:LkWk}, we have
\begin{align*}
|\langle \Lambda_j W_j,P_k \rangle|
&\leq |\langle \Lambda_j(W_j - Q_j),P_k \rangle| + |\langle \Lambda_j Q_j,P_k \rangle| \\
&\lesssim \|\Lambda_j(W_j - Q_j)\|_{L^2} \|P_k\|_{L^2} + \|\Lambda_j Q_j\|_{L^1} \lesssim 1
\end{align*}
and, using~\eqref{eq:WkHm} for $m=1$ and $\langle P,Q' \rangle = 0$,
we find similarly $|\langle \pa_y W_j,P_k \rangle| \lesssim |s|^{-\frac 12}$.
Thus, gathering the previous estimates, using~\eqref{est:mk} and then~\eqref{apriori:eps}, we obtain
\[
\left| \sum_j \langle \vec m_j \cdot \vec \M_j V_j,P_k \rangle \right|
\lesssim \sum_{j\leq k} |\vec m_j| + |s|^{-10} \sum_{j>k} |\vec m_j|
\lesssim \sum_{j\leq k} \|\e\|_{L^2_j} + \sum_j \|\e\|_{L_j^2}^2 + |s|^{-2}.
\]

Next, by~\eqref{est:dotrk} and the decay properties of $R_j$, we find
\[
\left| \sum_j \dot r_j \langle R_j,P_k \rangle \right| \lesssim \sum_j |\dot r_j|
\lesssim |s|^{-1} \sum_j \|\e\|_{L^2_j} + |s|^{-2} \lesssim |s|^{-\frac 32}.
\]

For the next term, we first note that, by~\eqref{est:P} and~\eqref{ref:gamma},
we have $|\langle P_j,P_k \rangle| \lesssim |s|^{-10}$ for $j\neq k$.
For $j=k$, we compute
\[
\int P_k^2 = \tilde\mu_k^{-1} \int P^2 \left( \frac {y-y_k}{\tilde\mu_k} \right)
\chi^2 \left( \frac {y-y_k}{\tilde\mu_k} |s|^{-1} \right) dy = |s| \int P^2(|s|z) \chi^2(z) \,dz.
\]
But, again from~\eqref{est:P}, we find
\[
|s| \int_{z>0} P^2(|s|z) \chi^2(z) \,dz \lesssim |s| \int_{z>0} e^{-|s|z} \,dz \lesssim 1
\]
and
\[
|s| \left| \int_{z<0} P^2(|s|z) \chi^2(z) \,dz - \frac 14 \|Q\|_{L^1}^2 \int_{z<0} \chi^2(z) \,dz \right|
\lesssim |s| \int_{z<0} e^{|s|z/2} \,dz \lesssim 1.
\]
Thus, denoting $c = \frac 14 \|Q\|_{L^1}^2 \int_{z<0} \chi^2(z) \, dz >0$,
we have found $|\langle P_k,P_k \rangle + c s| \lesssim 1$
and so, using~\eqref{apriori:eps} and~\eqref{est:dotak},
\[
\left| \sum_j \dot a_j \langle P_j,P_k \rangle + c s \dot a_k \right|
\lesssim |\dot a_k| + |s|^{-11} \lesssim \sum_j \|\e\|_{L_j^2}^2 + |s|^{-2}.
\]

Finally, to control the source term $\langle \P,P_k \rangle$, we use~\eqref{eq:Pknorms},
\eqref{ref:gamma} and~\eqref{est:Psik}, and thus obtain, for $k<K$,
\[
|\langle \P,P_k \rangle| \lesssim \|P_k\|_{L^2} \|\P\|_{L^2(y > \frac 12 (y_{k+1} + y_k))}
\lesssim |s|^{\frac 12} \|\P\|_{L^2(y > y_{k+1} - |s|^{\frac 14})}
\lesssim |s|^{-\frac 98} + \sum_{j\leq k} |a_j|.
\]
Similarly, for $k=K$, using~\eqref{est:Psi}, we find
\[
|\langle \P,P_K \rangle| \lesssim \|P_K\|_{L^2} \|\P\|_{L^2} \lesssim |s|^{-\frac 54} + \sum_j |a_j|.
\]
Gathering the above estimates, we get~\eqref{epsPk}, which concludes the proof of Lemma~\ref{lem:epsPk}.
\end{proof}

\subsection{Weak $H^1$ stability of the decomposition}

The next lemma shows that the decomposition of Lemma~\ref{lem:syst_diff} is stable by weak $H^1$ limit.
To prove such a result, we rely on the weak $H^1$ stability of the flow of~\eqref{kdv},
as stated in~\cite[Lemma 2.10]{CM1} and proved \emph{e.g.}~in~\cite{MMjmpa}.

\begin{lemma} \label{lem:weak}
Let $(v_{0,n})$ be a sequence of $H^1$ functions such that
\[
v_{0,n} \rightharpoonup v_0 \quad \m{in } H^1 \m{ weak as } n\to +\infty.
\]
Let $I = [S_1,S_2]$ with $S_1 < S_2 < 2s_1 < 0$ and $S_0\in I$.
Assume that, for all $n$ large, the solution $v_n$ of~\eqref{rescaledkdv}
such that $v_n(S_0) = v_{0,n}$ satisfies the assumptions of Lemma~\ref{lem:syst_diff} on $I$.
In particular, for any $n$, there exists a unique $\Gab_n$ such that $\Gab_n$ and $\e_n$
defined by $\e_n = v_n - \V[\Gab_n]$ satisfy~\eqref{def:tauk}, \eqref{apriori:mutauy},
\eqref{apriori:a}, \eqref{apriori:eps} and~\eqref{ortho} on $I$.
Let $v$ be the solution of~\eqref{rescaledkdv} satisfying $v(S_0) = v_0$.
Then,
\[
\forall s\in I, \quad v_n(s) \rightharpoonup v(s) \quad \m{in } H^1 \m{ weak as } n\to +\infty,
\]
and
\[
\forall s\in I, \quad \e_n(s) \rightharpoonup \e(s) \quad \m{in } H^1 \m{ weak},\quad
\Gab_n(s)\to \Gab(s) \m{ as } n\to +\infty,
\]
where $\Gab$ and $\e = v - \V[\Gab]$ satisfy~\eqref{def:tauk}, \eqref{apriori:mutauy},
\eqref{apriori:a}, \eqref{apriori:eps} and~\eqref{ortho} on $I$.
\end{lemma}

\begin{proof}
Let $u_n$ and $u$ be the solutions of~\eqref{kdv} defined from the change of variables~\eqref{changevar}
such that $v_n = \tilde u_n$ and $v = \tilde u$ respectively. Let $T_i = \frac 1{\sqrt{-2S_i}}$ for $i=0,1,2$.
Then $u_n(T_0) \rightharpoonup u(T_0)$ in $H^1$ weak as $n\to +\infty$ and,
since $v_n$ satisfies the assumptions of Lemma~\ref{lem:syst_diff},
there exists $C_1 > 0$ such that $\max_{t\in [T_1,T_2]} \|u_n(t)\|_{H^1} \leq C_1$.
We deduce from the $H^1$ weak stability of the flow of~\eqref{kdv} that, for all $t\in [T_1,T_2]$,
$u_n(t) \rightharpoonup u(t)$ in $H^1$ weak and so, from~\eqref{changevar},
$v_n(s) \rightharpoonup v(s)$ in $H^1$ weak for any $s\in I$.

Next, it follows from~\eqref{apriori:mutauy}, \eqref{apriori:a} and~\eqref{def:tauk},
\eqref{est:mk}, \eqref{est:dotak} that $|\Gab_n| + |\dot \Gab_n| \leq C$ on $I$ for some $C > 0$.
Thus, by Ascoli's theorem, up to the extraction of a subsequence, there exists a continuous function $\Gab$
satisfying~\eqref{apriori:mutauy} and~\eqref{apriori:a} such that $\Gab_n\to \Gab$ uniformly on $I$ as $n\to +\infty$.
By the definition of $\V[\Gab]$, it follows that $\V[\Gab_n]\to \V[\Gab]$ in $H^1(\R)$ as $n\to +\infty$.
The weak $H^1$ convergence of $\e_n$ to $\e$ then follows.
Writing the integral form of the differential system satisfied by $\Gab_n$
as stated in Remark~\ref{rem:system} and passing to the limit, we obtain that
$\Gab$ is a $\C^1$ function of time which also satisfies~\eqref{def:tauk} on $I$.

By weak convergence, $\langle \e,\Lambda_k Q_k \rangle$, $\langle \e,(\cdot-y_k)\Lambda_k Q_k \rangle$
and $\langle \e,Q_k \rangle$ are constant functions of time and thus~\eqref{ortho} holds for $\Gab$ and $\e$.
By weak convergence again, the function $v$ satisfies the assumptions of Lemma~\ref{lem:syst_diff} on $I$,
and the uniqueness statement proves that $\Gab$ and $\e$ correspond to the decomposition of $v$ from Lemma~\ref{lem:syst_diff}.
It follows that the above limits hold for the whole sequence.
\end{proof}

\section{Proof of Theorem~\ref{thm:main}}

The proof of Theorem~\ref{thm:main} is mainly based on the following proposition,
written with the notation of Section~\ref{sec:decomposition},
\emph{i.e.}~in the rescaled variables~\eqref{changevar}--\eqref{rescaledkdv}
and using the approximate solution $\V(s) = \V[\Gab(s)]$.

\begin{proposition} \label{prop:v}
There exist $S_0<-1$ with $|S_0|$ large, an $H^1$ solution $v$ of~\eqref{rescaledkdv} and a $\C^1$ function
$\Gab = \left(\tau_1,\mu_1,y_1,a_1,\ldots,\tau_K,\mu_K,y_K,a_K\right)^\mathsf{T}$ defined on $(-\infty,S_0]$ such that,
defining $\e$ by
\[
v(s) = \V[\Gab(s)] + \e(s),
\]
for all $1\leq k\leq K$, for all $s\leq S_0$,
\be \label{est:final}
\left\{
\begin{gathered}
\left| \frac {\mu_k(s)}{\ell_k} - 1 \right| + \left| \frac {\tau_k(s)}{s\ell_k^{-3}} - 1 \right|
+ \left| \frac {y_k(s)}{2s\ell_k^{-2}} - 1 \right| \leq |s|^{-\frac 1{43}}, \\
|a_k(s)| \leq |s|^{-1 - \frac 1{43}},\quad \|\e(s)\|_{H^1} \leq |s|^{-\frac 12 - \frac 1{43}},
\end{gathered}
\right.
\ee
and
\be \label{ortho:final}
\langle \e(s),\Lambda_k Q_k(s) \rangle = \langle \e(s),(\cdot - y_k)\Lambda_k Q_k(s) \rangle
= \langle \e(s),Q_k(s) \rangle = 0.
\ee
\end{proposition}

The proof of Proposition~\ref{prop:v} is by compactness.
For $n>1$ large, we let $S_n = -n$, and we construct a solution $v_n$
of~\eqref{rescaledkdv} with a well-prepared data $v_n(S_n)$.
Proposition~\ref{pr:boot} below shows uniform estimates on $v_n$ on a time interval $[S_n,S_0]$,
where $S_0 < -1$ and $|S_0| > 2|s_1|$ is large enough but independent of $n$.
In Section~\ref{sec:conclusion}, we obtain $v_n(S_0) \rightharpoonup v_0$
passing to the weak limit in $n$ (up to a subsequence).
Then, from the uniform estimates and Lemma~\ref{lem:weak}, the initial data $v(S_0) = v_0$
provides the desired solution $v$ of~\eqref{rescaledkdv}.

The proof of Theorem~\ref{thm:main} is also given in Section~\ref{sec:conclusion},
as a straightforward consequence of Proposition~\ref{prop:v} and the change of variables~\eqref{changevar}.

\subsection{Formal discussion on the system of parameters} \label{sec:discussion}

We discuss formally how to derive estimates on the various parameters
on a time interval $[S_n,S_0]$ from Lemmas~\ref{lem:ortho} and~\ref{lem:orthobis},
in particular from~\eqref{est:dotek}, \eqref{est:fk} and~\eqref{est:barykfk}.
We refer to Sections~\ref{sec:parameters} and~\ref{sec:Brouwer} below for rigorous arguments.
Suppose that we already know from an induction argument that, for all $j < k$,
\[
|\bar\mu_j| + |\bar\tau_j| + |\bar y_j| \leq |s|^{-\d_{k-1}},\quad |a_j| \leq |s|^{-1 - \d_{k-1}},
\]
for some small $\d_{k-1} > 0$.
Then, from~\eqref{est:dotek} and neglecting $\e$, we obtain
$|\dot e_k| \lesssim |s|^{-1 - \d_{k-1}}$ and so, by integration and
assuming $e_k(S_n) = \ell_k \left( \frac 12 + \theta_k \right)$, we find
$\left| e_k - \ell_k \left( \frac 12 + \theta_k \right) \right| \lesssim |s|^{-\d_{k-1}}$.
Next, from~\eqref{est:fk}, inserting the above information on $e_k$,
taking $\d_{k-1} \leq \frac 1{42}$ and neglecting for the moment the higher order terms $|\bar \mu_k|^2$
and $|\bar y_k|^2$, we get
\[
\left| \dot f_k + \frac 12 \left( 1 + 3\theta_k \right) \frac {f_k}s \right|
\lesssim |s|^{-1 - \d_{k-1}}.
\]
The general behavior of functions $f_k$ satisfying such a differential inequality
depends on the value of the parameter $\frac 12 (1 + 3\theta_k)$.
Indeed, the last inequality is equivalent to
\[
\left| \frac d{d s} \left[ (-s)^{\frac 12 (1 + 3\theta_k)} f_k \right] \right|
\lesssim |s|^{-1 - \d_{k-1} + \frac 12 (1 + 3\theta_k)}.
\]
If $\frac 12 (1 + 3\theta_k) < \d_{k-1}$, then it is clear by direct integration on $[S_n,s]$
that any data $f_k(S_n)$ such that $|f_k(S_n)| \lesssim |S_n|^{-\d_{k-1}}$ leads to $|f_k(s)| \lesssim |s|^{-\d_{k-1}}$.
On the contrary, if $\frac 12 (1 + 3\theta_k) > \d_{k-1}$, we have to choose the data $f_k(S_n)$ carefully
to obtain the same estimate on $f_k(s)$ and avoid the instability.
At the technical level, we will need to treat such indices $k$ by a global topological argument.
For the sake of simplicity, we will choose the values of $\d_k$ so that we avoid the remaining case,
\emph{i.e.}~when $\frac 12 (1 + 3\theta_k) = \d_{k-1}$.
Next, from~\eqref{est:barykfk}, we get similarly
\[
\left| \dot{\bar y}_k - \frac {\bar y_k}{2s} \right| \lesssim \left| \frac {f_k}s \right|
+ |s|^{-2} \lesssim |s|^{-1 - \d_{k-1}}, \quad\m{which rewrites as}\quad
\left| \frac d{d s} \left[ (-s)^{-\frac 12} \bar y_k \right] \right| \lesssim |s|^{-1 - \d_{k-1} - \frac 12}.
\]
Here, we note that any sufficiently small data $\bar y_k(S_n)$ leads to the estimate
$|\bar y_k(s)| \lesssim |s|^{-\d_{k-1}}$,
and so $|\bar \mu_k(s)| \lesssim |s|^{-\d_{k-1}}$ since $f_k = \bar \mu_k + \bar y_k$.
However, from~\eqref{def:tauk}, since we obtain
\[
\left| \dot{\bar \tau}_k + \frac {\bar \tau_k}s \right| \lesssim |s|^{-1} |\bar \mu_k|
\lesssim |s|^{-1 - \d_{k-1}}, \quad\m{which rewrites as}\quad
\left| \frac d{d s} (s\bar \tau_k) \right| \lesssim |s|^{-\d_{k-1}},
\]
we deduce that $\bar \tau_k$ is an unstable parameter for all $1\leq k\leq K$,
and so the data $\bar \tau_k(S_n)$ will also have to be included in the global topological argument.
Finally, the parameter $a_k$ is controlled using the almost conservation of $e_k$ and~\eqref{def:ek}.
To get rid of multiplicative constants, we will choose $\d_k < \d_{k-1}$ at each step and take $|S_0|$ large enough.

\smallskip

In view of the above discussion, it is natural to define
\[
\K^- = \left\{ k\in \{1,\ldots,K\} \, | \, \frac 12 \left( 1 + 3\theta_k \right) \leq \frac 1{43} \right\} \ \m{and}\
\K^+ = \left\{ k\in \{1,\ldots,K\} \, | \, \frac 12 \left( 1 + 3\theta_k \right) > \frac 1{43} \right\}.
\]
We observe that $\K^+$ contains $1$ since $\theta_1 = 0$ and we denote by $d = |\K^+| \geq 1$ the cardinal of~$\K^+$.
The set $\K^-$ may be empty or not, depending on the choice of the parameters $\{\ell_k\}_k$ and $\{\eps_k\}_k$.
The proof is the same in both cases.
Let
\[
\d^+ = \min_{k\in \K^+} \frac 12 (1 + 3\theta_k) > \frac 1{43},\quad
\d_0 = \min\left\{ \frac 1{42},\d^+ \right\},\quad \d_{K+1} = \frac 1{43}.
\]
For all $1\leq k\leq K$, we define $\d_k^-$, $\d_k$ and $\d_k^+$ so that
\be \label{prop:d}
\frac 1{43} = \d_{K+1} < \d_{K+1}^+ \leq \d_{k+1}^+ < \d_k^- < \d_k < \d_k^+ < \d_{k-1}^-
\leq \d_0^- < \d_0 \leq \frac 1{42}.
\ee
Indeed, for technical reasons, it will be convenient to have several orders
of smallness in $|s|^{-1}$ for each $1\leq k\leq K$.

\subsection{The bootstrap setting} \label{sec:bootstrap}

From the formal discussion in the previous section,
we define $\V_n(S_n) = \V[\Gab_n](S_n)$ as in Section~\ref{sec:V},
for parameters $\{\mu_k(S_n)\}_k$, $\{\tau_k(S_n)\}_k$, $\{y_k(S_n)\}_k$ and $\{a_k(S_n)\}_k$
chosen as
\[
\left\{
\begin{aligned}
\mu_k(S_n) &= \ell_k \m{ for } k\in \K^-,\quad \mu_k(S_n) = \ell_k \left( 1 + |S_n|^{-\d_k^+} \xi_k \right) \m{ for } k\in \K^+, \\
\tau_k(S_n) &= S_n \ell_k^{-3} (1 + |S_n|^{-\d_k} \zeta_k),\quad y_k(S_n) = 2 S_n \ell_k^{-2}, \\
a_k(S_n) &= \frac {\ell_k\mu_k^2(S_n)}{S_n} \left( \frac 12 + \theta_k \right) - \frac 1{2\tau_k(S_n)} - \frac {4r_k(S_n)}{\|Q\|_{L^1}},
\end{aligned}
\right.
\]
where $\XI = \{\xi_k\}_{k\in \K^+}$ and $\ZETA = \{\zeta_k\}_{1\leq k\leq K}$,
satisfying $(\XI,\ZETA)\in \B_{d+K}$, will be chosen later.

We claim the following consequences of these choices of initial parameters:
\be \label{initial}
\left\{
\begin{aligned}
\bar \mu_k(S_n) &= 0\ \m{ for } k\in \K^-,\quad \bar\mu_k(S_n) = |S_n|^{-\d_k^+} \xi_k\ \m{ for } k\in \K^+, \\
\bar\tau_k(S_n) &= |S_n|^{-\d_k} \zeta_k,\quad \bar y_k(S_n) = 0, \\
e_k(S_n) &= \ell_k \left( \frac 12 + \theta_k \right),\quad |a_k(S_n)| \leq \frac 12 |S_n|^{-1 - \d_k^-},
\end{aligned}
\right.
\ee
and in particular $|\bar \mu_k(S_n)| \leq |S_n|^{-\d_k^+}$ for any $1\leq k\leq K$.
Indeed, the values of $\bar \mu_k(S_n)$, $\bar \tau_k(S_n)$ and $\bar y_k(S_n)$
are straightforward consequences of their definition in~\eqref{apriori:mutauy}.
The value of $e_k(S_n)$ is a direct consequence of the choice of $a_k(S_n)$ and the definition~\eqref{def:ek} of $e_k$.

Finally, by the definition of $a_k(S_n)$, we observe that, for $k\in \K^+$,
\[
a_k(S_n) = \frac {\ell_k^3 (1 + |S_n|^{-\d_k^+} \xi_k)^2}{S_n} \left( \frac 12 + \theta_k \right)
- \frac {\ell_k^3}{2S_n} (1 + |S_n|^{-\d_k} \zeta_k)^{-1} - \frac {4r_k(S_n)}{\|Q\|_{L^1}}.
\]
Moreover, by~\eqref{eq:rkasymp}, and since $2\d_k^+ > \d_{k-1}$ and $\frac 1{42} > \d_{k-1}$ from~\eqref{prop:d}, we obtain
\[
\left| r_k(S_n) - \frac {\|Q\|_{L^1}}{4S_n} \ell_k^3\theta_k \left( 1 + \frac 12 \bar\mu_k(S_n) \right) \right|
\lesssim |S_n|^{-1 - \d_{k-1}}
\]
and thus
\[
\left| \frac {4r_k(S_n)}{\|Q\|_{L^1}} - \frac {\ell_k^3\theta_k}{S_n} \right|
\lesssim |S_n|^{-1 - \d_k^+}.
\]
It follows that $|a_k(S_n)| \lesssim |S_n|^{-1 - \d_k}$ and thus,
for $n$ large enough, $|a_k(S_n)| \leq \frac 12 |S_n|^{-1 - \d_k^-}$.
Noticing that the proof is the same for $k\in \K^-$, it concludes the proof of~\eqref{initial}.

\medskip

Let $v_n$ be the solution of~\eqref{rescaledkdv} with initial data $v_n(S_n) = \V_n(S_n)$.
From Lemma~\ref{lem:syst_diff} and~\eqref{initial}, assuming $n$ large enough, we may decompose $v_n$ as
\[
v_n(s,y) = \V_n(s,y) + \e_n(s,y)
\]
on $[S_n,S_n + \bar s_n]$ for some $\bar s_n > 0$, where $\V_n = \V[\Gab_n]$ is defined as in Section~\ref{sec:V}
with parameters $\{\mu_k(s)\}_k$, $\{\tau_k(s)\}_k$, $\{y_k(s)\}_k$, $\{a_k(s)\}_k$
and $\e_n$ satisfies the orthogonality conditions~\eqref{ortho}.
Note that, for simplicity of notation, we do not specify the dependency in $n$
of the parameters $\mu_k$, $\tau_k$, $y_k$ and $a_k$.
Finally, in view of applying~\eqref{est:fk}, we denote
\[
g_k(s) = \bar \mu_k(s) + \bar y_k(s) + \int_{S_n}^s (1 + \bar \mu_k(\tau)) \langle \e_n(\tau),A_k(\tau) \rangle \,d\tau.
\]
Note also that $\e_n(S_n) \equiv 0$ and $g_k(S_n) = \bar \mu_k(S_n)$ for all $1\leq k\leq K$.

\medskip

We work with the following bootstrap estimates, for $1\leq k\leq K$:
\be \label{eq:BS}
\left\{
\begin{aligned}
& \|\e_n(s)\|_{H^1} \leq |s|^{-\frac 12 - \d_{K+1}^+},\quad \|\e_n(s)\|_{H^1(y > y_k)} \leq |s|^{-\frac 12 - \d_k^+}, \\
& \int_{S_n}^s |\tau|^{1 + \frac 1{43}} \left( \|\pa_y\e_n(\tau)\|_{L^2_k}^2
+ \|\e_n(\tau)\|_{L^2_k}^2 \right) d\tau \leq |s|^{-2\d_k^+ + \frac 1{43}}, \\
& |\bar \mu_k(s)| + |\bar y_k(s)| \leq |s|^{-\d_k},\quad |a_k(s)| \leq |s|^{-1 - \d_k^-},
\end{aligned}
\right.
\ee
and
\be \label{eq:BS2}
\sum_{k\in \K^+} \left[ |s|^{\d_k^+} g_k(s) \right]^2
+ \sum_{k=1}^K \left[ |s|^{\d_k} \bar \tau_k(s) \right]^2 \leq 1.
\ee
From the application of Lemma~\ref{lem:syst_diff} above and continuity arguments,
these bootstrap estimates are satisfied on $[S_n,S_n + \bar s_n]$,
for a possibly smaller $\bar s_n > 0$.
Now let $S_0 < -1$ with $|S_0|$ large to be fixed later.
Assuming $|S_0|$ large enough so that $|S_0|^{-\d_{K+1}^+} \leq \min \left( \frac 12,\frac \alpha 2 \right)$
with $\alpha$ given by Lemma~\ref{lem:syst_diff}, we may set
\[
S_n^*(\XI,\ZETA) = \sup\left\{ S_n \leq \tau \leq S_0 \m{ such that~\eqref{eq:BS}
and~\eqref{eq:BS2} are satisfied for all } s\in [S_n,\tau] \right\}.
\]
When there is no risk of confusion, we will denote $S^*_n(\XI,\ZETA)$ simply by $S_n^*$.

Note for future reference that, as a direct consequence of Lemma~\ref{lem:epsPk}, \eqref{prop:d} and~\eqref{eq:BS},
we have, for $|S_0|$ large enough, for all $1\leq k\leq K$, for all $s\in [S_n,S_n^*]$,
\begin{gather}
|\langle \e_n,P_k \rangle| \lesssim |s|^{-\d_{k+1}^+}, \label{b:Jk} \\
\left| \frac d{d s} \langle \e_n,P_k \rangle - c s \dot a_k \right|
\lesssim \sum_{j\leq k} \|\e_n\|_{L^2_j} + |s|^{-1 - \d_{k+1}^+}. \label{eq:Jkvariation}
\end{gather}

Note also that, from the orthogonality conditions~\eqref{ortho} satisfied by $\e_n$ and the identity $\e_n(S_n) \equiv 0$,
we obtain by integration, for all $1\leq k\leq K$, for all $s\in [S_n,S_n^*]$,
\be \label{orthobis}
\langle \e_n(s),\Lambda_k Q_k(s) \rangle = \langle \e_n(s),(\cdot-y_k)\Lambda_k Q_k(s) \rangle
= \langle \e_n(s),Q_k(s) \rangle = 0.
\ee

Now we prove that, for all $n$ large, there exists at least one choice of $(\XI,\ZETA)$
such that the bootstrap estimates~\eqref{eq:BS}--\eqref{eq:BS2}
hold up to some time $S_0$, with $|S_0|$ large, independent of $n$.

\begin{proposition} \label{pr:boot}
There exists $S_0 < -1$, independent of $n$, such that, for all $n$ large enough,
there exists $(\XI,\ZETA) \in \B_{d+K}$ such that $S_n^*(\XI,\ZETA)=S_0$.
\end{proposition}

The proof of Proposition~\ref{pr:boot} is given in Sections~\ref{sec:monotonicity} to~\ref{sec:Brouwer}.
We argue by contradiction, assuming that $S_n^*(\XI,\ZETA)\in [S_n,S_0)$ for all $(\XI,\ZETA) \in \B_{d+K}$.
The way to reach a contradiction is first to strictly improve all estimates in~\eqref{eq:BS}
(provided that $|S_0|$ is large enough, independently of $n$),
and second to use a topological obstruction on a certain continuous map related to the definition of $S_n^*(\XI,\ZETA)$,
the bootstrap~\eqref{eq:BS2}, the definition~\eqref{def:tauk} of $\tau_k$ for all $1\leq k\leq K$,
and the differential inequality~\eqref{est:fk} for $k\in \K^+$.

\subsection{Monotonicity of global energies} \label{sec:monotonicity}

We define the mass related to $\e$ as
\[
\G_n = \int \e_n^2 + 2 \sum_k a_k \langle \e_n,P_k \rangle + c |s| \sum_k a_k^2,
\]
where $c > 0$ is defined in Lemma~\ref{lem:epsPk},
the energy related to $\e$ as
\[
\H_n = \frac 12 \int (\pa_y \e_n)^2 - \frac 16 \int \left[ (\V_n + \e_n)^6 - \V_n^6 - 6\V_n^5\e_n \right],
\]
and we estimate their time variation in the following lemma.

\begin{lemma} \label{lem:mass-energy}
There exists $C > 0$ such that, for all $s\in [S_n,S_n^*]$,
\be \label{g.mass}
\frac d{d s} \left( |s| \G_n \right)
\leq C |s| \left( \sum_j \|\e_n\|_{L^2_j}^2 + |s|^{-2 - \d_K^- - \d_{K+1}^+} \right)
\ee
and
\be \label{g.energy}
\left| \frac d{d s} \left( |s| \H_n \right) \right|
\leq C |s| \left( \sum_j \|\e_n\|_{L^2_j}^2 + |s|^{-1} \|\e_n\|_{L^2}^2 + |s|^{-2 - 2\d_K^-} \right).
\ee
\end{lemma}

\begin{proof}
We denote $\e_n$, $\V_n$, $\G_n$ and $\H_n$ simply by $\e$, $\V$, $\G$ and $\H$.
For brevity, we also use the notation $f_s = \pa_s f$ and $f_y = \pa_y f$.

\smallskip

\textbf{Proof of~\eqref{g.mass}.}
We first compute
\begin{align*}
\frac 1{|s|} \frac d{d s} \left( |s| \G \right)
&= - \frac 1{|s|} \left( \int \e^2 + 2 \sum_k a_k \langle \e,P_k \rangle
+ c |s| \sum_k a_k^2 \right) + 2 \int \e_s \e \\
&\quad + 2 \sum_k \dot a_k \langle \e,P_k \rangle + 2 \sum_k a_k \frac d{d s}
\langle \e,P_k \rangle + 2c|s| \sum_k \dot a_k a_k - c \sum_k a_k^2,
\end{align*}
which gives
\begin{align*}
\frac 1{|s|} \frac d{d s} \left( |s| \G \right)
&\leq - \frac 1{|s|} \int \e^2 + 2 \sum_k a_k \frac d{d s} \langle \e,P_k \rangle + 2c|s| \sum_k \dot a_k a_k \\
&\quad + C |s|^{-1} \sum_k |a_k \langle \e,P_k \rangle|
+ 2 \int \e_s \e + 2 \sum_k \dot a_k \langle \e,P_k \rangle.
\end{align*}
Then, by estimates~\eqref{eq:Jkvariation} and $|a_k| \leq |s|^{-1 - \d_k^-}$ from~\eqref{eq:BS}, we find
\begin{align}
\left| a_k \frac d{d s} \langle \e,P_k \rangle + c|s| \dot a_k a_k \right|
&\lesssim |s|^{-1 - \d_k^-} \left( \sum_{j\leq k} \|\e\|_{L^2_j} + |s|^{-1 - \d_{k+1}^+} \right) \nonumber \\
&\lesssim \sum_{j\leq k} \|\e\|_{L^2_j}^2 + |s|^{-2 - \d_k^- - \d_{k+1}^+}. \label{G1}
\end{align}
Similarly, from~\eqref{b:Jk}, we find
\be \label{G2}
|s|^{-1} |a_k \langle \e,P_k \rangle| \lesssim |s|^{-2 - \d_k^- - \d_{k+1}^+}.
\ee
Next, using the equation~\eqref{eq:eps} satisfied by $\e$,
the expression~\eqref{decompEV} of $\E_\V$, and integration by parts (recall that $\int (\Lambda \e )\e = 0$),
we find
\[
\int \e_s \e + \sum_k \dot a_k \langle \e,P_k \rangle
= \int \left[ (\V + \e)^5 - \V^5 \right] \e_y - \int \left( \sum_k \vec m_k \cdot \vec \M_k V_k
+ \sum_k \dot r_k R_k +\P \right) \e.
\]
For the first term, integrating by parts, we have
\[
\int \left[ (\V + \e)^5 - \V^5 \right] \e_y = - \int \left[ (\V + \e)^5 - \V^5 - 5\V^4\e \right] \V_y
\]
and, by~\eqref{simpleVk} and~\eqref{apriori:eps}, we obtain
\begin{align*}
&\left| \int \left[ (\V + \e)^5 - \V^5 - 5\V^4\e \right] \V_y \right|
\lesssim \|\V_y\|_{L^\infty} \left( \int |\e|^5 + \int |\V|^3 \e^2 \right) \\
&\lesssim \|\e\|_{H^1}^5 + \|\e\|_{L^2}^2 \sum_j \|V_j - Q_j\|_{L^\infty}^3 + \sum_j \|\e\|_{L_j^2}^2
\lesssim |s|^{-\frac 52} + |s|^{-\frac {13}4} + \sum_j \|\e\|_{L_j^2}^2.
\end{align*}
Second, again by~\eqref{simpleVk}, we note that
\[
\left| \int (\vec \M_k V_k) \e \right|
\leq \int \left| \vec \M_k(V_k - Q_k) \right| |\e| + \int \left| \vec \M_k Q_k \right| |\e|
\lesssim |s|^{-\frac 12} \|\e\|_{L^2} + \|\e\|_{L^2_k}.
\]
Thus, by~\eqref{est:mk} and~\eqref{eq:BS},
\begin{align*}
\left| \int (\vec m_k \cdot \vec \M_k V_k) \e \right|
& \lesssim \left( \|\e\|_{L^2_k} + \sum_j \|\e\|_{L_j^2}^2 + |s|^{-2} \right)
\left( |s|^{-\frac 12} \|\e\|_{L^2} + \|\e\|_{L^2_k} \right) \\
& \leq C \sum_j \|\e\|_{L_j^2}^2 + \frac 1{8K|s|} \|\e\|_{L^2}^2 + C |s|^{-4}.
\end{align*}
Third, by~\eqref{est:dotrk} and the exponential decay of $R\in\Y$,
\[
\left| \dot r_k \int R_k \e \right|
\lesssim \left( |s|^{-1} \sum_j \|\e\|_{L^2_j} + |s|^{-2} \right) \|\e\|_{L^2_k}
\lesssim \sum_j \|\e\|_{L^2_j}^2 + |s|^{-4}.
\]
Finally, by~\eqref{est:Psi} and~\eqref{eq:BS}, we find
\[
\left| \int \P \e \right| \leq \|\P\|_{L^2} \|\e\|_{L^2}
\lesssim |s|^{-\frac 32 - \d_K^-} \|\e\|_{L^2}
\leq C|s|^{-2 - 2\d_K^-} + \frac 1{8|s|} \|\e\|_{L^2}^2.
\]
Gathering the above estimates, we have proved~\eqref{g.mass}.

\smallskip

\textbf{Proof of~\eqref{g.energy}.}
We proceed similarly, and first compute
\begin{align*}
\frac 1{|s|} \frac d{d s} \left( |s| \H \right)
&= -\frac 1{2|s|} \int \e_y^2 + \frac 1{6|s|} \int \left[ (\V + \e)^6 - \V^6 - 6\V^5\e \right] \\
&\quad - \int \e_s \left[ \e_{yy} + (\V + \e)^5 - \V^5 \right] - \int \V_s \left[ (\V + \e)^5 - \V^5 - 5\V^4\e \right].
\end{align*}
Using the equation~\eqref{eq:eps} of $\e$ and the definition of $\E_\V$ in Lemma~\ref{lem:V}, we obtain
\begin{align*}
\frac 1{|s|} \frac d{d s} \left( |s| \H \right)
&= \frac 1{2s} \int \e_y^2 - \frac 1{6s} \int \left[ (\V + \e)^6 - \V^6 - 6\V^5\e \right] + \frac 1{2s} \int (\Lambda\e) \e_{yy} \\
&\quad + \frac 1{2s} \int \Lambda \e \left[ (\V + \e)^5 - \V^5 \right] + \frac 1{2s} \int \Lambda \V \left[ (\V + \e)^5 - \V^5 - 5\V^4\e \right] \\
&\quad + \int \pa_y (\V_{yy} + \V^5) \left[ (\V + \e)^5 - \V^5 - 5\V^4\e \right] + \int \E_\V (\e_{yy} + 5\V^4\e).
\end{align*}
But, from integrations by parts, we get $\int (\Lambda\e) \e_{yy} = -\int \e_y^2$ and
\[
\int \Lambda \e \left[ (\V + \e)^5 - \V^5 \right] = \frac 13 \int \left[ (\V + \e)^6 - \V^6 - 6\V^5\e \right]
- \int \Lambda \V \left[ (\V + \e)^5 - \V^5 - 5\V^4\e \right].
\]
Thus, we have obtained the simplified expression
\[
\frac 1{|s|} \frac d{d s} \left( |s| \H \right)
= \int \left( \V_{yyy} + 5\V_y\V^4 \right) \left[ (\V + \e)^5 - \V^5 - 5\V^4\e \right]
+ \int \left( \pa_{yy}\E_\V + 5\V^4\E_\V \right) \e.
\]
For the first term, we obtain as before
\[
\left| \int \left( \V_{yyy} + 5\V_y\V^4 \right) \left[ (\V + \e)^5 - \V^5 - 5\V^4\e \right] \right|
\lesssim \int \left( |\e|^5 + |\V|^3 \e^2 \right) \lesssim |s|^{-\frac 52} + \sum_j \|\e\|_{L_j^2}^2.
\]
For the second term, we use the expression~\eqref{decompEV} of $\E_\V$ and get
\begin{align*}
\int \left( \pa_{yy}\E_\V + 5\V^4\E_\V \right) \e
&= \sum_k \int \vec m_k \cdot \left( \pa_{yy} \vec \M_k V_k + 5\V^4\vec \M_k V_k \right) \e
+ \sum_k \dot r_k \int \left( \pa_{yy}R_k + 5\V^4 R_k \right) \e \\
&\quad + \sum_k \dot a_k \int \left( \pa_{yy}P_k + 5\V^4 P_k \right) \e + \int \left( \P_{yy} + 5\V^4 \P \right) \e \\
&= h_1 + h_2 + h_3 + h_4.
\end{align*}
To estimate these four terms, we follow the proof of~\eqref{g.mass} above, and obtain similarly
\begin{align*}
|h_1| &\lesssim \sum_k |\vec m_k| \left( |s|^{-\frac 12} \|\e\|_{L^2} + \|\e\|_{L_k^2} \right)
\lesssim \sum_j \|\e\|_{L_j^2}^2 + |s|^{-1} \|\e\|_{L^2}^2 + |s|^{-4}, \\
|h_2| &\lesssim \sum_k |\dot r_k| \|\e\|_{L_k^2} \lesssim \sum_j \|\e\|_{L_j^2}^2 + |s|^{-4}, \\
|h_4| &\lesssim \|\P\|_{H^2} \|\e\|_{L^2} \lesssim |s|^{-\frac 32 - \d_K^-} \|\e\|_{L^2}
\lesssim |s|^{-1} \|\e\|_{L^2}^2 + |s|^{-2 - 2\d_K^-}.
\end{align*}
Finally, to estimate $h_3$, we use~\eqref{apriori:eps} and~\eqref{est:dotak} to obtain
\begin{align*}
|h_3| &\lesssim \sum_k |\dot a_k| \left( \|\pa_{yy} P_k\|_{L^2}
+ \|P_k\|_{L^\infty} \|\V\|_{L^\infty}^3 \|\V\|_{L^2} \right) \|\e\|_{L^2} \\
&\lesssim \|\e\|_{L^2} \left( \sum_k |\dot a_k| \right)
\lesssim |s|^{-\frac 12} \left( \sum_j \|\e\|_{L_j^2}^2 + |s|^{-2} \right)
\lesssim \sum_j \|\e\|_{L_j^2}^2 + |s|^{-\frac 52}.
\end{align*}
Gathering the above estimates, we obtain~\eqref{g.energy},
which concludes the proof of Lemma~\ref{lem:mass-energy}.
\end{proof}

\subsection{Monotonicity of local energies}

Let $\psi,\ph\in\Cinfini$ be nondecreasing functions such that
\[
\psi(y) = \left\{ \begin{aligned} e^y \quad & \m{for } y < -1, \\ 1 \quad & \m{for } y > -\frac 12, \end{aligned} \right.
\qquad \m{and} \qquad
\ph(y) = \left\{ \begin{aligned} e^y \quad & \m{for } y < -1, \\ y + 1 \quad & \m{for } -\frac 12 < y < \frac 12, \\
2 - e^{-y}\quad & \m{for } y > 1. \end{aligned} \right.
\]
We note that such functions satisfy $\frac 12 e^y \leq \psi(y) \leq 3e^y$
and $\frac 13 e^y \leq \ph(y) \leq 3e^y$ for $y<0$,
and $\frac 12 \ph(y) \leq \psi(y) \leq 3\ph(y)$ for all $y\in\R$.
Moreover, we may choose the function $\ph$ such that $\frac 14 \leq \ph'(y) \leq 1$ for $y\in[-1,1]$
and so $\frac 13 e^{-|y|} \leq \ph'(y) \leq 3e^{-|y|} $ for all $y\in\R$.

For $B > 100(1 + \ell_1)$ large to be chosen later and $1\leq k\leq K$, we define
\[
\psi_k(s,y) = \psi\left( \frac {y-y_k(s)}B \right) \quad\m{and}\quad
\ph_k(s,y) = \ph\left( \frac {y-y_k(s)}B \right).
\]
Note that, directly from the definitions of $\psi$ and $\ph$, we have, for all $y\in\R$,
\be \label{cut:onR}
\left\{
\begin{gathered}
\frac 13 e^{-\frac 1B |y - y_k|} \leq B\pa_y\ph_k(y) \leq 3e^{-\frac 1B |y - y_k|} \leq 9\ph_k(y),\quad
\frac 12 \ph_k(y) \leq \psi_k(y) \leq 3\ph_k(y), \\
\pa_y\psi_k(y) + B^2|\pa_{yyy}\psi_k(y)| + B^2|\pa_{yyy}\ph_k(y)| \lesssim \pa_y\ph_k(y),
\end{gathered}
\right.
\ee
and, for all $y<y_k$,
\be \label{cut:left}
\frac 12 e^{\frac 1B (y - y_k)} \leq \psi_k(y) \leq 3 e^{\frac 1B (y - y_k)} \quad\m{ and }\quad
\frac 13 e^{\frac 1B (y - y_k)} \leq \ph_k(y) \leq 3 e^{\frac 1B (y - y_k)}.
\ee
In particular, from the definition of the $L^2_k$ norm in Section~\ref{sec:V}, we have the control
\be \label{eL2k}
\|\e_n\|_{L^2_k}^2 \lesssim B \int \e_n^2 \pa_y\ph_k.
\ee

Similarly as in~\cite{CM1,MMR1}, we define now the mixed energy--virial functional
\[
\F_{k,n} = \int (\pa_y\e_n)^2 \psi_k + \mu_k^{-2} \int \e_n^2 \ph_k
- \frac 13 \int \left[ (\V_n + \e_n)^6 - \V_n^6 - 6\V_n^5\e_n \right] \psi_k,
\]
and we estimate it, as long as its time variation, in the following lemma.

\begin{lemma} \label{lem:enerloc}
There exist $B > 100$, $C > 0$ and $\kappa > 0$ such that,
for all $1\leq k\leq K$ and all $s\in [S_n,S_n^*]$, the following hold.
\begin{enumerate}[label=\emph{(\roman*)}]
\item \emph{Almost coercivity of $\F_{k,n}$:}
\be \label{eq:Fcoercivity}
\F_{k,n} \geq \kappa \left[ \int (\pa_y\e_n)^2 \psi_k + \int \e_n^2 \ph_k \right] - |s|^{-10}.
\ee
\item \emph{Time variation of $\F_{k,n}$:}
\begin{multline} \label{eq:Fvariation}
\frac d{d s} \left\{ |s|^{1 + \frac 1{43}} \left[ \F_{k,n} + 4 \mu_k^{-2} \sum_{j<k} a_j \langle \e_n,P_j \rangle
+ 2c \mu_k^{-2} |s| \sum_{j<k} a_j^2 \right] \right\} + \kappa |s|^{1 + \frac 1{43}}
\left[ \|\pa_y\e_n\|_{L^2_k}^2 + \|\e_n\|_{L^2_k}^2 \right] \\
\leq C |s|^{1 + \frac 1{43}} \left[ \sum_{j<k} \|\e_n\|_{L^2_j}^2
+|s|^{-\frac 12} \sum_j \|\e_n\|_{L^2_j}^2 + |s|^{-2 - \d_k^+ - \d_{k-1}^-} \right].
\end{multline}
\end{enumerate}
\end{lemma}

\begin{proof}
We denote again $\e_n$, $\V_n$ and $\F_{k,n}$ simply by $\e$, $\V$ and $\F_k$.
As in the proof of Lemma~\ref{lem:mass-energy},
we also use the notation $f_s = \pa_s f$ and $f_y = \pa_y f$.

\smallskip

(i) To prove~\eqref{eq:Fcoercivity}, we rely on the coercivity of the linearized energy~\eqref{coercivity},
together with the choice of orthogonality conditions~\eqref{orthobis} and standard localization arguments
(we refer to the proof of Lemma~4 in~\cite{MMT} for similar arguments).
We first decompose $\F_k$ as
\begin{align*}
\F_k &= \int \left[ \e_y^2 \psi_k + \tilde \mu_k^{-2} \e^2 \ph_k - 5 \e^2 \left( \sum_{j\leq k} Q_j^4 \right) \psi_k \right]
+ (\mu_k^{-2} - \tilde\mu_k^{-2}) \int \e^2 \ph_k \\
&\quad- 5 \int \left( \V^4 - \sum_{j\leq k} Q_j^4 \right) \e^2 \psi_k
- \frac 13 \int \left[ (\V + \e)^6 - \V^6 - 6\V^5\e - 15\V^4\e^2 \right] \psi_k.
\end{align*}

To estimate the first term, we claim that, for some $\bar \kappa > 0$ and for $B$ large enough,
\[
\int \left[ \e_y^2 \psi_k + \tilde\mu_k^{-2} \e^2 \ph_k - 5\e^2 \left( \sum_{j\leq k} Q_j^4 \right) \psi_k \right]
\geq \bar \kappa \int \left( \e_y^2 \psi_k + \e^2 \ph_k \right).
\]
Proceeding as in the Appendix~A of~\cite{MMannals} (see also Claim~(29) in the proof of Lemma~4 in~\cite{MMT}),
we recall without proof a localized version of~\eqref{coercivity}.
Consider a smooth even function~$\Phi$ such that $\Phi' \leq 0$ on $[0,+\infty)$,
$\Phi \equiv 1$ on $[0,1]$, $\Phi(x) = e^{-x}$ for $x\in [2,+\infty)$
and $e^{-x} \leq \Phi(x) \leq 3e^{-x}$ for $x\in [0,+\infty)$.
For $D>0$, let $\Phi_D(x) = \Phi(\frac xD)$.
Then there exists $0 < \kappa_1 < 1$ such that, for $D$ large enough, for all $\eta\in H^1(\R)$
such that $\langle \eta,Q \rangle = \langle \eta,\Lambda Q \rangle = \langle \eta,y\Lambda Q \rangle = 0$, there holds
\[
(1 - \kappa_1) \int \left( \eta_x^2 + \eta^2 \right) \Phi_D \geq 5 \int \eta^2 Q^4 \Phi_D.
\]
For $1\leq j\leq k$, let $\Phi_j(s,y) = \Phi_D(x)$ and define $\eta_j$ such that
$\e(s,y) = \tilde\mu_j^{-\frac 12} \eta_j(s,x)$, with $x = \tilde\mu_j^{-1} (y - y_j)$.
Then, since $\langle \eta_j,Q \rangle = \langle \eta_j,\Lambda Q \rangle = \langle \eta_j,y\Lambda Q \rangle = 0$
by~\eqref{orthobis}, we may apply the above estimate to $\eta_j$ and obtain
\[
(1 - \kappa_1) \int \left( \e_y^2 + \tilde\mu_j^{-2} \e^2 \right) \Phi_j \geq 5 \int \e^2 Q_j^4 \Phi_j.
\]
Moreover, we have $\tilde\mu_j^{-2} \leq \tilde\mu_k^{-2}$ for $1\leq j\leq k$,
and thus, summing over $1\leq j\leq k$, we find
\[
(1 - \kappa_1) \int \left( \e_y^2 + \tilde\mu_k^{-2} \e^2 \right) \left( \sum_{j\leq k} \Phi_j \right)
\geq 5 \int \e^2 \left( \sum_{j\leq k} Q_j^4 \Phi_j \right).
\]
But, by the definitions of $\psi$, $\ph$ and $\Phi$, we have
\[
\sum_{j\leq k} \Phi_j \leq (1 + |s|^{-10})\psi_k \quad\m{and}\quad
\sum_{j\leq k} \Phi_j \leq (1 + |s|^{-10})\ph_k,
\]
and, using moreover the exponential decay of $Q$,
\[
\sum_{j\leq k} Q_j^4 \Phi_j \geq \left( \sum_{j\leq k} Q_j^4 \right) \psi_k
- \left( |s|^{-10} + B^{-10} + D^{-10} \right) \ph_k.
\]
Thus, for $D$, $B$ and $|S_0|$ large enough, we deduce that
\[
\left( 1 - \frac {\kappa_1}2 \right) \int \left( \e_y^2 \psi_k + \tilde\mu_k^{-2} \e^2 \ph_k \right)
\geq 5 \int \e^2 \left( \sum_{j\leq k} Q_j^4 \right) \psi_k,
\]
which proves the claim with $\bar \kappa = \frac {\kappa_1}2 > 0$.

To estimate the second term, we use~\eqref{apriori:mutilde} to obtain, for $|S_0|$ large enough,
\[
\left| (\mu_k^{-2} - \tilde\mu_k^{-2}) \int \e^2 \ph_k \right|
\lesssim |s|^{-1} \int \e^2 \ph_k \leq \frac {\bar\kappa}4 \int \e^2 \ph_k.
\]
To estimate the third term, we use~\eqref{simpleVk} and~\eqref{cut:onR} to obtain
\[
5 \left| \int \left( \V^4 - \sum_{j\leq k} Q_j^4 \right) \e^2 \psi_k \right|
\lesssim |s|^{-\frac 34} \int \e^2 \ph_k + \sum_{j>k} \int Q_j^4 \e^2 \psi_k.
\]
Moreover, for $j>k$, we have, from~\eqref{apriori:eps}, \eqref{cut:onR}--\eqref{cut:left} and the exponential decay of $Q$,
\[
\int Q_j^4 \e^2 \psi_k \lesssim \int_{y < y_k - |s|^{\frac 14}} \e^2 \psi_k
+ \int_{y > y_k - |s|^{\frac 14}} Q_j^4 \e^2 \ph_k
\lesssim |s|^{-11} + |s|^{-10} \int \e^2 \ph_k.
\]
Thus, for $|S_0|$ large enough,
\[
5 \left| \int \left( \V^4 - \sum_{j\leq k} Q_j^4 \right) \e^2 \psi_k \right|
\lesssim |s|^{-\frac 34} \int \e^2 \ph_k + |s|^{-11}
\leq \frac {\bar\kappa}4 \int \e^2 \ph_k + |s|^{-10}.
\]
Finally, the nonlinear term is estimated as
\[
\frac 13 \left| \int \left[ (\V + \e)^6 - \V^6 - 6\V^5\e - 15\V^4\e^2 \right] \psi_k \right|
\lesssim \|\e\|_{L^\infty} \int \e^2 \ph_k \leq \frac {\bar\kappa}4 \int \e^2 \ph_k,
\]
by choosing again $|S_0|$ large enough.
Gathering the above estimates and letting $\kappa = \frac {\bar\kappa}4 > 0$, we obtain~\eqref{eq:Fcoercivity}.

\medskip

(ii) The proof is partly similar to the proof of Lemma~3.4 in~\cite{CM1} (see also Proposition~3.1 in~\cite{MMR1}).
We first compute
\begin{align*}
&\frac 1{|s|^{1 + \frac 1{43}}} \frac d{d s} \left\{ |s|^{1 + \frac 1{43}}
\left[ \F_k + 4 \mu_k^{-2} \sum_{j<k} a_j \langle \e,P_j \rangle + 2c \mu_k^{-2} |s| \sum_{j<k} a_j^2 \right] \right\}
= \frac {d \F_k}{d s} + 4 \mu_k^{-2} \sum_{j<k} \dot a_j \langle \e,P_j \rangle \\
&+ 4 \mu_k^{-2} \left( \sum_{j<k} a_j \frac d{d s} \langle \e,P_j \rangle + c|s| \sum_{j<k} \dot a_j a_j \right)
- 4 \dot \mu_k \mu_k^{-3} \left( 2 \sum_{j<k} a_j \langle \e,P_j \rangle + c|s| \sum_{j<k} a_j^2 \right) \\
&- 2c \mu_k^{-2} \sum_{j<k} a_j^2 - \left( 1 + \frac 1{43} \right) \frac 1{|s|}
\left[ \F_k + 4 \mu_k^{-2} \sum_{j<k} a_j \langle \e,P_j \rangle + 2c \mu_k^{-2} |s| \sum_{j<k} a_j^2 \right].
\end{align*}
Now note that, by combining~\eqref{def:mod}, \eqref{apriori:eps} and~\eqref{est:mk}, we have
$|\dot \mu_k| \lesssim \|\e\|_{L^2_k} + |s|^{-1}$.
Thus, using also~\eqref{G1} and~\eqref{G2}, and the estimate on $|a_j|$ from~\eqref{eq:BS}, we find
\begin{align*}
&\frac 1{|s|^{1 + \frac 1{43}}} \frac d{d s} \left\{ |s|^{1 + \frac 1{43}}
\left[ \F_k + 4 \mu_k^{-2} \sum_{j<k} a_j \langle \e,P_j \rangle
+ 2c \mu_k^{-2} |s| \sum_{j<k} a_j^2 \right] \right\}
\leq - \left( 1 + \frac 1{43} \right) \frac {\F_k}{|s|} + \frac {d \F_k}{d s} \\
&\quad + 4 \mu_k^{-2} \sum_{j<k} \dot a_j \langle \e,P_j \rangle
+ C |s|^{-1 - \d_{k-1}^- - \d_k^+} \|\e\|_{L^2_k}
+ C \sum_{j< k} \|\e\|_{L_j^2}^2 + C |s|^{-2 - \d_{k-1}^- - \d_k^+}.
\end{align*}
Next, denoting $G_k(\e) = \Big\{ -\e_y \pa_y\psi_k - \e_{yy} \psi_k + \mu_k^{-2} \e \ph_k - [(\V + \e)^5 - \V^5] \psi_k \Big\}$,
an integration by parts gives, since $\pa_s\psi_k = -\dot y_k \pa_y\psi_k$ and $\pa_s\ph_k = -\dot y_k \pa_y\ph_k$,
\begin{align*}
\frac {d \F_k}{d s} &= - 2 \int \left( \V_s + \dot y_k \V_y \right)
\left[ (\V + \e)^5 - \V^5 - 5\V^4\e \right] \psi_k \\
&\quad + 2 \int \left( \e_s + \dot y_k \e_y \right) G_k(\e)
- 2 \dot\mu_k \mu_k^{-3} \int \e^2 \ph_k.
\end{align*}
Therefore, using the equation~\eqref{eq:eps} satisfied by $\e$, the identity
$\Lambda_k = \frac 12 + (y-y_k)\pa_y = \Lambda - y_k\pa_y$, and the definition of $\E_\V$ in Lemma~\ref{lem:V},
we find
\[
- \left( 1 + \frac 1{43} \right) \frac {\F_k}{|s|}
+ \frac {d \F_k}{d s} + 4 \mu_k^{-2} \sum_{j<k} \dot a_j \langle \e,P_j \rangle
= f_1 + f_2 + f_3 + f_4
\]
where, denoting $\Gt_k(\e) = \Big\{ -\e_y \pa_y\psi_k - \e_{yy} \psi_k + \mu_k^{-2} \e \ph_k - 5\V^4\e \psi_k \Big\}$,
we have set
\begin{align*}
f_1 &= -2 \int \pa_y \left[ \e_{yy} - \mu_k^{-2} \e + (\V + \e)^5 - \V^5 \right] G_k(\e), \\
f_2 &= - \left( 1 + \frac 1{43} \right) \frac {\F_k}{|s|} + \frac 1{|s|} \int (\Lambda_k\e) G_k(\e)
- \frac 1{|s|} \int (\Lambda_k \V) \left[ (\V + \e)^5 - \V^5 - 5\V^4\e \right] \psi_k, \\
f_3 &= -2 \int \E_\V \Gt_k(\e) + 4 \mu_k^{-2} \sum_{j<k} \dot a_j \langle \e,P_j \rangle
- 2\dot\mu_k \mu_k^{-3} \int \e^2 \ph_k, \\
f_4 &= 2 \int \pa_y \left[ \V_{yy} + \V^5 - \left( \dot y_k - \frac {y_k}{2s} \right) \V \right]
\left[ (\V + \e)^5 - \V^5 - 5\V^4\e \right] \psi_k \\
&\quad + 2 \left( \dot y_k - \frac {y_k}{2s} - \frac 1{\mu_k^2} \right) \int \e_y G_k(\e).
\end{align*}
We now estimate these four terms separately.

\smallskip

\textbf{Control of $f_1$.}
From multiple integrations by parts, we may rewrite $f_1$ as
\begin{align*}
f_1 &= -\int \left[ 3\e_{yy}^2 \pa_y\psi_k + \e_y^2 \left( 3\mu_k^{-2} \pa_y\ph_k
+ \mu_k^{-2} \pa_y\psi_k - \pa_{yyy}\psi_k \right) + \e^2 \left( \mu_k^{-4} \pa_y\ph_k - \mu_k^{-2} \pa_{yyy}\ph_k \right) \right] \\
&\quad - \frac 13 \mu_k^{-2} \int \left[ (\V + \e)^6 - \V^6 - 6\V^5\e
-6 \left( (\V + \e)^5 - \V^5 \right) \e \right] \Bigl( \pa_y\ph_k - \pa_y\psi_k \Bigr) \\
&\quad - 2\mu_k^{-2} \int \V_y \left[ (\V + \e)^5 - \V^5 - 5\V^4\e \right] (\ph_k - \psi_k) \\
&\quad + 10 \int \e_y \left[ \V_y \left( (\V + \e)^4 - \V^4 \right) + \e_y (\V + \e)^4 \right] \pa_y\psi_k \\
&\quad - \int \left\{ \left[ -\e_{yy} + \mu_k^{-2} \e - \left( (\V + \e)^5 - \V^5 \right) \right]^2
- \left[ -\e_{yy} + \mu_k^{-2} \e \right]^2 \right\} \pa_y\psi_k
= f_1^< + f_1^\sim + f_1^>,
\end{align*}
where $f_1^<$, $f_1^\sim$ and $f_1^>$ respectively correspond to integration
on $y - y_k < -\frac B2$, $|y - y_k| \leq \frac B2$ and $y - y_k > \frac B2$.

\emph{Estimate of $f_1^<$.}
From~\eqref{apriori:mutauy} and~\eqref{cut:onR}--\eqref{cut:left},
taking $B$ large enough (depending on $\ell_k$), we find
\begin{align*}
f_1^< &\leq -3 \int_{y - y_k < -\frac B2} \e_{yy}^2 \pa_y\psi_k
- \frac 12 \mu_k^{-2} \int_{y - y_k < -\frac B2} \left(\e_y^2 + \mu_k^{-2} \e^2 \right) \pa_y\ph_k \\
&\quad + C \int_{y - y_k < -\frac B2} (\e^6 + \V^4 \e^2) \pa_y\ph_k
+ C B \int_{y - y_k < -\frac B2} |\V_y| (|\e|^5 + |\V|^3\e^2) \pa_y\ph_k \\
&\quad + C \int_{y - y_k < -\frac B2} |\e_y| \left[ |\V_y| \left( |\V|^3 |\e| + \e^4 \right) + |\e_y| (\V^4 + \e^4) \right] \pa_y\psi_k \\
&\quad + C \int_{y - y_k < -\frac B2} \left( \V^4 |\e| + |\e|^5 \right) \left( |\e_{yy}| + |\e| + \V^4 |\e| + |\e|^5 \right) \pa_y\psi_k.
\end{align*}
Using~\eqref{simpleVk} we find, for $-|s|^{\frac 14} < y - y_k < - \frac B2$,
\[
|\V(y)| \lesssim e^{- \frac B{4\ell_k}} + |s|^{-\frac 34} \quad\m{ and }\quad
|\V_y(y)| \lesssim e^{- \frac B{4\ell_k}} + |s|^{-1}.
\]
Moreover, for $y - y_k < -|s|^{\frac 14}$, we have
$\pa_y\psi_k(y) + \pa_y\ph_k(y) \lesssim e^{-\frac 1B |s|^{\frac 14}}$ from~\eqref{cut:onR}.
Therefore,
\begin{align*}
f_1^< &\leq -3 \int_{y - y_k < -\frac B2} \e_{yy}^2 \pa_y\psi_k
- \frac 12 \mu_k^{-2} \int_{y - y_k < -\frac B2} \left( \e_y^2 + \mu_k^{-2} \e^2 \right) \pa_y\ph_k \\
&\quad + C_0 B \left( \|\e\|_{L^\infty}^4 + e^{-\frac B{\ell_k}} + |s|^{-3} \right) \mu_k^{-2}
\int_{-|s|^{\frac 14} < y - y_k < -\frac B2} \left( \e_y^2 + \mu_k^{-2} \e^2 \right) \pa_y\ph_k \\
&\quad + C_1 \left( \|\e\|_{L^\infty}^4 + e^{-\frac B{\ell_k}} + |s|^{-3} \right)
\int_{-|s|^{\frac 14} < y - y_k < -\frac B2} \e_{yy}^2 \pa_y\psi_k \\
&\quad + C_2 |s|^{-10} \int_{y - y_k < -|s|^{\frac 14}} \e_{yy}^2 \pa_y\psi_k + C|s|^{-10}
\end{align*}
for some constants $C_0,C_1,C_2 > 0$.
We choose $B$ large enough, and then $|S_0|$ large enough (depending on $B$),
so that $C_2 |s|^{-10} \leq 1$,
\[
C_1 \left( \|\e\|_{L^\infty}^4 + e^{-\frac B{\ell_k}} + |s|^{-3} \right) \leq 1 \quad\m{and}\quad
C_0 B \left( \|\e\|_{L^\infty}^4 + e^{-\frac B {\ell_k}} + |s|^{-3}\right) \leq \frac 14.
\]
Hence, we obtain
\[
f_1^< \leq -\frac 14 \mu_k^{-2} \int_{y - y_k < -\frac B2} \left( \e_y^2 + \mu_k^{-2} \e^2 \right) \pa_y\ph_k + C|s|^{-10}.
\]

\emph{Estimate of $f_1^>$.}
For $y - y_k > \frac B2$, we have $\psi_k(y) = 1$ and so $\pa_y\psi_k(y) = \pa_{yyy}\psi_k(y) = 0$.
Thus,
\begin{align*}
f_1^> &= - \int_{y - y_k > \frac B2} \left[ 3\mu_k^{-2} \e_y^2 \pa_y\ph_k
+ \mu_k^{-4} \e^2 \left( \pa_y\ph_k - \mu_k^2 \pa_{yyy}\ph_k \right) \right] \\
&\quad - \frac 13 \mu_k^{-2} \int_{y - y_k > \frac B2} \left[ (\V + \e)^6 - \V^6 - 6\V^5\e
- 6 \left( (\V + \e)^5 - \V^5 \right) \e \right] \pa_y\ph_k \\
&\quad -2\mu_k^{-2} \int_{y - y_k > \frac B2} \V_y \left[ (\V + \e)^5 - \V^5 - 5\V^4\e \right] (\ph_k - 1).
\end{align*}
Using~\eqref{cut:onR} and $\|\V_y\|_{L^\infty} \lesssim 1$,
and taking $B$ large enough as before, we find
\begin{align*}
f_1^> &\leq - \frac 12 \mu_k^{-2} \int_{y - y_k > \frac B2} \left( \e_y^2 + \mu_k^{-2} \e^2 \right) \pa_y\ph_k
+ C \int_{y - y_k > \frac B2} \left( \e^6 + \V^4\e^2 \right) \pa_y\ph_k \\
&\quad + C \int_{y - y_k > \frac B2} \left( |\e|^5 + |\V|^3\e^2 \right).
\end{align*}
For $\frac B2 < y - y_k < |s|^{\frac 14}$, by~\eqref{simpleVk} and the exponential decay of $Q$, we find
$|\V(y)| \lesssim e^{-\frac B{4\ell_k}} + |s|^{-\frac 34}$.
For $y - y_k > |s|^{\frac 14}$, by~\eqref{cut:onR}, we have $\pa_y\ph_k(y) \lesssim e^{-\frac 1B |s|^{\frac 14}}$.
Thus, we obtain
\begin{align*}
f_1^> &\leq - \frac 12 \mu_k^{-2} \int_{y - y_k > \frac B2} \left( \e_y^2 + \mu_k^{-2} \e^2 \right) \pa_y\ph_k
+ C |s|^{-10} + C \|\e\|_{H^1}^5 + C \int_{y - y_k > \frac B2} |\V|^3 \e^2 \\
&\quad + C \left( \|\e\|_{L^\infty}^4 + e^{-\frac B{\ell_k}} + |s|^{-3} \right) \int_{\frac B2 < y - y_k < |s|^{\frac 14}} \e^2 \pa_y\ph_k.
\end{align*}
Then, using also~\eqref{apriori:eps}, we find as before, for $B$ large enough and $|S_0|$ large enough,
\[
f_1^> \leq - \frac 14 \mu_k^{-2} \int_{y - y_k > \frac B2} \left( \e_y^2 + \mu_k^{-2} \e^2 \right) \pa_y\ph_k
+ C \int_{y - y_k > \frac B2} |\V|^3 \e^2 + C |s|^{-\frac 52}.
\]
But, by~\eqref{simpleVk} and the exponential decay of $Q$, we have
\[
\left\| \V - \sum_{j\leq k} Q_j \right\|_{L^\infty(y - y_k > \frac B2)}
\lesssim \sum_{j > k} \|Q_j\|_{L^\infty(y - y_k > \frac B2)} + \sum_j \|V_j - Q_j\|_{L^\infty}
\lesssim |s|^{-\frac 34}.
\]
And since $\int |Q_j|^3 \e^2 \lesssim \|\e\|_{L^2_j}^2$ and
$\int_{y - y_k > \frac B2} |Q_k|^3 \e^2 \lesssim e^{-\frac B{4\ell_k}} \|\e\|_{L^2_k}^2 \lesssim B^{-2} \|\e\|_{L^2_k}^2$,
we find
\[
\int_{y - y_k > \frac B2} |\V|^3 \e^2 \lesssim \sum_{j < k} \|\e\|_{L^2_j}^2
+ B^{-2} \|\e\|_{L^2_k}^2 + |s|^{-\frac 94} \|\e\|_{L^2}^2.
\]
In conclusion for this term, we have
\[
f_1^> \leq - \frac 14 \mu_k^{-2} \int_{y - y_k > \frac B2} \left( \e_y^2 + \mu_k^{-2} \e^2 \right) \pa_y\ph_k
+ C \sum_{j<k} \|\e\|_{L_j^2}^2 + C B^{-2} \|\e\|_{L^2_k}^2 + C |s|^{-\frac 52}.
\]

\emph{Estimate of $f_1^\sim$.}
For $|y - y_k| \leq \frac B2$, we still have $\psi_k(y) = 1$ and $\pa_y\psi_k(y) = \pa_{yyy}\psi_k(y) = 0$.
Moreover, $\ph_k(y) = \frac {y - y_k}B + 1$, $\pa_y\ph_k(y) = \frac 1B$ and $\pa_{yyy}\ph_k(y) = 0$.
Thus, we may rewrite $f_1^\sim$ as
\begin{align*}
B f_1^\sim
& = - \mu_k^{-2} \int_{|y - y_k| \leq \frac B2} \left( 3\e_y^2 + \mu_k^{-2} \e^2 \right) \\
& \quad - \frac 13 \mu_k^{-2} \int_{|y - y_k| \leq \frac B2} \left[ (\V + \e)^6 - \V^6 - 6\V^5\e
- 6 \left( (\V + \e)^5 - \V^5 \right) \e \right] \\
& \quad -2\mu_k^{-2} \int_{|y - y_k| \leq \frac B2} (y - y_k) \V_y \left[ (\V + \e)^5 - \V^5 - 5\V^4\e \right] \\
& = -\mu_k^{-2} \int_{|y - y_k| \leq \frac B2} \left( 3\e_y^2
+ \tilde \mu_k^{-2} \e^2 - 5Q_k^4 \e^2 + 20 (y - y_k)\pa_y Q_k Q_k^3 \e^2 \right) + \mu_k^{-2} R_{\mathrm{Vir}}(\e),
\end{align*}
with
\begin{align*}
R_{\mathrm{Vir}}(\e)
& = (\tilde \mu_k^{-2} - \mu_k^{-2}) \int_{|y - y_k| \leq \frac B2} \e^2
+ \frac 13 \int_{|y - y_k| \leq \frac B2} \left( 40\V^3 + 45\V^2\e + 24\V\e^2 + 5\e^3 \right) \e^3 \\
& \quad + 5 \int_{|y - y_k| \leq \frac B2} \left[ \left( \V^4 - Q_k^4 \right)
- 4 (y-y_k) \left( \V_y \V^3 - \pa_y Q_k Q_k^3 \right) \right] \e^2 \\
& \quad - 2 \int_{|y - y_k| \leq \frac B2} (y-y_k) \V_y \left( 10\V^2 + 5\V\e + \e^2 \right) \e^3.
\end{align*}
With the change of space variable $x = \tilde\mu_k^{-1} (y-y_k)$, and letting
$\e(s,y)=\tilde\mu_k^{-\frac 12} \eta(s,x)$, we find
\begin{align*}
\int_{|y - y_k| \leq \frac B2} & \left( 3\e_y^2 + \tilde \mu_k^{-2} \e^2 - 5 Q_k^4 \e^2 + 20(y - y_k)\pa_y Q_k Q_k^3 \e^2 \right) \\
& = \tilde\mu_k^{-2} \int_{|x| \leq \tilde \mu_k^{-1} \frac B2} \left( 3\eta_x^2 + \eta^2 - 5Q^4 \eta^2 + 20 xQ'Q^3 \eta^2 \right)
\end{align*}
and, by~\eqref{orthobis},
\[
\langle \eta,\Lambda Q \rangle = \langle \eta,y\Lambda Q \rangle = \langle \eta,Q \rangle = 0.
\]
But, from Lemma~3.4 in~\cite{MMR1} (see also Proposition~4 in~\cite{MMjmpa}),
these orthogonality conditions ensure that,
for $B$ large enough (here depending on $\ell_k$) and for some $\kappa_0 > 0$,
\[
\int_{|x| \leq \tilde \mu_k^{-1} \frac B2} \left( 3\eta_x^2 + \eta^2 - 5Q^4 \eta^2 + 20 xQ'Q^3 \eta^2 \right)
\geq \kappa_0 \int_{|x| \leq \tilde \mu_k^{-1} \frac B2} \left( \eta_x^2 + \eta^2 \right)
- B^{-1} \int \eta^2 e^{-|x|/2}.
\]
Thus, going back to the original space variable, we find
\begin{align*}
\int_{|y - y_k| \leq \frac B2} & \left( 3\e_y^2 + \tilde\mu_k^{-2} \e^2 - 5Q_k^4 \e^2 + 20(y-y_k)\pa_y Q_k Q_k^3 \e^2 \right) \\
& \geq \kappa_0 \int_{|y - y_k| \leq \frac B2} \left( \e_y^2 + \tilde \mu_k^{-2} \e^2 \right) - B^{-1} \tilde\mu_k^{-2} \|\e\|_{L^2_k}^2.
\end{align*}
Now, we estimate $R_{\mathrm{Vir}}(\e)$.
By~\eqref{simpleVk} and the exponential decay of $Q$, we observe that
\[
\|\V - Q_k\|_{L^\infty(|y - y_k| \leq \frac B2)} \lesssim \sum_j \|V_j - Q_j\|_{L^\infty}
+ \sum_{j\neq k} \|Q_j\|_{L^\infty(|y - y_k| \leq \frac B2)} \lesssim |s|^{-\frac 34}
\]
and, similarly,
$\| \pa_y(\V - Q_k)\|_{L^\infty(|y - y_k| \leq \frac B2)} \lesssim |s|^{-1}$.
Using also~\eqref{apriori:mutilde} and $\|\e\|_{L^\infty} \lesssim |s|^{-\frac 12}$
from~\eqref{apriori:eps}, we thus find
\[
|R_{\mathrm{Vir}}(\e)| \lesssim B |s|^{-\frac 12} \int_{|y - y_k| \leq \frac B2} \e^2.
\]
In conclusion for this term, we have obtained, for $|S_0|$ large enough (depending on $B$),
\[
f_1^\sim \leq - \frac {\kappa_0}2 \mu_k^{-2} \int_{|y - y_k| \leq \frac B2}
\left( \e_y^2 + \mu_k^{-2} \e^2 \right) \pa_y\ph_k + C B^{-2} \|\e\|_{L^2_k}^2.
\]

\emph{Estimate of $f_1$.}
Gathering the above estimates of $f_1^<$, $f_1^\sim$ and $f_1^>$,
using~\eqref{eL2k}, assuming $\kappa_0 < \frac 12$ and $B$ large enough,
and finally letting $\kappa_1 = \frac {\kappa_0}4 > 0$, we obtain
\[
f_1 \leq - \kappa_1 \mu_k^{-2} \int \left( \e_y^2 + \mu_k^{-2} \e^2 \right) \pa_y\ph_k
+ C \sum_{j<k} \|\e\|_{L_j^2}^2 + C |s|^{-\frac 52}.
\]

\medskip

\textbf{Control of $f_2$.}
Using integrations by parts, we find the identities
\begin{align*}
-\int (\Lambda_k \e) \pa_y(\psi_k\e_y) &= \int \e_y^2 \psi_k - \frac 12 \int (y - y_k) \e_y^2 \pa_y\psi_k, \\
\int (\Lambda_k \e) \e \ph_k &= - \frac 12 \int (y - y_k) \e^2 \pa_y\ph_k,
\intertext{and}
-\int (\Lambda_k \e) \left[ (\V + \e)^5 - \V^5 \right] \psi_k
&= \frac 16 \int \left[ (y - y_k) \pa_y\psi_k - 2\psi_k \right] \left[ (\V + \e)^6 - \V^6 - 6\V^5\e \right] \\
&\quad + \int (\Lambda_k \V) \left[ (\V + \e)^5 - \V^5 - 5\V^4\e \right] \psi_k.
\end{align*}
Thus, from the definition of $\F_k$, we may rewrite $f_2$ as
\begin{align*}
f_2 &= -\frac {\F_k}{43|s|} - \frac {\mu_k^{-2}}{|s|} \int \e^2 \ph_k - \frac 12 |s|^{-1} \int (y - y_k) \e_y^2 \pa_y\psi_k
- \frac 12 \mu_k^{-2} |s|^{-1} \int (y - y_k) \e^2 \pa_y\ph_k \\
&\quad + \frac 16 |s|^{-1} \int (y - y_k) \left[ (\V + \e)^6 - \V^6 - 6\V^5\e \right] \pa_y\psi_k,
\end{align*}
and so, using~\eqref{cut:onR} and~\eqref{eq:Fcoercivity},
\[
f_2 \leq - \frac {\mu_k^{-2}}{|s|} \int \e^2 \ph_k
+ C |s|^{-1} \int |y - y_k| \left( \e_y^2 + \mu_k^{-2} \e^2 \right) \pa_y\ph_k + C|s|^{-10}.
\]
From~\eqref{cut:onR} again, if $|y - y_k| \geq |s|^{\frac 14}$, then
$|y - y_k| \pa_y\ph_k(y) \lesssim e^{-\frac 1{2B} |s|^{\frac 14}}$,
and so, using also~\eqref{apriori:eps},
\[
|s|^{-1} \int |y - y_k| \left( \e_y^2 + \mu_k^{-2} \e^2 \right) \pa_y\ph_k
\lesssim |s|^{-\frac 34} \int_{|y - y_k| \leq |s|^{\frac 14}} \left( \e_y^2 + \mu_k^{-2} \e^2 \right) \pa_y\ph_k + |s|^{-10}.
\]
In conclusion for this term, we have obtained, for $|S_0|$ large enough,
\[
f_2 \leq - \frac {\mu_k^{-2}}{|s|} \int \e^2 \ph_k
+ \frac {\kappa_1}{10} \mu_k^{-2} \int \left( \e_y^2 + \mu_k^{-2} \e^2 \right) \pa_y\ph_k + C |s|^{-10}.
\]

\medskip

\textbf{Control of $f_3$.}
From the expression~\eqref{decompEV} of $\E_\V$, we may rewrite $f_3$ as
\begin{align*}
\frac 12 f_3 &=\sum_j m_{j,1} \int (\Lambda_j V_j) \Gt_k(\e) + \sum_j m_{j,2} \int (\pa_y V_j) \Gt_k(\e)
- \sum_j \dot r_j \int R_j \Gt_k(\e) \\
&\quad + \left( 2\mu_k^{-2} \sum_{j<k} \dot a_j \langle \e,P_j \rangle - \sum_j \dot a_j \int P_j \Gt_k(\e) \right)
- \int \P \Gt_k(\e) - \dot\mu_k \mu_k^{-3} \int \e^2 \ph_k \\
&= f_{3,1} + f_{3,2} + f_{3,3} + f_{3,4} + f_{3,5} + f_{3,6}.
\end{align*}

\emph{Estimate of $f_{3,1}$.}
By integration by parts, we have
\[
\int (\Lambda_j V_j) \Gt_k(\e) = - \int \e \left[ \pa_{yy} (\Lambda_j V_j) \psi_k - \mu_k^{-2} (\Lambda_j V_j) \ph_k
+ 5\V^4 (\Lambda_j V_j) \psi_k + \pa_y(\Lambda_j V_j) \pa_y\psi_k \right].
\]
First, by $R\in \Y$, \eqref{b:Pky} and~\eqref{eq:Wkright},
we note that, for all $1\leq j\leq K$ and all $0\leq m\leq 19$,
$|\pa_y^m V_j(y)| \lesssim e^{-\rho_{j,m} (y - y_j)}$ for $y > y_j$.
In particular, for $j > k$, using~\eqref{cut:left} and the estimate $|m_{j,1}| \lesssim \|\e\|_{L^2} \lesssim |s|^{-\frac 12}$
from~\eqref{apriori:eps} and~\eqref{est:mk}, we find
\[
\left| m_{j,1} \int (\Lambda_j V_j) \Gt_k(\e) \right| \lesssim |s|^{-10}.
\]
Next, we focus on the case $j = k$. First, by~\eqref{simpleVk}, we obtain
\[
\int_{|y - y_k| \leq |s|^{\frac 14}} |\Lambda_k V_k - \Lambda _k Q_k|^2 \lesssim \left( \|V_k - Q_k\|_{L^\infty}^2
+ |s|^{\frac 12} \|\pa_y (V_k - Q_k)\|_{L^\infty}^2 \right) |s|^{\frac 14} \lesssim |s|^{-\frac 54}.
\]
Moreover, by the decay properties of $V_k$ and $Q_k$ for $y > y_k$ and by~\eqref{cut:left}, we have
\[
\int_{|y - y_k| > |s|^{\frac 14}} |\Lambda_k V_k - \Lambda _k Q_k|^2 \ph_k \lesssim |s|^{-10} \quad\m{and so}\quad
\int |\Lambda_k V_k - \Lambda_k Q_k|^2 \ph_k \lesssim |s|^{-\frac 54}.
\]
Similarly, we get $ \ds \int |\pa_{yy}(\Lambda_k V_k) - \pa_{yy}(\Lambda_k Q_k)|^2 \psi_k \lesssim |s|^{-\frac 54}$
and also
\[
\int |\V^4(\Lambda_k V_k) - Q_k^4(\Lambda_k Q_k)|^2 \psi_k
+ \int |\pa_y(\Lambda_k V_k) - \pa_y(\Lambda_k Q_k)|^2 \pa_y\psi_k \lesssim |s|^{-\frac 54}.
\]
We also observe that, by the definitions of $\ph_k$ and $\psi_k$ and the exponential decay of $Q$,
\begin{multline*}
\int |\Lambda_k Q_k| \left| \ph_k - \left( 1 + \frac {y-y_k}B \right) \right|^2
+ \int \left( |\pa_{yy}(\Lambda_k Q_k)| + |Q_k^4 (\Lambda_k Q_k)| \right) \left| \psi_k - 1 \right|^2 \\
+ \int |\pa_y(\Lambda_k Q_k)| \left| \pa_y\psi_k \right|^2 \lesssim e^{-\frac B{2\ell_k}}.
\end{multline*}
Therefore, by the Cauchy--Schwarz inequality,
\begin{multline*}
\left| \int (\Lambda_k V_k) \Gt_k(\e) - \int \e \left[ \pa_{yy}(\Lambda_k Q_k)
- \mu_k^{-2} \left( 1 + \frac {y - y_k}B \right) (\Lambda_k Q_k) + 5Q_k^4 (\Lambda_k Q_k) \right] \right| \\
\lesssim |s|^{-\frac 58} \left( \int \e^2 \ph_k \right)^{\frac 12} + e^{-\frac B{4\ell_k}} \|\e\|_{L^2_k}.
\end{multline*}
Now note that, since $L (\Lambda Q) = - 2Q$ from Lemma~\ref{lem:L}, we have the identity
\[
\pa_{yy}(\Lambda_k Q_k) - \tilde\mu_k^{-2} (\Lambda_k Q_k) + 5Q_k^4 (\Lambda_k Q_k) = 2\tilde\mu_k^{-2} Q_k.
\]
Thus, from the three orthogonality conditions~\eqref{orthobis},
the second term in the last estimate cancels and we find
\[
\left| \int (\Lambda_k V_k) \Gt_k(\e) \right| \lesssim |s|^{-\frac 58}
\left( \int \e^2 \ph_k \right)^{\frac 12} + e^{-\frac B{4\ell_k}} \|\e\|_{L^2_k}.
\]
Using also~\eqref{apriori:eps} and~\eqref{est:mk}, it follows that
\begin{align*}
\left| m_{k,1} \int (\Lambda_k V_k) \Gt_k(\e) \right|
& \lesssim \left[ \|\e\|_{L^2_k} + \sum_{j'} \|\e\|_{L^2_{j'}}^2 + |s|^{-2} \right]
\left[ |s|^{-\frac 58} \left( \int \e^2 \ph_k \right)^{\frac 12} + e^{-\frac B{4\ell_k}} \|\e\|_{L^2_k} \right] \\
& \lesssim \left( e^{-\frac B{4\ell_k}} + |s|^{-\frac 18} \right) \|\e\|_{L^2_k}^2
+ |s|^{-\frac 98} \int \e^2 \ph_k + |s|^{-\frac 12} \sum_{j'} \|\e\|_{L^2_{j'}}^2 + |s|^{-4}.
\end{align*}
Finally, for $j < k$, we observe that, by~\eqref{simpleVk},
$\|\Lambda_j V_j - \Lambda_j Q_j\|_{L^2} \lesssim |s|^{-\frac 12}$,
and similarly
\[
\|\pa_{yy}(\Lambda_j V_j) - \pa_{yy}(\Lambda_j Q_j)\|_{L^2}
+ \|\V^4 (\Lambda_j V_j) - Q_j^4 (\Lambda_j Q_j)\|_{L^2}
+ \|\pa_y(\Lambda_j V_j) \pa_y\psi_k\|_{L^2} \lesssim |s|^{-\frac 12}.
\]
Proceeding as before, using~\eqref{cut:onR}, we obtain in this case
\[
\left| \int (\Lambda_j V_j) \Gt_k(\e) \right| \lesssim |s|^{-\frac 12} \left( \int \e^2 \ph_k \right)^{\frac 12} + \|\e\|_{L_j^2}.
\]
By~\eqref{apriori:eps} and~\eqref{est:mk}, it follows that
\begin{align*}
\left| m_{j,1} \int (\Lambda_j V_j) \Gt_k(\e) \right|
& \lesssim \left[ \|\e\|_{L^2_j} + \sum_{j'} \|\e\|_{L^2_{j'}}^2 + |s|^{-2} \right]
\left[ |s|^{-\frac 12} \left( \int \e^2 \ph_k \right)^{\frac 12} + \|\e\|_{L_j^2} \right] \\
& \leq C \|\e\|_{L^2_j}^2 + \frac {\mu_k^{-2}}{10K|s|} \int \e^2 \ph_k
+ C |s|^{-\frac 12} \sum_{j'} \|\e\|_{L^2_{j'}}^2 + C |s|^{-4}.
\end{align*}
Therefore we have obtained, for $B$ and $|S_0|$ large enough, and using~\eqref{eL2k},
\[
|f_{3,1}| \leq \frac {\kappa_1}{10} \mu_k^{-4} \int \e^2 \pa_y\ph_k
+ \frac {2\mu_k^{-2}}{10|s|} \int \e^2 \ph_k + C |s|^{-\frac 12} \sum_j \|\e\|_{L^2_j}^2
+ C \sum_{j<k} \|\e\|_{L_j^2}^2 + C |s|^{-4}.
\]

\emph{Estimate of $f_{3,2}$.}
By integration by parts, we have
\[
\int \pa_y(V_j) \Gt_k(\e) = - \int \e \left( \pa_{yyy} V_j \psi_k - \mu_k^{-2} \pa_y V_j \ph_k
+ 5\V^4 \pa_y V_j \psi_k + \pa_{yy} V_j \pa_y\psi_k \right).
\]
We use similar arguments as for $f_{3,1}$, and obtain the same estimates in the cases $j > k$ and $j < k$.
In the case $j = k$, we also obtain similarly, using~\eqref{simpleVk},
\begin{multline*}
\left| \int \pa_y(V_k) \Gt_k(\e) - \int \e \left[ \pa_{yyy} Q_k
- \mu_k^{-2} \left( 1 + \frac {y - y_k}B \right) \pa_y Q_k + 5Q_k^4 \pa_y Q_k \right] \right| \\
\lesssim |s|^{-\frac 34} \left( \int \e^2 \ph_k \right)^{\frac 12} + e^{-\frac B{4\ell_k}} \|\e\|_{L^2_k}.
\end{multline*}
Using again~\eqref{orthobis}, we find
$\langle \e,(\cdot - y_k) \pa_y Q_k \rangle = \langle \e,\Lambda_k Q_k \rangle - \frac 12 \langle \e, Q_k \rangle = 0$.
Thus, from the cancellation~\eqref{LQkprime} and the estimate~\eqref{apriori:mutilde}, we have
\[
\left| \int \pa_y(V_k) \Gt_k(\e) \right| \lesssim |s|^{-\frac 34} \left( \int \e^2 \ph_k \right)^{\frac 12}
+ \left( e^{-\frac B{4\ell_k}} + |s|^{-1} \right) \|\e\|_{L^2_k}.
\]
Therefore we also obtain, for $B$ and $|S_0|$ large enough,
\[
|f_{3,2}| \leq \frac {\kappa_1}{10} \mu_k^{-4} \int \e^2 \pa_y\ph_k
+ \frac {2\mu_k^{-2}}{10|s|} \int \e^2 \ph_k + C |s|^{-\frac 12} \sum_j \|\e\|_{L^2_j}^2
+ C \sum_{j<k} \|\e\|_{L_j^2}^2 + C |s|^{-4}.
\]

\emph{Estimate of $f_{3,3}$.}
By integration by parts, we have
\[
-\int R_j \Gt_k(\e) = \int \e \left( \pa_{yy} R_j \psi_k + \pa_y R_j \pa_y\psi_k
+ 5\V^4 R_j \psi_k - \mu_k^{-2} R_j \ph_k \right).
\]
Since $R\in\Y$, we have $\|R_j\|_{H^2} \lesssim 1$ for any $1\leq j\leq K$ and so we find,
using also $\|\V\|_{L^\infty} \lesssim 1$ and~\eqref{apriori:eps},
\[
\left| \int R_j \Gt_k(\e) \right| \lesssim \|R_j\|_{H^2} \|\e\|_{L^2} \lesssim |s|^{-\frac 12}.
\]
Thus, using~\eqref{est:dotrk}, we obtain as before
\[
|f_{3,3}| \lesssim |s|^{-\frac 12} \sum_j |\dot r_j| \lesssim |s|^{-\frac 12} \sum_j \|\e\|_{L^2_j}^2 + |s|^{-\frac 52}.
\]

\emph{Estimate of $f_{3,4}$.}
By integration by parts, we have
\[
-\int P_j \Gt_k(\e) = \int \e \left( \pa_{yy} P_j \psi_k + \pa_y P_j \pa_y\psi_k
+ 5\V^4 P_j \psi_k - \mu_k^{-2} P_j \ph_k \right).
\]
From~\eqref{eq:Pknorms}, $\| \pa_y^m P_j \|_{L^2} \lesssim 1$ for $m = 1,2$ and so,
by~\eqref{apriori:eps}, for any $1\leq j\leq K$,
\[
\left| \int \e \Bigl( \pa_{yy} P_j \psi_k + \pa_y P_j \pa_y\psi_k \Bigr) \right|
\lesssim \|\e\|_{L^2} \lesssim |s|^{-\frac 12}.
\]
From~\eqref{simpleVk}, we have $\|\V\|_{L^2} + \|\V\|_{L^\infty} \lesssim 1$.
Thus, for any $1\leq j\leq K$,
\[
\left| \int \V^4 \e P_j \psi_k \right| \lesssim \int \V^4 |\e|
\lesssim \|\V\|_{L^\infty}^3 \|\V\|_{L^2} \|\e\|_{L^2} \lesssim \|\e\|_{L^2} \lesssim |s|^{-\frac 12}.
\]
Next, for $j\geq k$, by~\eqref{b:Pky} and $\ph_k(y) \lesssim e^{\frac 1B (y - y_k)}$ for $y < y_k$ from~\eqref{cut:left},
we have $\|P_j \ph_k\|_{L^2} \lesssim 1$, and so
\[
\left| \langle \e,P_j \ph_k \rangle \right| \lesssim \|\e\|_{L^2} \|P_j \ph_k\|_{L^2}
\lesssim \|\e\|_{L^2} \lesssim |s|^{-\frac 12}.
\]
In contrast, for $j < k$, by~\eqref{b:Pky} and $\ph_k(y) = 2 - e^{-\frac 1B (y - y_k)}$ for $y > y_k + B$
from the definition of $\ph$, we have $\|P_j (\ph_k - 2)\|_{L^2} \lesssim |s|^{-10}$, and so
\[
\left| 2\langle \e,P_j \rangle - \langle \e,P_j \ph_k \rangle \right|
\lesssim \|\e\|_{L^2} \|P_j (\ph_k - 2)\|_{L^2} \lesssim |s|^{-10}.
\]
Using also~\eqref{est:dotak}, it follows from these estimates that
\[
|f_{3,4}| \lesssim |s|^{-\frac 12} \sum_j |\dot a_j|
\lesssim |s|^{-\frac 12} \sum_j \|\e\|_{L^2_j}^2 + |s|^{-\frac 52}.
\]

\emph{Estimate of $f_{3,5}$.}
By integration by parts, we have
\[
- \int \P \Gt_k(\e) = \int \e \left( \P_{yy} \psi_k + \P_y \pa_y\psi_k
+ 5\V^4 \P \psi_k - \mu_k^{-2} \P \ph_k \right).
\]
From~\eqref{cut:onR}, we have $\psi_k + \pa_y\psi_k \lesssim \ph_k$ on $\R$ and
$\ph_k(y) \lesssim e^{\frac 1B (y - y_k)}$ for $y < y_k$.
Thus, using also $\|\V\|_{L^\infty} \lesssim 1$ from~\eqref{simpleVk}, we find
\[
|f_{3,5}| \lesssim \|\P\|_{H^2(y > y_k - |s|^{\frac 14})} \left( \int \e^2 \ph_k \right)^{\frac 12}
+ e^{-\frac 1B |s|^{\frac 14}} \|\P\|_{H^2} \|\e\|_{L^2}.
\]
From~\eqref{est:Psi}, \eqref{est:Psik} and~\eqref{eq:BS}, we obtain the estimate
\[
|f_{3,5}| \lesssim \left( |s|^{-\frac {13}8} + |s|^{-\frac 32 - \d_{k-1}^-} \right)
\left( \int \e^2 \ph_k \right)^{\frac 12} + |s|^{-10}
\leq \frac {\mu_k^{-2}}{10|s|} \int \e^2 \ph_k + C |s|^{-2 - 2\d_{k-1}^-}.
\]

\emph{Estimate of $f_{3,6}$.}
We first decompose $f_{3,6}$ as
\[
- f_{3,6} = \left[ \left( \frac {\dot \mu_k}{\mu_k} + \frac 1{2\mu_k^3\tau_k} - \frac 1{2s} + \frac {a_k}{\mu_k^3} \right)
- \left( \frac 1{2\mu_k^3\tau_k} - \frac 1{2s} \right) - \frac {a_k}{\mu_k^3} \right] \mu_k^{-2} \int \e^2 \ph_k.
\]
Thus, using on the one hand the definition~\eqref{def:mod} and the estimate~\eqref{est:mk} of $m_{k,1}$,
and on the other hand the estimates $\left| \frac 1{2\mu_k^3\tau_k} - \frac 1{2s} \right| \lesssim |s|^{-1 - \d_k}$
and $|a_k| \leq |s|^{-1 - \d_k^-}$ from~\eqref{eq:BS}--\eqref{eq:BS2}, we find
\begin{align*}
|f_{3,6}| &\lesssim \left( \|\e\|_{L_k^2} + \sum_j \|\e\|_{L_j^2}^2 + |s|^{-1 - \d_k^-} \right) \mu_k^{-2} \int \e^2 \ph_k \\
&\lesssim \|\e\|_{L^2}^2 \|\e\|_{L_k^2} + \|\e\|_{L^2}^2 \sum_j \|\e\|_{L_j^2}^2 + |s|^{-1 - \d_k^-} \mu_k^{-2} \int \e^2 \ph_k.
\end{align*}
Using the estimate on $\|\e\|_{L^2}$ from~\eqref{eq:BS}, we obtain
\[
\|\e\|_{L^2}^2 \|\e\|_{L_k^2} \lesssim |s|^{-1 - 2\d_{K+1}^+} \|\e\|_{L_k^2}
\lesssim |s|^{-\d_{K+1}^+} \|\e\|_{L_k^2}^2 + |s|^{-2 - 3\d_{K+1}^+},
\]
and so, since $3\d_{K+1}^+ > 3\d_{K+1} > 2\d_0 > 2\d_{k-1}^-$ from~\eqref{prop:d},
\[
|f_{3,6}| \lesssim |s|^{-\d_{K+1}^+} \mu_k^{-4} \int \e^2 \pa_y\ph_k + |s|^{-2 - 2\d_{k-1}^-}
+ |s|^{-1} \sum_j \|\e\|_{L_j^2}^2 + |s|^{-1 - \d_k^-} \mu_k^{-2} \int \e^2 \ph_k.
\]
Thus, taking $|S_0|$ large enough, we obtain
\[
|f_{3,6}| \leq \frac {\kappa_1}{10} \mu_k^{-4} \int \e^2 \pa_y\ph_k + \frac {\mu_k^{-2}}{10|s|} \int \e^2 \ph_k
+ C |s|^{-1} \sum_j \|\e\|_{L_j^2}^2 + C |s|^{-2 - 2\d_{k-1}^-}.
\]

\emph{Estimate of $f_3$.}
Gathering the above estimates of $f_{3,1},\ldots,f_{3,6}$, we obtain
\[
|f_3| \leq \frac {3\kappa_1}{10} \mu_k^{-4} \int \e^2 \pa_y\ph_k + \frac {6\mu_k^{-2}}{10|s|} \int \e^2 \ph_k
+ C |s|^{-\frac 12} \sum_j \|\e\|_{L^2_j}^2 + C \sum_{j<k} \|\e\|_{L^2_j}^2 + C |s|^{-2 - 2\d_{k-1}^-}.
\]

\medskip

\textbf{Control of $f_4$.}
By integration by parts, we have
\begin{align*}
\int \e_y G_k(\e) &= -\frac 12 \int \e_y^2 \pa_y\psi_k - \frac 12 \mu_k^{-2} \int \e^2 \pa_y\ph_k
+ \frac 16 \int \left[ (\V + \e)^6 - \V^6 - 6\V^5\e \right] \pa_y\psi_k \\
&\quad + \int \V_y \left[ (\V + \e)^5 - \V^5 - 5\V^4\e \right] \psi_k.
\end{align*}
Thus, we may rewrite $f_4$ as
\begin{align*}
f_4 &= - \left( \dot y_k - \frac {y_k}{2s} - \frac 1{\mu_k^2} \right)
\int \left\{ \e_y^2 \pa_y\psi_k + \mu_k^{-2} \e^2 \pa_y\ph_k
- \frac 13 \left[ (\V + \e)^6 - \V^6 - 6\V^5\e \right] \pa_y\psi_k \right\} \\
&\quad + 2 \int \left[ \V_{yyy} + 5\V_y\V^4 - \mu_k^{-2} \V_y \right] \left[ (\V + \e)^5 - \V^5 - 5\V^4\e \right] \psi_k
= f_{4,1} + f_{4,2}.
\end{align*}

\emph{Estimate of $f_{4,1}$.}
We first note that, from the definition~\eqref{def:mod} and the estimate~\eqref{est:mk} of~$m_{k,2}$,
and from~\eqref{apriori:eps},
\[
\left| \dot y_k - \frac {y_k}{2s} - \frac 1{\mu_k^2} \right| \lesssim |\vec m_k| + |s|^{-1}
\lesssim \|\e\|_{L_k^2} + |s|^{-1} \lesssim |s|^{-\frac 12}.
\]
Thus, using~\eqref{cut:onR} and taking $|S_0|$ large enough, we obtain
\[
|f_{4,1}| \lesssim |s|^{-\frac 12} \int \left( \e_y^2 +\mu_k^{-2} \e^2 \right) \pa_y\ph_k
\leq \frac {\kappa_1}{10} \mu_k^{-2} \int \left( \e_y^2 + \mu_k^{-2} \e^2 \right) \pa_y\ph_k.
\]

\emph{Estimate of $f_{4,2}$.}
We decompose $f_{4,2}$ as $f_{4,2} = f_{4,2}^< + f_{4,2}^\sim + f_{4,2}^>$,
where $f_{4,2}^<$, $f_{4,2}^\sim$ and $f_{4,2}^>$ respectively correspond to integration
on $y - y_k < -\frac B2$, $|y - y_k| \leq \frac B2$ and $y - y_k > \frac B2$,
and we follow the above calculation done for the estimate of $f_1$.

First, to estimate $f_{4,2}^<$, we recall that $\|\V\|_{L^\infty} \lesssim 1$
from~\eqref{simpleVk}, and so, using also~\eqref{cut:onR} and~\eqref{cut:left},
\[
|f_{4,2}^<| \lesssim B \int_{y - y_k < -\frac B2} \left( |\V_{yyy}| + |\V_y| \right)
\left( |\e|^5 + |\V|^3\e^2 \right) \pa_y\ph_k.
\]
Thus, we may estimate this term as the one similar in $f_1^<$ and find,
taking $B$ large enough and then $|S_0|$ large enough,
\[
|f_{4,2}^<| \leq \frac {\kappa_1}{10} \mu_k^{-4} \int \e^2 \pa_y\ph_k + C |s|^{-10}.
\]

Second, to estimate $f_{4,2}^>$, we recall that $\|\pa_y^m\V\|_{L^\infty} \lesssim 1$
for $m=1,3$ from~\eqref{simpleVk} and so, proceeding as in the estimate of $f_1^>$, we find
\[
|f_{4,2}^>| \lesssim \int_{y - y_k > \frac B2} \left( |\e|^5 + |\V|^3\e^2 \right)
\lesssim \sum_{j < k} \|\e\|_{L^2_j}^2 + B^{-2} \|\e\|_{L^2_k}^2 + |s|^{-\frac 52}.
\]
Thus, using~\eqref{eL2k} and taking $B$ large enough, we find
\[
|f_{4,2}^>| \leq \frac {\kappa_1}{10} \mu_k^{-4} \int \e^2 \pa_y\ph_k
+ C \sum_{j < k} \|\e\|_{L^2_j}^2 + C |s|^{-\frac 52}.
\]

Third, to estimate $f_{4,2}^\sim$, we find as in the estimate of $f_1^\sim$ that, for $m=1,3$,
\[
\|\V - Q_k\|_{L^\infty(|y - y_k| \leq \frac B2)} \lesssim |s|^{-\frac 34} \quad\m{and}\quad
\|\pa_y^m(\V - Q_k)\|_{L^\infty(|y - y_k| \leq \frac B2)} \lesssim |s|^{-1}.
\]
Thus, using also~\eqref{apriori:mutilde}, the cancellation~\eqref{LQkprime}
and the identity $\pa_y\ph_k(y) = \frac 1B$ for $|y - y_k| \leq \frac B2$, we find
\[
|f_{4,2}^\sim| \lesssim \int_{|y - y_k| \leq \frac B2} \left| \V_{yyy} + 5\V_y\V^4 - \mu_k^{-2} \V_y \right| \e^2
\lesssim |s|^{-\frac 34} \int_{|y - y_k| \leq \frac B2} \e^2 \lesssim B |s|^{-\frac 34} \int \e^2 \pa_y\ph_k.
\]

Finally, gathering the estimates of $f_{4,2}^<$, $f_{4,2}^\sim$ and $f_{4,2}^>$,
and taking $|S_0|$ large enough, we have
\[
|f_{4,2}| \leq \frac {3\kappa_1}{10} \mu_k^{-4} \int \e^2 \pa_y\ph_k
+ C \sum_{j < k} \|\e\|_{L^2_j}^2 + C |s|^{-\frac 52}.
\]

\emph{Estimate of $f_4$.}
Gathering the above estimates of $f_{4,1}$ and $f_{4,2}$, we obtain
\[
|f_4| \leq \frac {4\kappa_1}{10} \mu_k^{-2} \int \left( \e_y^2 + \mu_k^{-2} \e^2 \right) \pa_y\ph_k
+ C \sum_{j < k} \|\e\|_{L^2_j}^2 + C |s|^{-\frac 52}.
\]

\medskip

\textbf{Conclusion.}
Gathering the above estimates of $f_1$, $f_2$, $f_3$ and $f_4$, we have obtained
\begin{align*}
f_1 + f_2 + f_3 + f_4 &\leq -\frac {\kappa_1}5 \mu_k^{-2} \int \left( \e_y^2 + \mu_k^{-2} \e^2 \right) \pa_y\ph_k
- \frac {2\mu_k^{-2}}{5|s|} \int \e^2 \ph_k \\
&\quad + C \sum_{j < k} \|\e\|_{L_j^2}^2 + C |s|^{-\frac 12} \sum_j \|\e\|_{L^2_j}^2 + C |s|^{-2 - 2\d_{k-1}^-}.
\end{align*}
Thus, taking $|S_0|$ large enough, we have
\begin{multline*}
\frac 1{|s|^{1 + \frac 1{43}}} \frac d{d s} \left\{ |s|^{1 + \frac 1{43}}
\left[ \F_k + 4\mu_k^{-2} \sum_{j < k} a_j \langle \e,P_j \rangle + 2c \mu_k^{-2} |s| \sum_{j < k} a_j^2 \right] \right\} \\
\leq -\frac {\kappa_1}{10} \mu_k^{-2} \int \left( \e_y^2 + \mu_k^{-2} \e^2 \right) \pa_y\ph_k + C \sum_{j < k} \|\e\|_{L_j^2}^2
+ C |s|^{-\frac 12} \sum_j \|\e\|_{L^2_j}^2 + C |s|^{-2 - \d_{k-1}^- - \d_k^+},
\end{multline*}
which proves~\eqref{eq:Fvariation} and concludes the proof of Lemma~\ref{lem:enerloc}.
\end{proof}

\subsection{Local and global estimates on $\e_n$}

We claim the following estimates on $\e_n$,
strictly improving the first two lines of~\eqref{eq:BS} for $|S_0|$ large enough:
\begin{gather}
\|\e_n(s)\|_{H^1}^2 \lesssim |s|^{-1 - \d_K^- - \d_{K+1}^+}
\leq \frac 14 |s|^{-1 - 2\d_{K+1}^+}, \label{close:1} \\
\|\e_n(s)\|_{H^1(y > y_k)}^2 \lesssim |s|^{-1 - \d_{k-1}^- - \d_k^+}
\leq \frac 14 |s|^{-1 - 2\d_k^+}, \label{close:1bis} \\
\int_{S_n}^s |\tau|^{1 + \frac 1{43}} \left( \|\pa_y\e_n(\tau)\|_{L^2_k}^2 + \|\e_n(\tau)\|_{L^2_k}^2 \right) d\tau
\lesssim |s|^{-\d_{k-1}^- - \d_k^+ + \frac 1{43}} \leq \frac 12 |s|^{-2\d_k^+ + \frac 1{43}}. \label{close:2}
\end{gather}
Note that, as a consequence of~\eqref{close:2}, by the Cauchy--Schwarz inequality,
\be \label{loc}
\int_{S_n}^s \|\e_n(\tau)\|_{L^2_k}\, d\tau
\lesssim |s|^{-\frac 1{86}} \left(\int_{S_n}^s |\tau|^{1+\frac 1{43}}
\|\e_n(\tau)\|_{L^2_k}^2\, d\tau \right)^{\frac 12} \lesssim |s|^{-\frac12(\d_{k-1}^-+\d_k^+)}.
\ee

\emph{Proof of~\eqref{close:1}.}
As before, we denote $\e_n$ and $\V_n$ simply by $\e$ and $\V$.
First, we prove the bound on $\|\e(s)\|_{L^2}$.
Integrating~\eqref{g.mass} on $[S_n,s]$, using
$\int_{S_n}^s |\tau| \|\e(\tau)\|_{L^2_j}^2 \,d\tau \leq |s|^{-2\d_j^+}$
from~\eqref{eq:BS} and finally~\eqref{initial}, we get
\begin{align*}
&\|\e\|_{L^2}^2 + 2 \sum_k a_k \langle \e,P_k \rangle + c |s| \sum_k a_k^2 \\
&\leq C |s|^{-1} \sum_j \int_{S_n}^s |\tau| \|\e(\tau)\|_{L^2_j}^2 \,d\tau
+ C |s|^{-1 - \d_K^- - \d_{K+1}^+} + c |s|^{-1} |S_n|^2 \sum_k a_k^2(S_n) \\
&\leq C |s|^{-1 - 2\d_K^+} + C |s|^{-1 - \d_K^- - \d_{K+1}^+} + C |s|^{-1 - 2\d_K^-}
\leq C |s|^{-1 - \d_K^- - \d_{K+1}^+}.
\end{align*}
Using~\eqref{b:Jk} and the bound on $|a_k|$ from~\eqref{eq:BS},
we obtain $\|\e\|_{L^2}^2 \lesssim |s|^{-1 - \d_K^- - \d_{K+1}^+}$.
Now, we may prove the bound on $\|\e(s)\|_{\dot H^1}$.
Integrating~\eqref{g.energy} on $[S_n,s]$, using the bound on $\|\e\|_{L^2}$ found above, we also get
\begin{align*}
\|\e_y\|_{L^2}^2 &\lesssim |s|^{-1} \sum_j \int_{S_n}^s |\tau| \|\e(\tau)\|_{L^2_j}^2 \,d\tau
+ |s|^{-1 - \d_K^- - \d_{K+1}^+} +  \int \left| (\V + \e)^6 - \V^6 - 6\V^5\e \right| \\
&\lesssim |s|^{-1 - 2\d_K^+} + |s|^{-1 - \d_K^- - \d_{K+1}^+} + \|\e\|_{L^2}^2 \lesssim |s|^{-1 - \d_K^- - \d_{K+1}^+}.
\end{align*}

\emph{Proof of~\eqref{close:1bis} and~\eqref{close:2}.}
We denote $\e_n$, $\V_n$ and $\F_{k,n}$ by $\e$, $\V$ and $\F_k$.
Integrating~\eqref{eq:Fvariation} on $[S_n,s]$, using from~\eqref{eq:BS} the estimates
\[
\begin{alignedat}{2}
\int_{S_n}^s |\tau|^{1 + \frac 1{43}} \|\e(\tau)\|_{L^2_j}^2 \,d\tau
&\leq |s|^{-2\d_j^+ + \frac 1{43}}, &&\m{for } j < k, \\
\int_{S_n}^s |\tau|^{\frac 12 + \frac 1{43}} \|\e(\tau)\|_{L^2_j}^2 \,d\tau
&\leq |s|^{-\frac 12 - 2\d_j^+ + \frac 1{43}},\quad &&\m{for all } 1\leq j\leq K,
\end{alignedat}
\]
we find, using also the estimate $a_j^2(S_n) \lesssim |S_n|^{-2 - 2\d_{k-1}^-}$ for $j < k$ from~\eqref{initial},
\begin{multline*}
|s|^{1 + \frac 1{43}} \left( \F_k + 4 \mu_k^{-2} \sum_{j < k} a_j \langle \e,P_j \rangle
+ 2c \mu_k^{-2} |s| \sum_{j < k} a_j^2 \right) + \kappa \int_{S_n}^s |\tau|^{1 + \frac 1{43}}
\left( \|\pa_y\e\|_{L^2_k}^2 + \|\e\|_{L^2_k}^2 \right) d\tau \\
\lesssim |s|^{-2\d_{k-1}^+ + \frac 1{43}} + |s|^{-\frac 12 - 2\d_K^+ + \frac 1{43}}
+ |s|^{-\d_k^+ - \d_{k-1}^- + \frac 1{43}} + |s|^{-2\d_{k-1}^- + \frac 1{43}}
\lesssim |s|^{-\d_k^+ - \d_{k-1}^- + \frac 1{43}}.
\end{multline*}
Moreover, using~\eqref{b:Jk} and the bound on $|a_j|$ from~\eqref{eq:BS}, we have
\[
\sum_{j < k} \left| a_j \langle \e,P_j \rangle \right| \lesssim |s|^{-1 - \d_k^+ - \d_{k-1}^-}.
\]
Thus, using also~\eqref{eq:Fcoercivity} and the definitions of $\psi_k$ and $\ph_k$, we obtain
\[
\|\e\|_{H^1(y > y_k)}^2 \leq \int (\pa_y \e)^2 \psi_k + \int \e^2 \ph_k \lesssim |s|^{-1 - \d_k^+ - \d_{k-1}^-}
\]
and
\[
\int_{S_n}^s |\tau|^{1 + \frac 1{43}} \left( \|\pa_y\e(\tau)\|_{L^2_k}^2 + \|\e(\tau)\|_{L^2_k}^2 \right) d\tau
\lesssim |s|^{-\d_k^+ - \d_{k-1}^- + \frac 1{43}}.
\]

\subsection{Parameters estimates} \label{sec:parameters}

Now, we close the parameters estimates in~\eqref{eq:BS} on the time interval $[S_n,S_n^*]$.
Technically, even if the parameters $e_k$ and $f_k$ do not appear in the bootstrap~\eqref{eq:BS},
their estimates are necessary to handle the other parameters.

\medskip

\emph{Estimate of $e_k$.}
Integrating~\eqref{est:dotek} on $[S_n,s]$, using~\eqref{loc} for the terms in $\|\e_n\|_{L_j^2}$
for $j\leq k$ and~\eqref{eq:BS} for the terms in $\|\e_n\|_{L_j^2}^2$ and $a_j$,
we obtain, using also~\eqref{prop:d},
\begin{align*}
|e_k(s) - e_k(S_n)| &\lesssim \sum_{j\leq k} \int_{S_n}^s \|\e_n(\tau)\|_{L_j^2} \,d\tau
+ \sum_j \int_{S_n}^s |\tau| \|\e_n(\tau)\|_{L_j^2}^2 \,d\tau
+ \sum_{j < k} \int_{S_n}^s |a_j(\tau)| \,d\tau + |s|^{-\frac 12} \\
&\lesssim |s|^{-\frac 12 (\d_{k-1}^- + \d_k^+)} + |s|^{-2\d_K^+} + |s|^{-\d_{k-1}^-} + |s|^{-\frac 12}
\lesssim |s|^{-\frac 12 (\d_{k-1}^- + \d_k^+)}.
\end{align*}
Thus, the value of $e_k(S_n)$ being given in~\eqref{initial}, we find
\be \label{close:3}
\left |e_k(s) - \ell_k \left( \frac 12 + \theta_k \right) \right| \lesssim |s|^{-\frac 12 (\d_{k-1}^- + \d_k^+)}.
\ee

\emph{Estimate of $f_k$.}
We first rewrite~\eqref{est:fk} in terms of
\[
g_k(s) = f_k(s) + \int_{S_n}^s (1 + \bar \mu_k(\tau)) \langle \e_n(\tau),A_k(\tau) \rangle \,d\tau,
\]
as
\begin{multline*}
\left| \dot g_k + \frac 12 \left( 1 + 3\theta_k \right) \frac {g_k}s \right|
\lesssim |s|^{-1} \left| \int_{S_n}^s (1 + \bar \mu_k(\tau))
\langle \e_n(\tau),A_k(\tau) \rangle \,d\tau \right| + \sum_j \|\e\|_{L_j^2}^2 \\
+ |s|^{-1} \left( |s|^{-\frac 1{42}} + |\bar\mu_k|^2 + |\bar y_k|^2
+ \left| e_k - \ell_k \left( \frac 12 + \theta_k \right) \right|
+ \sum_{j < k} (|\bar\mu_j| + |\bar\tau_j| + |\bar y_j|) \right).
\end{multline*}
Using the decay properties of $A_k$, \eqref{eq:BS}--\eqref{eq:BS2} and~\eqref{close:3}, we obtain
\begin{align*}
\left| \dot g_k + \frac 12 \left( 1 + 3\theta_k \right) \frac {g_k}s \right|
&\lesssim |s|^{-1} \int_{S_n}^s \|\e_n(\tau)\|_{L_k^2} \,d\tau + |s|^{-1 - 2\d_{K+1}^+} \\
&\quad + |s|^{-1} \left( |s|^{-\frac 1{42}} + |s|^{-2\d_k} + |s|^{-\frac 12 (\d_{k-1}^- + \d_k^+)}
+ |s|^{-\d_{k-1}} \right)
\end{align*}
and so, using also~\eqref{loc} and $2\d_k > 2\d_{K+1}^+ > \frac 1{42} > \d_{k-1} > \d_{k-1}^- > \d_k^+$ from~\eqref{prop:d},
\be \label{surg}
\left| \dot g_k + \frac 12 \left( 1 + 3\theta_k \right) \frac {g_k}s \right|
\lesssim |s|^{-1 - \frac 12 (\d_{k-1}^- + \d_k^+)}.
\ee

Following the discussion in Section~\ref{sec:discussion},
we consider now separately the cases $k\in \K^-$ and $k\in \K^+$.
First, let $k\in \K^-$ (if $\K^-$ is empty, we just skip this case).
Then~\eqref{surg} rewrites as
\[
\left| \frac d{d s} \left[ (-s)^{\frac 12 (1 + 3\theta_k)} g_k \right] \right|
\lesssim |s|^{-1 - \frac 12(\d_{k-1}^- + \d_k^+) + \frac 12 (1 + 3\theta_k)}.
\]
Since $-\frac 12(\d_{k-1}^- + \d_k^+) + \frac 12 (1 + 3\theta_k) < -\d_K^+ + \d_{K+1} < 0$ for $k\in \K^-$,
integrating on $[S_n,s]$, using $g_k(S_n) = 0$ from~\eqref{initial},
we get $|g_k(s)| \lesssim |s|^{-\frac 12(\d_{k-1}^- + \d_k^+)}$.
Thus, by~\eqref{loc}, for $|S_0|$ large enough,
\[
|f_k(s)| \lesssim |g_k(s)| + \int_{S_n}^s \|\e_n(\tau)\|_{L^2_k} \,d\tau
\lesssim |s|^{-\frac 12(\d_{k-1}^- + \d_k^+)} \leq |s|^{-\d_k^+}.
\]

Now, let $k\in \K^+$. We use~\eqref{eq:BS2} to obtain directly $|g_k(s)| \leq |s|^{-\d_k^+}$.
As above we obtain, for $|S_0|$ large enough again,
$|f_k(s)| \leq |s|^{-\d_k^+} + C |s|^{-\frac 12(\d_{k-1}^- + \d_k^+)} \leq 2|s|^{-\d_k^+}$.

In conclusion, for all $1\leq k\leq K$, we have
\be \label{close:4}
|f_k(s)| \leq 2 |s|^{-\d_k^+}.
\ee

\emph{Estimate of $\bar y_k$.}
Note that~\eqref{est:barykfk} rewrites as
\[
\left| \frac d{d s} \left[ (-s)^{-\frac 12} \bar y_k \right] \right|
\lesssim |s|^{-\frac 32} |f_k| + |s|^{-\frac 12} \sum_j \|\e\|_{L_j^2}^2 + |s|^{-\frac 32} |\bar\mu_k|^2 + |s|^{-\frac 52}.
\]
Thus, using~\eqref{eq:BS} and~\eqref{close:4}, we obtain,
since $2\d_k > 2\d_{K+1}^+ > \d_0 > \d_k^+$ from~\eqref{prop:d},
\[
\left| \frac d{d s} \left[ (-s)^{-\frac 12} \bar y_k \right] \right|
\lesssim |s|^{-\frac 32 - \d_k^+} + |s|^{-\frac 32 - 2\d_{K+1}^+}
+ |s|^{-\frac 32 - 2\d_k} + |s|^{-\frac 52} \lesssim |s|^{-\frac 32 - \d_k^+}.
\]
Integrating on $[S_n,s]$, using $\bar y_k(S_n) = 0$ from~\eqref{initial}, we obtain, for $|S_0|$ large enough,
\be \label{close:5}
|\bar y_k(s)| \lesssim |s|^{-\d_k^+} \leq \frac 12 |s|^{-\d_k}.
\ee

\emph{Estimate of $\bar \mu_k$.}
From~\eqref{close:4} and~\eqref{close:5}, since $f_k = \bar \mu_k + \bar y_k$, we obtain also, for $|S_0|$ large enough,
\be \label{close:6}
|\bar \mu_k(s)| \lesssim |s|^{-\d_k^+} \leq \frac 12 |s|^{-\d_k}.
\ee

\emph{Estimate of $a_k$.}
We extract $a_k$ from the definition~\eqref{def:ek} of the local energy $e_k$ as
\[
a_k = \frac {\mu_k^2}s e_k - \frac 1{2\tau_k} - \frac {4r_k}{\|Q\|_{L^1}}.
\]
But, by~\eqref{eq:rkasymp} and~\eqref{eq:BS}--\eqref{eq:BS2}, we find
\[
\left| \frac {4r_k}{\|Q\|_{L^1}} - \frac {\ell_k^3\theta_k}s \right| \lesssim |s|^{-1 - \d_k}.
\]
Thus, from~\eqref{close:3} and again~\eqref{eq:BS}--\eqref{eq:BS2}, we have
\[
a_k = \frac {\ell_k^3}s (1 + \bar \mu_k)^2 \left( \frac 12 + \theta_k \right)
- \frac {\ell_k^3}{2s} (1 + \bar \tau_k)^{-1} - \frac {\ell_k^3\theta_k}s
+ O\left( |s|^{-1 - \d_k} \right) = O\left( |s|^{-1 - \d_k} \right)
\]
and so, for $|S_0|$ large enough,
\[
|a_k| \lesssim |s|^{-1 - \d_k} \leq \frac 12 |s|^{-1 - \d_k^-}.
\]

In conclusion, we have strictly improved all the estimates in~\eqref{eq:BS}.

\subsection{Topological obstruction} \label{sec:Brouwer}

Since the estimates in~\eqref{eq:BS} have been strictly improved on $[S_n,S_n^*]$
and since we have assumed $S^*_n < S_0$, by continuity,
we know that the bootstrap estimate~\eqref{eq:BS2} has to be saturated at $s = S^*_n$ and,
for some $S_n^{**} \in (S_n^*,S_0)$ close enough to $S_n^*$,
the bootstrap estimates~\eqref{eq:BS} hold on $[S_n,S_n^{**}]$.

Thus, for $s\in [S_n,S_n^{**}]$, we may consider
\[
\NN(s) = \sum_{k\in \K^+} \left[ |s|^{\d_k^+} g_k(s) \right]^2
+ \sum_{k=1}^K \left[ |s|^{\d_k} \bar \tau_k(s) \right]^2
\]
and note that $\NN(S_n^*) = 1$.
By continuity again, we choose $S_n^{**}$ possibly closer to $S_n^*$ so that $\NN(s) \leq 2$
for all $s\in [S_n,S_n^{**}]$.
Now we claim that, for all $s^*\in [S_n,S_n^{**}]$ such that $\NN(s^*) = 1$,
\be \label{eq:BS4}
\frac {d \NN}{d s}(s^*) > 0.
\ee

Indeed, we first compute
\[
\frac 12 \frac {d \NN}{d s}
= \sum_{k\in \K^+} \left( -\d_k^+ |s|^{2\d_k^+ - 1} g_k^2 + |s|^{2\d_k^+} \dot g_k g_k \right)
+ \sum_k \left( -\d_k |s|^{2\d_k - 1} \bar\tau_k^2 + |s|^{2\d_k} \dot{\bar\tau}_k \bar\tau_k \right).
\]
But from~\eqref{surg} we have, for all $1\leq k\leq K$,
\[
\dot g_k = \frac 12 \left( 1 + 3\theta_k \right) \frac {g_k}{|s|}
+ O\left( |s|^{-1 - \frac 12(\d_{k-1}^- + \d_k^+)} \right).
\]
Moreover, from~\eqref{def:tauk} and~\eqref{close:6},
\begin{align*}
\dot{\bar\tau}_k
&= \ell_k^3 \left( \frac {\dot \tau_k}s - \frac {\tau_k}{s^2} \right)
= \frac {\ell_k^3}s \left[ \mu_k^{-3} - \ell_k^{-3} (1 + \bar \tau_k) \right] \\
&= -|s|^{-1} \left[ (1 + \bar\mu_k)^{-3} - (1 + \bar \tau_k) \right]
= |s|^{-1} \bar \tau_k + O\left( |s|^{-1 - \d_k^+} \right).
\end{align*}
Inserting these estimates in the above formula of $\frac {d \NN}{d s}$, we obtain
\begin{align*}
\frac 12 \frac {d \NN}{d s}
&= \sum_{k\in \K^+} \left[ -\d_k^+ + \frac 12 \left( 1 + 3\theta_k \right) \right] |s|^{2\d_k^+ - 1} g_k^2
+ O\left( |s|^{-1 - \frac 12 \d_{k-1}^- + \frac 32 \d_k^+} |g_k| \right) \\
&\quad + \sum_k \left( -\d_k + 1 \right) |s|^{2\d_k - 1} \bar\tau_k^2
+ O\left( |s|^{-1 - \d_k^+ + 2\d_k} |\bar\tau_k| \right)
\end{align*}
and so, since $\NN(s) \leq 2$ for $s\in [S_n,S_n^{**}]$,
\begin{align*}
\frac 12 \frac {d \NN}{d s}
&= \sum_{k\in \K^+} \left[ -\d_k^+ + \frac 12 \left( 1 + 3\theta_k \right) \right] |s|^{2\d_k^+ - 1} g_k^2
+ \sum_k \left( -\d_k + 1 \right) |s|^{2\d_k - 1} \bar\tau_k^2 \\
&\quad + O\left( |s|^{-1 - \frac12 (\d_{k-1}^- - \d_k^+)} \right) + O\left( |s|^{-1 - (\d_k^+ - \d_k)} \right).
\end{align*}
Now we recall that $\frac 12 \left( 1 + 3\theta_k \right) \geq \d_0 > \d_1^+ \geq \d_k^+$ for $k\in \K^+$,
from~\eqref{prop:d} and the definition of~$\K^+$, and similarly
$1 > \frac 1{42} \geq \d_0 > \d_1^+ > \d_k$ for all $1\leq k\leq K$.
Thus,
\[
\frac 12 \frac {d \NN}{d s} \geq (\d_0 - \d_1^+) |s|^{-1} \NN
- C |s|^{-1 - \frac12 (\d_{k-1}^- - \d_k^+)} - C |s|^{-1 - (\d_k^+ - \d_k)}.
\]
Since $\NN(s^*) = 1$, taking $|S_0|$ large enough, we obtain
\[
\frac 12 \frac {d \NN}{d s}(s^*)
\geq (\d_0 - \d_1^+) |s^*|^{-1} - C |s^*|^{-1 - \frac 12 (\d_{k-1}^- - \d_k^+)} - C |s^*|^{-1 - (\d_k^+ - \d_k)}
\geq \frac 12 (\d_0 - \d_1^+) |S_n|^{-1} > 0,
\]
and~\eqref{eq:BS4} is proved.

Finally, we follow classical arguments as in~\cite{C,CMM} to conclude.
We first note that the map $(\XI,\ZETA) \mapsto S_n^*(\XI,\ZETA)$ is continuous.
Indeed, if $\eps > 0$ is given, since $\frac {d \NN}{d s}(S_n^*) > 0$ from~\eqref{eq:BS4},
$\NN$ is strictly increasing on $[S_n^* - \eps,S_n^* + \eps]$ and there exists $\d > 0$
such that $\NN(S_n^* + \eps) > 1 + \d$ and $\NN(S_n^* - \eps) < 1 - \d$.
But, from the transversality property~\eqref{eq:BS4}, we may choose $\d > 0$ possibly smaller
so that $\NN(s) > 1 + \d$ for $s\in [S_n^* + \eps,S_n^{**}]$ and $\NN(s) < 1 - \d$ for $s\in [S_n,S_n^* - \eps]$.
Now note that, by continuity of the flow, there exists $\eta > 0$ such that
if $\| (\tilde\XI,\tilde\ZETA) - (\XI,\ZETA) \| \leq \eta$
then $|\tilde\NN(s) - \NN(s)| \leq \d/2$ for all $s\in [S_n,S_n^{**}]$.
Thus, from the definition of $S_n^*$, we may conclude that
$|S_n^*(\tilde\XI,\tilde\ZETA) - S_n^*(\XI,\ZETA)| \leq \eps$
whenever $\| (\tilde\XI,\tilde\ZETA) - (\XI,\ZETA) \| \leq \eta$,
which proves that $(\XI,\ZETA) \mapsto S_n^*(\XI,\ZETA)$ is indeed continuous.

In particular, the map
\begin{align*}
\mathcal{M} : \B_{d+K} & \to \S_{d+K} \\
(\XI,\ZETA) & \mapsto \left( \{ |S_n^*|^{\d_k^+} g_k(S_n^*) \}_{k\in \K^+},
\{ |S_n^*|^{\d_k} \bar\tau_k(S_n^*) \}_{1\leq k\leq K} \right)
\end{align*}
is also continuous.
Moreover, if $(\XI,\ZETA)\in \S_{d+K}$ then $\NN(S_n) = 1$ from~\eqref{initial},
thus $S_n^* = S_n$ from~\eqref{eq:BS4} and the definition of $S_n^*$,
and so $\mathcal{M}(\XI,\ZETA) = (\XI,\ZETA)$ again from~\eqref{initial}.
In other words, $\mathcal{M}$ restricted to $\S_{d+K}$ is the identity,
and the existence of such a map is contradictory with Brouwer's fixed-point theorem.
Therefore, Proposition~\ref{pr:boot} is proved.

\subsection{Conclusion} \label{sec:conclusion}

\begin{proof}[End of the proof of Proposition~\ref{prop:v}]
By Proposition~\ref{pr:boot}, we may consider a sequence $(v_n)$ of solutions of~\eqref{rescaledkdv}
as in Section~\ref{sec:bootstrap}, such that their decomposition $v_n = \V[\Gab_n] + \e_n$
satisfies the uniform estimates~\eqref{eq:BS}--\eqref{eq:BS2}
and the orthogonality conditions~\eqref{orthobis} on some time interval $[S_n,S_0]$.
In particular, it follows from these estimates that $\|v_n(S_0)\|_{H^1} \leq C$ for some $C > 0$,
so there exists $v_0\in H^1$ such that, up to a subsequence, $v_n(S_0) \rightharpoonup v_0$ in $H^1$ weak,
and we consider the solution $v$ of~\eqref{rescaledkdv} such that $v(S_0) = v_0$.

Let any $S < S_0$, and consider $n$ large enough so that $S_n < S$.
Then, applying Lemma~\ref{lem:weak} on $[S,S_0]$,
we obtain the existence of a $\C^1$ function $\Gab$ such that $\Gab$ and $\e$ defined by $\e = v - \V[\Gab]$
satisfy the estimates~\eqref{eq:BS}--\eqref{eq:BS2} on $[S,S_0]$,
and also~\eqref{orthobis} since $\e_n(s) \rightharpoonup \e(s)$ for all $s\in [S,S_0]$.
Since $S < S_0$ is arbitrary, the function $\Gab$ and $\e$ satisfy the estimates~\eqref{est:final}
and the orthogonality conditions~\eqref{ortho:final} on $(-\infty,S_0]$,
which concludes the proof of Proposition~\ref{prop:v}.
\end{proof}

\begin{proof}[Proof of Theorem~\ref{thm:main}]
We consider the solution $v$ of~\eqref{rescaledkdv} obtained in Proposition~\ref{prop:v}.
First note that, from~\eqref{est:final}, we have
$\|v - \V\|_{H^1} = \|\e\|_{H^1} \leq |s|^{-\frac 12 - \frac 1{43}}$ for all $s\leq S_0$.
Moreover, by the definition of $\V$ and~\eqref{simpleVk}, we recall that
$\|\V - \sum_k Q_k\|_{L^2} \lesssim |s|^{-\frac 12}$ and $\|\V - \sum_k Q_k\|_{\dot H^1} \lesssim |s|^{-1}$.
By the triangle inequality, it follows that
\[
\left\| v - \sum_k Q_k \right\|_{L^2} \lesssim |s|^{-\frac 12} \quad\m{and}\quad
\left\| v - \sum_k Q_k \right\|_{\dot H^1} \lesssim |s|^{-\frac 12 - \frac 1{43}}.
\]
Now we let $T_0 = \frac 1{\sqrt{-2S_0}}$ and consider the solution $u$ of~\eqref{kdv} defined on $(0,T_0]$
by $v = \tilde u$ from the change of variables~\eqref{changevar}.
For $0 < t \leq T_0$, set
\[
\lambda_k(t) = t \tilde \mu_k(s),\quad x_k(t) = t y_k(s),\quad s = -\frac 1{2t^2}.
\]
Then, by~\eqref{apriori:mutilde}, \eqref{est:final} and possibly taking a smaller $T_0 > 0$,
\[
|\lambda_k(t) - \ell_k t| \lesssim t^{1 + \frac 2{43}} \leq t^{1 + \frac 1{22}},\quad
|x_k(t) + \ell_k^{-2} t^{-1}| \lesssim t^{-1 + \frac 2{43}} \leq t^{-1 + \frac 1{22}}.
\]
Moreover, from the above estimates on $v$, we obtain
\[
\left\| u(t) - \sum_k \eps_k \lambda_k^{-\frac 12}(t)
Q\left( \frac {\cdot - x_k(t)}{\lambda_k(t)} \right) \right\|_{L^2} \lesssim t,\quad
\left\| u(t) - \sum_k \eps_k \lambda_k^{-\frac 12}(t)
Q\left( \frac {\cdot - x_k(t)}{\lambda_k(t)} \right) \right\|_{\dot H^1}
\lesssim t^{\frac 2{43}},
\]
and thus~\eqref{maintime} by possibly taking a smaller $T_0 > 0$.

Now, we compute the mass and the energy of the solution $u(t)$.
First, by Lemma~\ref{lem:Vmassenergy} and~\eqref{est:final}, for any $s \leq S_0$, we note that
\[
\left| \|\V(s)\|_{L^2}^2 - K\|Q\|_{L^2}^2 \right| \lesssim |s|^{-1 - \frac 1{43}}
\]
and
\[
\left| E(\V(s)) + \frac {\|Q\|_{L^1}^2}{32s} \sum_k \ell_k (1 + 2\theta_k) \right| \lesssim |s|^{-1 - \frac 1{43}}.
\]

By~\eqref{est:final} again, it first follows that
\[
\left| \|v(s)\|_{L^2}^2 - K \|Q\|_{L^2}^2 \right|
\lesssim \|v(s) - \V(s)\|_{L^2} + \left| \|\V(s)\|_{L^2}^2 - K \|Q\|_{L^2}^2 \right|
\lesssim |s|^{-\frac 12 - \frac 1{43}}.
\]
Moreover, from~\eqref{changevar}, we have $\|u(t)\|_{L^2}^2 = \|v(s)\|_{L^2}^2$ for $s = -\frac 1{2t^2}$.
Thus, since the mass of $u(t)$ is constant, passing to the limit $s\to -\infty$ in the last estimate,
we obtain $\|u(t)\|_{L^2}^2 = K \|Q\|_{L^2}^2$.

To estimate similarly $E(v(s)) = E(\V(s) + \e(s))$, we first note that,
by the definition of the energy and~\eqref{est:final},
\[
\left| E(\V + \e) - E(\V) - \int (\pa_y\V \pa_y\e - \V^5\e) \right|
\lesssim \|\e\|_{H^1}^2 \lesssim |s|^{-1 - \frac 2{43}}.
\]
Moreover, by~\eqref{simpleVk} and~\eqref{est:final}, we have
\begin{multline*}
\left| \int (\pa_y\V \pa_y\e - \V^5\e) - \sum_k \int \left( \pa_y Q_k \pa_y\e - Q_k^5 \e \right) \right| \\
\lesssim \left( \left\| \V - \sum_k Q_k \right\|_{\dot H^1} + \left\| \V - \sum_k Q_k \right\|_{L^\infty}
+ \left\| \left( \sum_k Q_k \right)^5 - \sum_k Q_k^5 \right\|_{L^2} \right) \|\e\|_{H^1}
\lesssim |s|^{-1 - \frac 14}.
\end{multline*}
But, by integration by parts and~\eqref{ortho:final},
for all $1\leq k\leq K$, we also have the cancellation
\[
\int \left( \pa_y Q_k \pa_y\e - Q_k^5 \e \right) = - \int \left( \pa_{yy} Q_k + Q_k^5 \right) \e
= - \tilde \mu_k^{-2} \langle Q_k,\e \rangle = 0.
\]
Thus, we have proved $\left| E(\V + \e) - E(\V) \right| \lesssim |s|^{-1 - \frac 2{43}}$ and so
\[
\left| E(v(s)) + \frac {\|Q\|_{L^1}^2}{32s} \sum_k \ell_k (1 + 2\theta_k) \right| \lesssim |s|^{-1 - \frac 1{43}}.
\]
Moreover, from~\eqref{changevar}, we have $E(u(t)) = -2s E(v(s))$ for $s = -\frac 1{2t^2}$.
Thus, since $E(u(t))$ is constant, multiplying the last estimate by 2|s| and passing to the limit $s\to -\infty$,
we obtain
\[
E(u(t)) = \frac {\|Q\|_{L^1}^2}{16} \sum_k \ell_k \left( 1 + 2\theta_k \right).
\]

We finally note that $E(u(t)) > 0$. Indeed, using the definition~\eqref{def:thetak} of $\theta_k$ and rewriting
\begin{align*}
\frac {16}{\|Q\|_{L^1}^2} E(u(t))
&= \sum_{k=1}^K \ell_k^2 \left( \frac 1{\ell_k} + 2 \frac {\eps_k}{\sqrt{\ell_k}}
\sum_{j=1}^{k-1} \frac {\eps_j}{\sqrt{\ell_j}} \right)
= \sum_{k=1}^K \ell_k^2 \left[ \left( \sum_{j=1}^k \frac {\eps_j}{\sqrt{\ell_j}} \right)^2
- \left( \sum_{j=1}^{k-1} \frac {\eps_j}{\sqrt{\ell_j}} \right)^2 \right],
\end{align*}
we obtain, after integration by parts and since $\ell_k > \ell_{k+1}$ for all $1\leq k\leq K-1$,
\[
\frac {16}{\|Q\|_{L^1}^2} E(u(t)) = \sum_{k=1}^{K-1} \left[ (\ell_k^2 - \ell_{k+1}^2)
\left( \sum_{j=1}^k \frac {\eps_j}{\sqrt{\ell_j}} \right)^2 \right]
+ \ell_K^2 \left( \sum_{j=1}^K \frac {\eps_j}{\sqrt{\ell_j}} \right)^2 > 0.
\]
This finishes the proof of Theorem~\ref{thm:main}.
\end{proof}

\begin{remark} \label{rk:final}
We prove that the solution $u(t,x)$ of~\eqref{kdv} constructed in Theorem~\ref{thm:main}
satisfies an estimate as $x\to -\infty$ similar to, but weaker than,
the estimate~\eqref{th:ptwise0} for $S(t,x)$ in Theorem~\ref{thm:mainspaceold}.

Indeed, since $\|v(s) - \V(s)\|_{H^1} \leq |s|^{-\frac 12 - \frac 1{43}}$ by~\eqref{est:final},
we find $\left\| v(s) - \sum_k W_k(s) \right\|_{L^2} \lesssim |s|^{-\frac 12 - \frac 1{43}}$
by the estimates $|r_k(s)| \lesssim |s|^{-1}$ from~\eqref{apriori:rd}, $\|R_k\|_{L^2} \lesssim 1$,
$|a_k(s)| \leq |s|^{-1 - \frac 1{43}}$ from~\eqref{est:final}
and $\|P_k\|_{L^2} \lesssim |s|^{\frac 12}$ from~\eqref{eq:Pknorms}, valid for all $1\leq k\leq K$.
Therefore, from the definition of $W_k$ and the change of variables~\eqref{changevar},
\be \label{sharpL2}
\left\| u(t) - \sum_k \eps_k \tilde \lambda_k^{-\frac 12}(t)
S\left( \rho_k(t),\frac {\cdot - \tilde x_k(t)}{\tilde \lambda_k(t)} \right)
\right\|_{L^2} \lesssim t^{1 + \frac 2{43}}
\ee
with
\[
\rho_k(t) = \frac 1{\sqrt{-2 \tau_k(s)}},\quad \tilde \lambda_k(t) = t \mu_k(s) \sqrt{-2 \tau_k(s)},\quad
\tilde x_k(t) = t z_k(s),\quad s = -\frac 1{2t^2}.
\]
Moreover, from~\eqref{est:final}, we have $|\rho_k(t) - \ell_k^{\frac 32} t| \lesssim t^{1 + \frac 2{43}}$,
$|\tilde \lambda_k(t)- \ell_k^{-\frac 12}| \lesssim t^{\frac 2{43}}$ and $|\tilde x_k(t)| \lesssim t^{-1 + \frac 2{43}}$.

Using~\eqref{th:ptwise0}, a first interesting consequence of~\eqref{sharpL2}
is the following $L^2$ version of~\eqref{th:ptwise0} on the solution $u(t)$:
\[
\left\| u(t) + \left( \frac 12 \|Q\|_{L^1} \sum_k \frac {\eps_k}{\sqrt{\ell_k}} \right) |\cdot|^{-\frac 32}
\right\|_{L^2(x < -2\ell_K^{-2} t^{-1})} \leq t^{1 + \frac 1{22}}.
\]
It means that an explicit tail, sum of the tails of each rescaled version of $S(t)$,
is visible in the asymptotic behavior of $u(t,x)$ for $x\to -\infty$.

Finally, using the estimate $\| |\cdot|^{-\frac 32} \|_{L^2(x < -2\ell_K^{-2} t^{-1})} \gtrsim t$,
the exponential decay of $Q$ and the $L^2$ estimate on $u(t)$ given in the above proof of Theorem~\ref{thm:main},
we obtain the sharp control
\[
t \lesssim \left\| u(t) - \sum_k \eps_k \lambda_k^{-\frac 12}(t)
Q\left( \frac {\cdot - x_k(t)}{\lambda_k(t)} \right) \right\|_{L^2} \lesssim t.
\]
Thus, as a second interesting consequence of~\eqref{sharpL2},
we may notice that, in $L^2$ norm, $u(t)$ is closer to the sum of $K$ modulated versions of $S$
than to the corresponding sum of pure modulated solitons $Q$.
\end{remark}

\appendix

\section{} \label{appendix:corollary}

As a corollary of Theorems~\ref{thm:maintimeold} and~\ref{thm:mainspaceold},
we prove in this appendix the following time estimates in exponential weighted spaces,
which are used to prove~\eqref{eq:Wkweighted}.
Note that we use the notation~\eqref{antideriv} for the antiderivative.

\begin{corollary}[Time asymptotics in weighted spaces] \label{cor:weighted}
With the notation of Theorem~\ref{thm:maintimeold}, there exists $B_0 > 0$ such that,
for all $B\geq B_0$, $M\geq 0$, and all $t\in (0,T_0]$,
\be \label{eq:timeweighted}
\left\| \left[ \pa_x^m S(t) - \sum_{k=0}^M \frac 1{t^{\frac 12 + m - 2k}}
Q_k^{(m-k)} \left( \frac {\cdot + \frac 1t}t + c_0 \right) \right]
e^{\frac {\cdot + \frac 1t}{B t}} \right\|_{L^2} \lesssim t^{2M + 2 - m}.
\ee
\end{corollary}

\begin{remark}
A weaker but useful estimate may be directly deduced from~\eqref{eq:timeweighted}.
Indeed, we observe that, for all $B > 0$, for all $M\geq 0$, for all $t\in (0,T_0]$,
\be \label{eq:timeweightedweak}
\left\| \left[ \pa_x^m S(t) - \sum_{k=0}^M \frac 1{t^{\frac 12 + m - 2k}}
Q_k^{(m-k)} \left( \frac {\cdot + \frac 1t}t + c_0 \right) \right]
e^{-\frac {|\cdot + \frac 1t|}{B t}} \right\|_{L^2} \lesssim t^{2M + 2 - m}.
\ee
Note that this estimate becomes weaker as $B$ gets smaller,
and thus is valid for all $B > 0$ without restriction.
Note also that this estimate could be directly deduced
from~\eqref{th:timem} for $M\leq m - 1$, and thus is only interesting for $M\geq m$.
In particular, we observe that~\eqref{eq:timeweightedweak}
is sharper than~\eqref{th:timem} in the limit case $M = m$, where both estimates are valid.
\end{remark}

To prove Corollary~\ref{cor:weighted}, we first need the following technical lemma.

\begin{lemma} \label{lem:weightright}
Let $p\geq 1$. Assume that $f\in \Cinfini\cap \dot H^p(\R)$ satisfies,
for all $0\leq \ell\leq p$, for all $x > 0$,
\be \label{eq:expdecayright}
|f^{(\ell)}(x)| \lesssim e^{-x}.
\ee
Then, for all $0 < \kappa < 1$, all $0 < \sigma < 1 - \kappa$, and all $x\in\R$,
\be \label{eq:weightLi}
|f(x)| e^{\kappa x} \lesssim \|f^{(p)}\|_{L^2}^\sigma + \|f^{(p)}\|_{L^2}.
\ee
\end{lemma}

\begin{proof}
By the Taylor formula with integral remainder term, we have, for all $x,y\in\R$,
\[
f(x) = f(y) + (x-y) f'(y) + \cdots + \frac {(x-y)^{p-1}}{(p-1)!} f^{(p-1)}(y)
+ \int_y^x \frac {f^{(p)}(z)}{(p-1)!} (x-z)^{p-1} \,dz.
\]
By~\eqref{eq:expdecayright}, passing to the limit as $y\to +\infty$,
we obtain, for all $x\in\R$,
\[
f(x) = - \int_x^{+\infty} \frac {f^{(p)}(z)}{(p-1)!} (x-z)^{p-1} \,dz.
\]
For $x\geq 0$, we deduce, from~\eqref{eq:expdecayright} and H\"older's inequality,
\begin{align*}
|f(x)| e^{\kappa x}
&\lesssim \int_x^{+\infty} z^{p-1} e^{\kappa z} |f^{(p)}(z)| \,dz \\
&\lesssim \int_x^{+\infty} [z^{p-1} e^{\kappa z} e^{-(1 - \sigma) z}] |e^z f^{(p)}(z)|^{1 - \sigma} |f^{(p)}(z)|^\sigma \,dz
\lesssim \|f^{(p)}\|_{L^2}^\sigma.
\end{align*}
For $x < 0$, we obtain similarly
\begin{align*}
|f(x)| e^{\kappa x}
&\lesssim e^{\kappa x} \int_0^{+\infty} (|x|^{p-1} + z^{p-1}) |f^{(p)}(z)| \,dz
+ e^{\kappa x} |x|^{p-1} \int_x^0 |f^{(p)}(z)| \,dz \\
&\lesssim |x|^{p-1} e^{\kappa x} \int_0^{+\infty} e^{-(1 - \sigma) z}
|e^z f^{(p)}(z)|^{1 - \sigma} |f^{(p)}(z)|^\sigma \,dz \\
&\quad + \int_0^{+\infty} [z^{p-1} e^{-(1 - \sigma) z}]
|e^z f^{(p)}(z)|^{1 - \sigma} |f^{(p)}(z)|^\sigma \,dz
+ e^{\kappa x} |x|^{p - 1 + \frac 12} \|f^{(p)}\|_{L^2} \\
&\lesssim \|f^{(p)}\|_{L^2}^\sigma + \|f^{(p)}\|_{L^2},
\end{align*}
which concludes the proof of Lemma~\ref{lem:weightright}.
\end{proof}

\begin{proof}[Proof of Corollary~\ref{cor:weighted}]
Let $B>0$, $m\geq 0$, $M\geq 0$ and $t\in (0,T_0]$.
We prove~\eqref{eq:timeweighted} as a consequence of the time estimate~\eqref{th:timem}
and the space estimate~\eqref{th:ptwiser}, that we rewrite in a rescaled version
so that we may apply Lemma~\ref{lem:weightright}.

For this purpose, we let $L\in\N$ to be fixed later, we denote
\[
\bar\g_0 = \min_{0\leq \ell\leq L} \g_\ell > 0,\quad
\bar\g = \min \left\{ \bar\g_0,\frac 12 \right\},\quad A = \bar\g^{-1} \geq 2,
\]
where $\g_\ell$ is defined in~\eqref{th:ptwiser}, and we consider the change of variables
\[
y = \bar\g \left( \frac {x + \frac 1t}t + c_0 \right) \quad\m{or, equivalently,}\quad
x = A t y - c_0 t - \frac 1t.
\]
First note that the rescaled version of~\eqref{eq:timeweighted} that we want to prove rewrites
$\|g_{A/B}\|_{L^2} \lesssim t^{2M+2}$ where, for $r > 0$, $g_r(y) = g(y)e^{r y}$ and
\[
g(y) = t^{\frac 12 + m} \pa_x^m S \left( t,A t y - c_0 t - \frac 1t \right) - \sum_{k=0}^M t^{2k} Q_k^{(m-k)}(A y).
\]
Next, from~\eqref{th:timem}, we obtain, for all $0\leq \ell\leq L$,
\be \label{timel}
\left\| t^{\frac 12 + \ell} \pa_x^\ell S \left( t,A t \cdot - c_0 t - \frac 1t \right) -
\sum_{k=0}^\ell t^{2k} Q_k^{(\ell-k)}(A \cdot) \right\|_{L^2} \lesssim t^{1 + 2\ell}.
\ee
And, from~\eqref{th:ptwiser}, we get, for all $0\leq \ell\leq L$, for all $y > 0$,
\[
\left| t^{\frac 12 + \ell} \pa_x^\ell S \left( t,A t y - c_0 t - \frac 1t \right) \right|
\lesssim e^{-\g_\ell A y} \lesssim e^{-y}.
\]

Now note that, since $Q_k\in\Y$ for all $k\geq 0$ from the proof of Theorem~\ref{thm:maintimeold} in~\cite{CM1},
we have $|Q_k^{(n)}(y)| \lesssim e^{-y/2}$ for all $n\in\N$ and $y > 0$.
Since moreover $A\geq 2$ and since we use the convention of antiderivative~\eqref{antideriv}, we obtain,
for all $y > 0$ and all $n\in\Z$,
\[
|Q_k^{(n)}(A y)|\lesssim e^{-y}.
\]
Thus, we may apply Lemma~\ref{lem:weightright} with $p = 2M + 2$ and
\[
f(y) = t^{\frac 12 + m} \pa_x^m S \left( t,A t y - c_0 t - \frac 1t \right) - \sum_{k=0}^{2M+1} t^{2k} Q_k^{(m-k)}(A y).
\]
Indeed, letting $L = m + p$, we have $f\in\Cinfini$
and $f$ satisfies~\eqref{eq:expdecayright} from the above space estimates.
Moreover, applying~\eqref{timel} with $\ell = L$, we get $\|f^{(p)}\|_{L^2} \lesssim t^{4M + 4}$.

For all $0 < \kappa \leq \frac 14$, we take $\sigma = 1 - 2\kappa \geq \frac 12$,
so that we have $0 < \sigma < 1 - \kappa$ and we obtain, from~\eqref{eq:weightLi}, for all $y\in\R$,
\[
|f(y)| e^{\kappa y} \lesssim \|f^{(p)}\|_{L^2}^{\sigma} + \|f^{(p)}\|_{L^2}
\lesssim t^{\sigma(4M + 4)} + t^{4M + 4} \lesssim t^{2M + 2}.
\]
In particular, we obtain $\|g_\kappa\|_{L^\infty} \lesssim t^{2M + 2}$ for all $0 < \kappa \leq \frac 14$.
To conclude, we estimate
\begin{align*}
\|g_\kappa\|_{L^2}^2
&= \int g^2(y) e^{2\kappa y} \,dy
= \int_{-\infty}^0 g^2(y) e^{2\kappa y} \,dy + \int_0^{+\infty} g^2(y) e^{2\kappa y} \,dy \\
&\leq \|g_{\kappa/2}\|_{L^\infty}^2 \int_{-\infty}^0 e^{\kappa y} \,dy
+ \|g_{3\kappa/2}\|_{L^\infty}^2 \int_0^{+\infty} e^{-\kappa y} \,dy
\lesssim \|g_{\kappa/2}\|_{L^\infty}^2 + \|g_{3\kappa/2}\|_{L^\infty}^2,
\end{align*}
and so we obtain $\|g_{A/B}\|_{L^2}\lesssim t^{2M + 2}$ as expected,
provided that $\kappa = \frac AB$ satisfies $\frac {3\kappa}2 \leq \frac 14$,
which rewrites $B\geq B_0$ with $B_0 = 6A$.
\end{proof}

\end{document}